\documentclass{amsart}
\usepackage{amsmath, amssymb, amsthm, xcolor, enumitem, comment}

\theoremstyle{definition}
\newtheorem{definition}{Definition}[section]

\newtheorem{example}[definition]{Example}

\newtheorem{question}[definition]{Question}

\newtheorem{notation}[definition]{Notation}
\newtheorem{claim}[definition]{Claim}
\newtheorem{subclaim}[definition]{Subclaim}

\theoremstyle{plain}
\newtheorem{theorem}[definition]{Theorem}
\newtheorem{proposition}[definition]{Proposition}
\newtheorem{lemma}[definition]{Lemma}
\newtheorem{corollary}[definition]{Corollary}

\newtheorem{remark}[definition]{Remark}

\newcommand{\al}{\alpha}
\newcommand{\be}{\beta}
\newcommand{\ga}{\gamma}
\newcommand{\de}{\delta}

\newcommand{\ka}{\kappa}
\newcommand{\lam}{\lambda}
\newcommand{\si}{\sigma}
\newcommand{\Si}{\Sigma}
\newcommand{\vp}{\varphi}
\newcommand{\om}{\omega}

\newcommand{\seq}{\subseteq}
\newcommand{\es}{\emptyset}
\newcommand{\bsl}{\backslash}
\newcommand{\we}{\wedge}
\newcommand{\noi}{\noindent}

\newcommand{\cal}{\mathcal}
\newcommand{\bb}{\mathbb}

\newcommand{\C}{\mathbb{C}}
\newcommand{\ps}{\mathbb{P}}
\newcommand{\po}{\mathbb{P}}
\newcommand{\R}{\mathbb{R}}

\newcommand{\fwc}{\mathcal{F}_{\operatorname{WC}}}
\renewcommand{\S}{\mathbb{S}}
\newcommand{\Q}{\mathbb{Q}}

\newcommand{\RR}[2]{\ps*\dot{\C}_{#1}*\dot{\S}_{#2}}

\newcommand{\col}{\operatorname{Col}}

\newcommand{\tr}{\operatorname{tr}}
\newcommand{\Sk}{\operatorname{Sk}}
\newcommand{\cof}{\operatorname{cof}}
\newcommand{\dom}{\operatorname{dom}}
\newcommand{\ran}{\operatorname{ran}}
\newcommand{\Add}{\operatorname{Add}}
\newcommand{\crit}{\operatorname{crit}}
\newcommand{\cf}{\operatorname{cf}}
\newcommand{\name}{\dot}
\newcommand{\power}{\mathcal{P}}
\newcommand{\uhr}{\upharpoonright}
\newcommand{\res}{\upharpoonright}
\newcommand{\lra}{\longrightarrow}
\newcommand{\elem}{\prec}

\newcommand{\lb}{\left\{}
\newcommand{\rb}{\right\}}
\newcommand{\la}{\langle}
\newcommand{\ra}{\rangle}

\newcommand{\ch}{\mathsf{CH}}
\newcommand{\zfc}{\mathsf{ZFC}}

\DeclareMathOperator{\rlex}{rLex}

\DeclareMathOperator{\cp}{cp}
\DeclareMathOperator{\Ult}{Ult}

\begin{document}

\title{Club Stationary Reflection and the Special Aronszajn Tree Property}\thanks{The first author was partially supported by the Israel 
Science Foundation (Grant 1832/19).}
\author{Omer Ben-Neria and Thomas Gilton}

\address{Hebrew University of Jerusalem
The Edmond J. Safra Campus - Givat Ram, Jerusalem, Israel, 9190401}

\email{omer.bn@mail.huji.ac.il}
\address{University of Pittsburgh
Department of Mathematics.
The Dietrich School of 
Arts and Sciences,
301 Thackeray Hall,
Pittsburgh, PA 15260, United States}
\email{tdg25@pitt.edu}

\date{\today}



\begin{abstract}
We prove that it is consistent that \emph{Club Stationary Reflection} and the \emph{Special Aronszajn Tree Property} simultaneously hold on $\om_2$, thereby contributing to the study of the tension between compactness and incompactness in set theory. The poset which produces the final model follows the collapse of an {ineffable} cardinal first with an iteration of club adding (with anticipation) and second with an iteration specializing Aronszajn trees.

In the first part of the paper, we prove a general theorem about specializing Aronszajn trees on $\om_2$ after forcing with what we call {$\cal{F}$\emph{-Strongly Proper} posets, where $\cal{F}$ is either the weakly compact filter or the filter dual to the ineffability ideal.} This type of poset, of which the Levy collapse is a degenerate example, uses systems of exact residue functions to create many strongly generic conditions. We prove a new result about stationary set preservation by quotients of this kind of poset; as a corollary, we show that the original Laver-Shelah model, which starts from a weakly compact cardinal, satisfies a strong stationary reflection principle, though it fails to satisfy the full Club Stationary Reflection. In the second part, we show that the composition of collapsing and club adding (with anticipation) is an $\cal{F}$-Strongly Proper poset. After proving a new result about Aronszajn tree preservation, we show how to obtain the final model.
\end{abstract}

\maketitle

\section{Introduction}

This work is a contribution to the study of the tension between compactness and incompactness principles in set theory. 
We focus on the second uncountable cardinal, $\omega_2$, and consider the strong compactness principle of Club Stationary Reflection 
and the strong incompactness principle known as the Special Aronszajn Tree Property (these are defined below).

The two properties have been shown to be consistent separately by Magidor \cite{Magidor} and Laver and Shelah \cite{LS}, respectively. Since the properties represent strong forms of opposing phenomena (compactness and incompactness) it is natural to suspect that they are jointly inconsistent.
The main result of this paper shows, on the contrary, that the conjunction of the two principles is consistent. More precisely, we prove: 

\begin{theorem}\label{thm:main}
It is consistent relative to the existence of an {ineffable} cardinal that Club Stationary Reflection and the Special Aronszajn Tree Property simultaneously hold at $\omega_2$. 
\end{theorem}

Our work also shows that a weaker version of stationary reflection holds in the original Laver-Shelah model ({which uses a weakly compact)}:

\begin{theorem}\label{theorem:OriginalLS}
In the original Laver-Shelah model the following stationary reflection principle holds:
for every sequence $\la S_\al \mid \al < \om_2\ra$ of stationary subsets of $\om_2 \cap \cof(\om)$ there is $\be < \om_2$ so that 
$S_\al \cap \be$ is stationary in $\be$, for every $\al < \be$. However, Club Stationary Reflection at $\omega_2$ fails.
\end{theorem}

We proceed to define the relevant terms and contextualize our result. If $\nu$ is a regular cardinal, we use $\cof(\nu)$ denote the class of ordinals with cofinality $\nu$. We recall that if $\cf(\al)>\om$, then $S\seq\al$ is \emph{stationary} if $S\cap C\neq\es$ for each club $C\seq\al$.  We say that $S$ \emph{reflects} if there is some $\be<\al$ with $\cf(\be)>\om$ so that $S\cap\be$ is stationary in $\be$. If $\ka$ is regular, we say that \emph{stationary reflection} holds at $\ka^{++}$ if every stationary $S\seq\ka^{++}\cap\cof(\leq\ka)$ reflects. Baumgartner originally showed (\cite{Baumgartner}) that stationary reflection at $\om_2$ is consistent from a weakly compact cardinal. Harrington and Shelah (\cite{HS}) later improved this, showing that the optimal assumption of a Mahlo cardinal suffices. One obtains stronger principles by requiring that multiple stationary sets reflect simultaneously. Recall that a collection $\{ S_i \mid i < \tau\}$ of $\tau < \al$ stationary subsets of $\al$ is said to reflect \emph{simultaneously} if there is some $\be<\al$ with $\cf(\be)>\om$ so that $S_i\cap\be$ is stationary in $\be$ for every $i < \tau$. 
Magidor (\cite{Magidor}) has shown that the consistency strength of ``any two stationary subsets of $\om_2\cap\cof(\om)$ simultaneous reflect" implies the consistency of a weakly compact cardinal. One may also consider stronger diagonal versions of the above, defined in the natural way.

We are interested in the following very strong form of stationary reflection which implies all of the above: 

\begin{definition} Suppose that $\ka$ is regular. We say that \textbf{Club Stationary Reflection} holds at $\ka^{++}$ if for any stationary $S\seq\ka^{++}\cap\cof(\leq\ka)$, there exists a club $C\seq\ka^{++}$ so that for all $\be\in C\cap\cof(\ka^+)$, $S$ reflects at $\be$. We write $\mathsf{CSR}(\ka^{++})$.
\end{definition}

We will concern ourselves with the case $\ka=\om$, {i.e., with stationary subsets of $\om_2\cap\cof(\om)$.} Most relevant for us, Magidor (\cite{Magidor}) showed that $\mathsf{CSR}(\om_2)$ is consistent from a weakly compact cardinal; by the above remarks, this is the optimal hypotheses.

Extensions of $\mathsf{CSR}$ to other cardinals have been shown to have limitations. 
For example, Jech and Shelah \cite{JechShelah} proved that for every $n < \omega$, if every stationary subset of $ \omega_{n+3} \cap \cof(\omega_{n+1})$ reflects then $\mathsf{CSR}(\om_{n+2})$ fails.
{However, Jech and Shelah \cite{JechShelah}, and later Cummings and Wylie \cite{CummingsWylie}, proved the consistency of certain best possible variations of club stationary reflection below $\aleph_\om$.}

 Limitations on stationary reflection emerge from incompactness principles. One of the most prominent of these is Jensen's $\square_\kappa$. In {\cite{Jensen}}, Jensen showed that $\square_{\ka}$ holds in $L$ for all $\ka>\om$ and that $\Box_\ka$ implies the existence of many nonreflecting stationary subsets of $\ka^+$.

Further studies showed that variations of $\square_\ka$ place limitations on the cofinality of reflection points, as well as the amount of simultaneous reflection.  For instance, in \cite{SchimmerlingSquares}, Schimmerling introduced the hierarchy of square principles, 
$\square_{\kappa,\lambda}$, $1 \leq \lambda \leq \ka^+$. As $\lam$ increases, this hierarchy is strictly decreasing in strength; see Jensen \cite{Jensen:BelowZeroPistol} for $\ka$ regular and \cite{CFM} for $\ka$ singular.
For a regular cardinal $\ka$ and $\lambda\leq\kappa$, Schimmerling and independently Foreman and Magidor have observed that 
if $\ka^{<\lambda}= \ka$ and $\square_{\kappa,<\lambda}$ holds then every stationary subset of $\ka^+$ has a stationary subset which does not reflect at any point of cofinality $\geq \lambda$; see \cite{CummingsSchimmerling}.
In particular, $\kappa^{<\ka} =\ka$ and $\square_{\ka,<\ka}$ imply that every stationary subset of $\ka^+$ has a stationary subset which does not reflect at any point in $\ka^+ \cap \cof(\ka)$. In \cite{CFM}, Cummings, Foreman, and Magidor extended these results and developed the theory for $\ka$ singular. 

Other notable weakenings of $\square_\ka$ were introduced and developed by Todor{\v c}evi{\'c} (\cite{TodorcevicSquares}). These principles, denoted $\square(\ka^+)$ and $ \square(\ka^+,\lambda)$, place refined limitations on the extent of stationary reflection. See \cite{HayutFontanella}, \cite{HayutLambie}, and \cite{RinotBSL}.

The weakest nontrivial form of square studied by Jensen is the so-called Weak Square, denoted $\square_{\ka}^*$, Remarkably, $\square_\ka^*$ is equivalent to a key incompactness phenomenon, the existence of a {special $\ka^+$-Aronszajn tree.} Let us recall the relevant definitions. A \emph{tree} is a partially ordered set $(T,\leq_T)$ so that for each $x\in T$, the set of $\leq_T$-predecessors of $x$ is well-ordered; we refer to the {\emph{height}} of $x$ in $T$ as the ordertype of this set. If $\al$ is an ordinal, we use Lev$_\al(T)$ to denote all $x\in T$ of height $\al$. The \emph{height} of $T$ is the least ordinal $\al$ so that $T$ has no elements of height $\al$. A \emph{branch} through $T$ is a linearly ordered subset of $T$, and a \emph{cofinal branch} is a branch which intersects every level below the height of $T$.

Let $\ka$ be regular. A \emph{$\ka$-tree} is a tree $T$ of height $\ka$ so that each level has size $<\ka$; we will always assume that for each such tree, each node in the tree has incompatible extensions to all higher levels. $\ka$ is said to have the \emph{tree property} if every $\ka$-tree has a cofinal branch. K$\ddot{\text{o}}$nig showed (\cite{Konig}) that $\om$ has the tree property, while Aronszjan has shown that the tree property fails at $\om_1$ (the result was communicated in \cite{Kurepa}). The extent of the tree property on cardinals greater than $\om_1$, a famous question of Magidor's, is independent of $\zfc$. A watershed in our understanding is due to Mitchell and Silver (\cite{MitchellUgh}) who showed that the tree property at $\om_2$ is consistent from a weakly compact cardinal.

A tree which witnesses the failure of the tree property is said to be \emph{Aronszajn} (i.e., a $\ka$-tree which has no cofinal branches); the existence of such a tree is an instance of incompactness.  A particularly strong witness that a tree is Aronszajn is given by a \emph{specializing function}: in the case that $\ka=\lam^+$, a specializing function is an $f:T\lra\lam$ so that if $x<_Ty$, then $f(x)\neq f(y)$. (For an exploration of these concepts at an arbitrarily regular cardinal, see \cite{KruegerWeakSquare}.) Having a specializing function is a particularly strong witness to being Aronszajn since the function witnesses that $T$ remains Aronszajn in {any extension of that model in which $\ka$ is still a cardinal.} $T$ is said to be a \emph{special Aronszajn tree} if there is a specializing function for $T$. The property of interest to us is the following:

\begin{definition} Let $\ka$ be regular. We say that $\ka^+$ has the \textbf{Special Aronszajn Tree Property} if there are Aronszajn trees on $\ka^+$, and if every Aronszajn tree on $\ka^+$ is special. We denote this property by $\mathsf{SATP}(\ka^+)$.
\end{definition}

By a result of Specker (\cite{Specker}), if $\ka^{<\ka}$ holds, then there is a special Aronszajn tree on $\ka^+$; in particular, if the $\ch$ holds, then there is a special Aronszajn tree on $\om_2$. Jensen (\cite{Jensen}) later showed that the principle $\Box^*_\ka$ holds if and only if there is a special Aronszajn tree on $\ka^+$; since $\ka^{<\ka}$ implies $\Box^*_\ka$, this strengthens Specker's result. 

With regards to constructing specializing functions by forcing, Baumgartner, Malitz, and Reinhardt showed (\cite{BMR}) that $\mathsf{MA}+\neg\ch$ implies $\mathsf{SATP}(\om_1)$. Later, Laver and Shelah showed (\cite{LS}) that $\mathsf{SATP}(\om_2)$ is consistent from a weakly compact cardinal. Generalizing this further, Golshani and Hayut have recently shown (\cite{GH}), using posets which specialize with anticipation, that it is consistent that, simultaneously, for every regular cardinal $\ka$, $\mathsf{SATP}(\ka^+)$ holds. Krueger has generalized the result of Laver-Shelah (and also Abraham-Shelah, \cite{AbrahamShelah}) in a different direction (\cite{Krueger}), showing that it is consistent with the $\ch$ that any two countably closed Aronszjan trees on $\om_2$ are club isomorphic. And finally, Asper\'{o} and Golshani (\cite{AG}) have announced a positive solution to the question of whether $\mathsf{SATP}(\om_2)$ is consistent with the  $\mathsf{GCH}$.
This work continues the study of the tension between different manifestations of compactness and incompactness phenomena in set theory, which together with the study of tension with other fundamental principles such as approximation principles (e.g., \cite{BNdiamond}, \cite{GoldbergDiamond}) and cardinal arithmetic (e.g., \cite{ForemanTodor}, \cite{ShelahReflcSCH}) is central to our understanding of their extent and limitations.
\vspace{4mm}

We proceed to describe our result in general terms and highlight the challenges that appear in the process. 
Let $\ka$ be a cardinal {which is either ineffable or weakly compact} in a ground model $V$ of $\mathsf{GCH}$; {\textbf{we will specify later (Definition \ref{def:fksp}) exactly when $\ka$ is ineffable or weakly compact.}} We obtain the model which witnesses Theorem \ref{thm:main} by first defining, in the extension by $\ps = \col(\omega_1,<\ka)$, a $\ka^+$-length iteration $\C_{\ka^+} = \la \C_\tau, \C(\tau) \mid \tau < \ka^+\ra$ of adding clubs which will eventually witnesses $\mathsf{CSR}(\om_2)$.

After forcing with $\C_{\ka^+}$, we then force with a $\ka^+$-iteration $\S_{\ka^+} = \la \S_\tau,\S(\tau) \mid \tau < \ka^+\ra$ specializing the desired Aronszajn trees $\name{T}_\tau$, i.e., $\S(\tau) = \S(\name{T_\tau})$ (see Section \ref{sec:prelim} for precise definitions of the posets).

To make this strategy work, we need, among other things, that  all stationary subsets of $\om_2\cap\cof(\om)$ which appear in the final generic extension by  $\ps * \dot{\C}_{\ka^+} *\dot{\S}_{\ka^+}$ reflect as in the definition of $\mathsf{CSR}(\om_2)$. Consequently, the club adding posets must anticipate names for stationary sets added by the later specializing iteration. In order to carry this through, we define the names $\dot\C_\tau$ and $\dot\S_\tau$, for $\tau < \ka^+$, simultaneously. More precisely, for each $\tau < \ka^+$, given that the $\ps$-name $\dot\C_\tau$ and the $(\ps\ast\dot\C_\tau)$-name $\dot\S_\tau$ have been defined, we use a bookkeeping function to pick the $(\ps*\dot\C_{\tau}*\dot\S_{\tau})$-name $\name{S_{\tau}}$ of a stationary subset of $\kappa \cap \cof(\omega)$, and we set $\dot\C(\tau)$ to be the $(\ps\ast\dot\C_\tau)$-name for the poset to add, with $\dot\S_\tau$-anticipation, the desired club. Then we select the $(\ps*\dot\C_{\tau+1}*\dot\S_\tau)$-name
$\name{T}_{\tau}$ for an Aronszajn tree on $\kappa$.

As expected, the tension between compactness and incompactness gives rise to tension between the different parts of the forcing construction.  We list three notable manifestations:

    \textbf{(1) Working with Intermediate Generic Extensions.} A central property of the Laver-Shelah forcing (\cite{LS}) is the existence of intermediate forcing extensions in which regular cardinals $\al < \ka$ become $\om_2$ and the relevant portion $\name{T}_\tau \cap (\alpha \times \omega_1)$ of the trees are Aronszajn trees on $\al$. {Accompanying this is machinery for projecting conditions} of $\ps * \dot\S_\tau$ to those intermediate extensions. In \cite{LS} the existence of such intermediate extensions is secured by the weak compactness of $\ka$, and the fact that $\ps * \dot\S_\tau$ is $\ka$-c.c. However, in our case, the presence of the poset $\dot\C_\tau$ prevents the initial segment $\ps *\dot\C_\tau$ from being $\ka$-c.c. To overcome this difficulty,  we use the fact that the full collapse poset $\ps$ absorbs many restricted subforcings of $\ps * \dot\C_\tau$, which allows us to place upper bounds on various generic filters of the restricted poset. We then couple this in Section \ref{section:clubscompletelyproper} with a generalization of a result of Abraham's (\cite{Abraham}) that (stated in current language) if $\dot{\Q}$ is an $\Add(\om,\om_1)$-name for an $\om_1$-closed poset, then $\Add(\om,\om_1)\ast\dot{\Q}$ is strongly proper. This secures the existence of sufficiently many strongly generic conditions (and in turn, the existence of intermediate extensions).

    \textbf{(2) Preservation of Stationary Sets by Quotients.} The ability to add a closed unbounded set through the reflection points $\al < \ka$ of a stationary set $\name{S_\tau} \subseteq \ka \cap \cof(\omega)$ hinges upon the fact that many such points exist. {The ineffability (in fact, just weak compactness)} of $\ka$ guarantees that for many $\al < \ka$,  $\name{S_\tau} \cap \alpha$ is a stationary subset of $\al$ in the restricted generic extension where $\al = \om_2$. The forcing construction of \cite{Magidor} uses the fact that the related quotient of $\ps * \dot\C_\tau$ by its initial segment is $\sigma$-closed, and an argument of Baumgartner's ({\cite{Baumgartner}}) shows that $\sigma$-closed posets preserve the stationarity of stationary sets of countable cofinality ordinals. By contrast, for us the stationary sets $\name{S}_\tau$ further rely on the specializing poset $\dot\S_\tau$, and although the poset $\ps * \dot\C_\tau*\dot\S_\tau$ is, $\sigma$-closed, it does not in general admit  $\sigma$-closed quotients by its natural restrictions to heights $\al < \ka$. Nevertheless, in Section \ref{section:stationarypreservation} we analyze the Laver-Shelah iteration $\dot\S_\tau$ to prove that the relevant quotients preserve the stationary of $\name{S_\tau} \cap \al$ for many suitable $\al < \ka$. 
    
    \textbf{(3) Preservation of Aronszajn Trees.}
    The organization of the posets $\dot\C_\tau$ and $\dot\S_{\tau}$, described above, guarantees that for each $\tau < \ka^+$, $\name{T}_\tau$ is 
    a $(\ps * \dot\C_{\tau+1}*\dot\S_\tau)$-name of an Aronszajn tree on $\ka$, which is specialized by $\ps * \dot\C_{\tau+1} * \dot\S_{\tau+1}$. 
  
    However, in the final forcing construction, $\S_{\tau+1}$ follows the extended iteration $\ps * \dot\C_{\ka^+}$, and on its face, $\ps * \dot\C_{\ka^+} * \dot\S_{\tau}$ might introduce a cofinal branch to $T_\tau$, causing the specializing poset $\S(\tau)$ to collapse $\ka$.
    To guarantee that this cannot occur, an Aronszajn preservation theorem is required for the quotient of $\ps * \dot\C_{\ka^+} * \dot\S_\tau $ by 
    $\ps * \dot\C_{\tau+1} * \dot\S_\tau$. The fact that no new reals are added during the iteration, and that $\dot\C_{\ka^+}$ is not $\ka$-closed, prevents us from using known preservation arguments (for instance those of \cite{Unger}). Therefore, in Section \ref{section:clubscompletelyproper} we develop an alternative preservation argument which fits the properties of the poset $\ps * \dot\C_{\ka^+}$, and we apply them in Section \ref{section:final} to show that the tree $T_\tau$ remains Aronszajn.

    \textbf{Structure of this work:} 
    In the rest of this section, we review relevant preliminaries regarding forcing {as well as ineffable and weakly compact cardinals. }The first part of the work consists of Sections \ref{sec:FWCstronglyProper} through \ref{section:stationarypreservation}.
    In Section \ref{sec:FWCstronglyProper} we develop the notion and fundamental properties of posets which are strongly proper with respect to {the filter $\cal{F}$ on $\ka$, which is either the weakly compact filter on $\ka$ or the filter dual to the ineffability ideal on $\ka$ (depending on whether $\ka$ is weakly compact or ineffable, respectively).} We will later verify that initial segments of the form $\ps * \dot\C_\tau$, for $\tau < \ka^+$, are members this class. Section \ref{Section:specialize} studies an iteration of specializing posets $\S_\tau$, following an $\cal{F}$-strongly proper poset $\ps^*$. We prove that the main results of the Laver-Shelah analysis apply in this context as well. {In the case when $\ps^*$ is just the Levy collapse to make $\ka$ become $\om_2$, we only need a weakly compact (and $\cal{F}$ is the weakly compact filter). This is just the Laver-Shelah argument. However, when $\ps^*$ becomes a more complicated poset, we needed to use the stronger assumption that $\ka$ is ineffable (and $\cal{F}$ is the filter dual to the ineffability ideal). Nevertheless, we only need the ineffability for the case when $\ps^*$ is not just the collapse, and in this case, only for the proof of Proposition \ref{prop:InductiveIandII} and the corollaries of that proposition.} Section \ref{section:stationarypreservation} is devoted to showing that suitable quotients of specializing iterations of the form $\ps^*  * \dot\S_\tau$, where $\ps^*$ is {$\cal{F}$-strongly proper (and in either case for $\cal{F}$)} preserve stationary subsets of countable cofinality ordinals.

    In Part 2 of the paper, we construct specific posets playing the role of $\ps^*$ above, and we prove our theorem.
    In Section \ref{sec:FWCcompletelyproper}, we introduce the complementary notion of posets which are completely proper with respect to  $\cal{F}$, and later we apply this analysis to $\dot\C_\tau$, $\tau < \ka^+$. We show that the composition of the Levy collapse and a poset which is completely proper with respect to $\cal{F}$ is strongly proper with respect to $\cal{F}$. Section \ref{section:clubscompletelyproper} develops the main properties of the club adding iteration $\C_\tau$. 
    And finally, we combine the results of the previous sections in Section \ref{section:final} to prove Theorem \ref{thm:main}.

\subsection{Forcing}\label{sec:prelim}

In this subsection, we review our conventions about forcing and provide explicit definitions of posets which we will use throughout the paper.

To begin, in order to anticipate working with iterations later, we will work with \textbf{pre-orderings} (i.e., relations which are transitive and reflexive) rather than partial orders. Moreover, we will use the \textbf{Jerusalem convention for forcing}. Thus we view a forcing poset as a triple $(\Q,\leq_\Q,0_\Q)$, where $\leq_\Q$ is a pre-ordering and where $0_\Q$ is a smallest element; for conditions $p,q\in\Q$, we will write $p\geq_\Q q$ to indicate that $p$ is an extension of $q$. When context is clear, we will drop explicit mention of $\Q$ in $0_\Q$ and $\leq_\Q$. Given that we are only working with pre-orderings rather than partial orders, we will often have conditions $p,q\in\Q$ so that $p\leq_\Q q$ and $q\leq_\Q p$ but $q$ and $p$ are not literally equal as sets. In this case, we will write $p=^*_\Q q$,
or simply $p=^*q$ if $\Q$ is clear from context.

If $\Q$ is a poset, we say that $\Q$ is \textbf{$\om_1$-closed with sups} if for any increasing sequence $\la q_n:n\in\om\ra$ of conditions in $\Q$, there exists a $\leq_\Q$-least upper bound $q$ of the sequence. Any such $q$ is referred to as \textbf{a sup} of the sequence. Note that this does \textbf{not} say that any two compatible conditions in $\Q$ have a sup. Moreover, it also does not require that a sup of an increasing $\om$-sequence is unique. However, if $q_1$ and $q_2$ are two sups of such a sequence, then $q_1=^*_\Q q_2$. These observations will be important later when we deal with iterations of posets with this property.

If $\Q$ is a poset and $q\in\Q$, then we use $\Q/q$ to denote all conditions in $\Q$ which extend $q$. 
\vspace{3mm}

Let $M \elem H(\theta)$ be an elementary substructure and $\bb{U} \in M$ a poset. 
A condition $u$ in  $\mathbb{U}$ is $(M,\bb{U})$-\textbf{completely generic}, if the set $\lb \bar{u}\in\bb{U}\cap M:u\geq\bar{u}\rb$ of weaker conditions in $M$ meets all dense subsets $D \subseteq \bb{U}$ which belong to $M$, and thus forms a $(M,\bb{U})$-generic filter.

For the remainder of the paper, we fix a {\textbf{cardinal $\ka$ which will be either weakly compact or ineffable} (we specify in Definition \ref{def:fksp} exactly when $\ka$ is ineffable or weakly compact)}. In the next subsection, we will review facts about ineffability and weak compactness.

Throughout the paper, we will use $\ps$ to denote the Levy collapse $\col(\om_1,<\ka)$. If $\al<\ka$ is inaccessible, we use $\ps\res\al$ to denote the collapse $\col(\om_1,<\al)$. We view conditions in $\col(\om_1,<\ka)$ as countable functions $p$ so that $\dom(p)\seq\ka$ and so that for each $\nu\in\dom(p)$, $p(\nu)$ is a countable, partial function from $\om_1$ to $\nu$. 
If $G$ is a $V$-generic filter over $\ps$, then we use $G\res\al$ to denote the $V$-generic filter $\lb p\res\al:p\in G\rb$ over $\ps\res\al$.

For adding clubs, we generalize the club-adding poset of Magidor (\cite{Magidor}) by incorporating anticipation; we only state the definition in the generality needed for our paper. Recall that if $S$ is stationary in $\al$, then the \emph{trace} of $S$, denoted $\tr(S)$, consists of all $\be<\al$ so that $S$ reflects at $\be$.

\begin{definition}\label{def:clubaddingposet} Let $\bb{S}$ be a cardinal-preserving poset in some model $W$, and let $\dot{S}$ be an $\bb{S}$-name for a stationary subset of $\om_2\cap\cof(\om)$. We let $\mathsf{CU}(\dot{S},\bb{S})$ denote the poset, defined in $W$, where conditions are closed, bounded subsets $c$ of $\om_2$ so that 
$$
\Vdash_{\bb{S}}\check{c}\seq\tr(\dot{S})\cup\left(\om_2\cap\cof(\om)\right).
$$
The ordering is end-extension.
\end{definition}
We emphasize that in order to be a condition in $\mathsf{CU}(\dot{S},\bb{S})$, a given closed, bounded subset of $\om_2$ must be outright forced by $\bb{S}$ to be contained, mod cofinality $\om$ points, in $\tr(\dot{S})$. Since any condition $c$ in $\mathsf{CU}(\dot{S},\bb{S})$ can be extended by placing an ordinal of cofinality $\om$ above $\max(c)$, we see that $\mathsf{CU}(\dot{S},\bb{S})$ does add a club subset of $\om_2$ of the model. Moreover, the poset is trivially $\om_1$-closed, so preserves $\om_1$. However preservation of $\om_2$ is a non-trivial matter.

We now review the definition of the poset which we will use to specialize Aronszajn trees on $\om_2$. The poset itself will decompose such a tree into a union of $\om_1$-many antichains, which in this case is equivalent to having a specializing function.

\begin{definition}\label{def:specposet} Suppose that $T$ is an Aronszajn tree on $\om_2$. Let $\bb{S}(T)$ denote the poset where conditions are functions $f$ with countable domain $\dom(f)\seq\om_1$, and where for each $\al\in\dom(f)$, $f(\al)\seq T$ is a countable antichain in $<_T$. Recalling that we are using the Jerusalem convention for {forcing}, we say that $g$ extends $f$, written $f\leq g$, if $\dom(f)\seq\dom(g)$ and if for all $\al\in\dom(f)$, $f(\al)\seq g(\al)$.
\end{definition}

It is clear that $\bb{S}(T)$ is $\om_1$-closed. Moreover, if a tree $T'$ is not Aronszajn, then the analogously defined poset $\bb{S}(T')$ will collapse $\om_2$.

\subsection{Weak Compactness}

{In this final subsection, we review facts about the ineffability of $\ka$. However, we will also need various facts about the weak compactness of $\ka$, and so we begin with these.}

\begin{definition}
$\fwc$ is the filter generated by 
subsets $A$ of $\ka$  for which there is some $U \subseteq V_\ka$ and a $\Pi^1_1$-statement $\Phi$, satisfied by $(V_\ka,\in,U)$, so that 
\[A = \{ \al <\ka \mid \al \text{ is regular, and } (V_\al,\in,U \cap V_\al) \models \Phi\}.\]
\end{definition}

The filter $\fwc$ is $\ka$-closed as well as normal. It will be helpful at later parts in our argument to phrase membership in the weakly compact filter in terms of embeddings. The idea will be that for a subset $B$ of $\ka$, where $B$ is a member of a $\ka$-model $M$,  $B\in\fwc$ iff for all $M$-normal ultrafilters $U$, $\ka\in j_U(B)$, where $j_U$ is the ultrapower embedding. We make this precise in the following few items.

\begin{definition} Suppose that $\al$ is an inaccessible cardinal. We say that a transitive set $M$ is an \textbf{$\al$-model} if $M\models\zfc^-$, $|M|=\al$, $\al\in M$, and $\,^{<\al}M\seq M$.
\end{definition}

{Weak compactness is naturally associated to various embedding properties; here we mention the following result from \cite{Hauser}:}

\begin{proposition}\label{prop:jinN} For any $\ka$-model $M$, there exist a $\ka$-model $N$ and an elementary embedding $j:M\lra N$ so that $\crit(j)=\ka$ and $j,M\in N$.
\end{proposition}

{However, we are mostly interested a different case, namely, when $j$ is the elementary embeddings associated with an $M$-normal ultrafilter on $\ka$ and where $N$ is the ultrapower of $M$. A filter $U \subseteq \power(\kappa) \cap M$ is an \textbf{$M$-normal ultrafilter} if
$U$ is an $M$-ultrafilter, and for every $A \in U$ and regressive function $f : A \to \kappa$ in $M$ there exists some $A' \subseteq A$ in $U$ so that $f\uhr A'$ is constant.} We note that being $M$-normal implies that $U$ is closed under intersections of $<\kappa$-sequences in $M$ consisting of sets in $U$. 

It is routine to verify that each elementary embedding $j : M \to N$ as in the proposition above gives rise to an $M$-normal ultrafilter 
$U_j = \{ A \in \power(\kappa) \cap M \mid \kappa \in j(A)\}$.
Conversely, we can associate to each $M$-normal ultrafilter $U$ its ultrapower embedding
$j_U : M \to N \cong \Ult(M,U)$.

\begin{proposition}\label{proposition:characterizefwc}
Let $M$ be a $\kappa$-model and $B \in M$ a subset of $\kappa$.
If $B \in U$ for every $M$-normal ultrafilter $U$ on $\kappa$, then 
$B\in\fwc$.
\end{proposition}
\begin{proof} 
Let $M,B$ be as in the statement of the proposition, {and fix a subset $E_M \subseteq \kappa \times \kappa$ so that $(\kappa,E_M)$ is isomorphic to $M$.} To prove that $B \in \fwc$, it suffices to show that there is a
 $\Pi^1_1$ statement $\Psi$ satisfied by $(V_\kappa,\in, E_M,B)$ so that the set
 $\{ \alpha<\ka : \al \text{ is regular and } (V_\al,\in,E_M \cap V_\al , B \cap \alpha) \models \Psi\}$
 is contained in $B$.
 
 To begin, we observe that the assertions that ``$E_M$ is well-founded'', that ``$(\kappa,E_M)$ is isomorphic to a transitive $\kappa$-model'', and that ``$B$ is represented in $(\kappa,E_M)$ by $b \in \kappa$",  are all within the class of $\Pi^1_1$ formulas  over $(V_\kappa,\in,E_M,B)$ (see \cite{KruegerWC}, Section 2 for details). Let $\Phi_0$ denote their conjunction.

We also note that for a subset $U_M \subseteq \kappa$, the assertion 
``$U_M$ codes a subset of $M \cap \power(\kappa)$ which is an $M$-normal ultrafilter," is $\Sigma^0_\omega$.
Therefore the assertion $\Phi_1$ stating that 
\[``\forall\, U_M \subseteq \kappa \text{, if } U_M \text{ codes an } M\text{-normal ultrafilter, then }B \in U_M "\]
is $\Pi^1_1$.
Let $\Phi$ be the conjunction $\Phi_0\we\Phi_1$, a $\Pi^1_1$ formula satisfied in $(V_\kappa,\in,E_M,B)$.

Define $X:=\{ \alpha<\ka : \al \text{ is regular, and } (V_\al,\in,E_M \cap V_\al , B \cap \alpha) \models \Phi\}$, and we show that $X\seq B$. Fix some $\al\in X$. Then the relation $E_M \uhr \al = E_M \cap V_\al \subseteq \al\times \al$ is well-founded, and it codes an $\al$-model $M_\al$ with $B \cap \al$ represented in $(\al,E_M\uhr \al)$ by the same element $b \in \ka$ which represents $B$ in $(\ka,E_M)$.
 Let $i_\alpha : M_\al \to M_\ka$ be the elementary embedding resulting from the identifications $M_\al \cong (\al,E_M\uhr \al) \elem (\ka,E_M) \cong M$. It is straightforward to verify that $\cp(i_\al) = \al$, $i_\al(\al) =\ka$, and $i_\al(B \cap \al) = B$. 

It follows that $U_{\al} = \{ A \subseteq \al : \al \in i_\al(A)\}$ is an $M_\al$-normal ultrafilter. 
Let 
$j_\alpha : M_\alpha \to N_\alpha$ be the induced ultrapower embedding, and let $k_\alpha : N_\alpha \to M$ be the factor map given by $k_\al([f]_{U_\al}) = i_\al(f)(\al)$. We note that ${\cp}(k_\alpha) > \alpha$. Since $(V_\al,\in,E_M\uhr\al,B \cap \al)$ satisfies $\Phi$,
$B \cap \al \in U_\al$. The last implies that 
$\al \in j_\al(B \cap \al)$, which in turn implies that 
$\al = k_\al(\al) \in k_\al \circ j_\alpha(B \cap \alpha) = i_\al(B \cap \al) = B$.
\end{proof}

{
We now review the relevant facts about ineffable cardinals; see \cite{BaumgartnerIneffability} for the details. A cardinal $\lam$ is \textbf{ineffable} if for any sequence $\vec{A}=\la A_\nu:\nu<\lam\ra$ so that $A_\nu\seq\nu$ for all $\nu<\lam$, there is an $A\seq\lam$ so that $\lb\nu<\lam:A_\nu=A\cap\nu\rb$ is stationary in $\lam$. Such a set $A$ is said to be coherent for $\vec{A}$.} 

\begin{notation}\label{notation:FI} {In the case that $\ka$ is ineffable, we will denote the \textbf{ineffability ideal} on $\ka$ by $\cal{I}_{in}$ throughout the paper, and we will let $\cal{F}_{in}$ denote the filter on $\ka$ which is dual to $\cal{I}$.}
\end{notation}
{$\cal{I}_{in}$ consists of all $S\seq\ka$ so that for some sequence $\vec{A}$ as above, no stationary subset of $S$ is coherent for $\vec{A}$. In the case that $\ka$ is ineffable, $\cal{I}_{in}$ is a proper, normal ideal on $\ka$. Finally, we mention the following theorem of Baumgartner's (\cite{BaumgartnerIneffability}, Theorem 7.2); see \cite{HolyLuecke:SmallLargeInduced} for more information.}

\begin{theorem}\label{thm:Baumgartner} {(Baumgartner) $F_{WC}$ is contained in $\cal{F}_{in}$.}
\end{theorem}

{In fact, Baumgartner showed that $\cal{F}_{WC}$ is contained in the filter dual to the \emph{weak} ineffability ideal on $\ka$.}

\begin{notation}\label{WhatisF} {$\cal{F}$ will denote either $\cal{F}_{WC}$ or $\cal{F}_{in}$ throughout this paper depending on whether $\ka$ is weakly compact or ineffable (respectively). $\cal{I}$ denotes the ideal dual to $\cal{F}$. We will specify in Definition \ref{def:fksp} in the next section exactly when $\cal{F}$ is equal to $\cal{F}_{WC}$ or equal to $\cal{F}_{in}$.}
\end{notation}

\begin{remark}\label{rmk:iewc}\hfill
\begin{enumerate}
\item {In our paper, we only use the ineffability of $\ka$ to prove Proposition \ref{prop:InductiveIandII} and in the corollaries of this proposition. All other results in our paper can be carried out assuming only that $\ka$ is weakly compact.}
\item {As is immediate from Theorem \ref{thm:Baumgartner} and the fact that $\cal{F}$ is either $\cal{F}_{WC}$ or $\cal{F}_{in}$, all $\cal{F}$-positive sets are also $F_{WC}$-positive.}
\item {Thus, if $B\in \cal{F}^+$ and if $M$ is a $\ka$-model with $B\in M$, then by Proposition \ref{proposition:characterizefwc} there is some $M$-normal ultrafilter $U$ on $\ka$ so that $B\in U$. We will use this fact throughout the paper.}
\end{enumerate}
\end{remark}

{As an illustration of the type of argument in (3) above, we prove the following fact which we will need in the proof of Proposition \ref{prop:statsetpreservation}.}

\begin{lemma}\label{lemma:SetAndTrace} 
{Suppose that $B\in\cal{F}^+$. Then $B\bsl\operatorname{tr}(B)\in\cal{I}$.}
\end{lemma}
\begin{proof} {Suppose otherwise, for a contradiction. Then $B\bsl\tr(B)\in\cal{F}^+$. Let $M^*$ be a $\ka$-model containing $B$, and hence $B\bsl\tr(B)$. By Proposition \ref{proposition:characterizefwc}, there is an $M^*$-normal ultrafilter $U$ so that, letting $j:M^*\lra N$ be the ultrapower embedding, $\ka\in j(B\bsl\tr(B))$. However, $B$ is a stationary subset of $\ka$, since $B\in\cal{F}^+$. Thus $\ka\in j(\tr(B))$. Since $\ka\in j(B)$ also, we have $\ka\in j(B)\cap j(\tr(B))=j(B\cap\tr(B))$, which is a contradiction.}
\end{proof}

\section{{$\cal{F}$}-strongly proper posets}\label{sec:FWCstronglyProper}

In this section we transition into the main body of the paper. After briefly reviewing some important facts about strong genericity in Subsection \ref{subsec:ReviewStonrgGen}, we then define, in Subsection \ref{subsec:FwcStrongProper}, the class of {$\cal{F}$-strongly proper posets (see Notation \ref{WhatisF} for some information about $\cal{F}$ and see Definition \ref{def:fksp} for the exact definition).} This class, which includes the collapse $\ps$, consists of posets for which we may build various residue systems and thereby obtain strongly generic conditions for models of interest. 


\subsection{Review of Strongly Generic Conditions}\label{subsec:ReviewStonrgGen}

Here we review the definition and basic properties of strongly generic conditions. Much of this material was originally developed by Mitchell \cite{MitchellUgh}. Parts of our exposition here summarize the exposition in \cite{Krueger}, Section 1, to which we refer the reader for proofs.

\begin{definition} Let $N\prec H(\theta)$, where $\theta$ is regular. Let $\Q\in N$ be a poset. A condition $q\in\Q$ is said to be a \textbf{strongly $(N,\Q)$-generic condition} if for any set $D$ which is dense in $\Q\cap N$, $D$ is predense above $q$ in $\Q$.
\end{definition}

\begin{remark}
Note that if $\Q\in N$, with $N$ as above, then any strongly $(N,\Q)$-generic condition is also an $(N,\Q)$-generic condition. Moreover, $q$ is a strongly $(N,\Q)$-generic condition iff $q\Vdash_{\Q}\dot{G}_\Q\cap N$ is a $V$-generic filter over $\Q\cap N$.
\end{remark}

We now review a combinatorial characterization of strongly generic conditions, implicit in \cite{MitchellUgh}, Proposition 2.15, in terms of the existence of residue functions.

\begin{definition}\label{def:residuestuff} Suppose that $\Q\in N\prec H(\theta)$, $q\in\Q$, and $s\in\Q\cap N$. $s$ is said to be a \textbf{residue of $q$ to $N$} if for all $t\geq_{\Q\cap N}s$, $t$ and $q$ are compatible in $\Q$.

A \textbf{residue function for $N$ above $q$} is a function $f_N$  defined on $\Q/q$ so that for each $r\in\Q/q$, $f_N(r)$ is a residue of $r$ to $N$.

Finally, if $q,r\in\Q$ and $s\in\Q\cap N$, we say that $s$ is a \textbf{dual residue of $q$ and $r$ to $N$} if $s$ is a residue for both $q$ and $r$ to $N$.
\end{definition}

\begin{lemma} $q\in\Q$ is $(N,\Q)$-strongly generic iff there is a residue function for $N$ above $q$.
\end{lemma}

In the next subsection, we will isolate further properties of residue functions of interest. For now, we review the process by which strongly generic conditions allow us to break apart the forcing $\Q$ into a two-step iteration. 

\begin{notation} Let $\Q$ be a poset and $q\in\Q$. Suppose $\Q\in N\prec H(\theta)$ and $q$ is a strongly $(N,\Q)$-generic condition. Fix a $V$-generic filter $\bar{G}$ over $\Q\cap N$. In $V[\bar{G}]$, let $(\Q/q)/\bar{G}$ denote the poset where conditions are all $r\in (\Q/q)$ which are $\Q$-compatible with every condition in $\bar{G}$. The ordering is the same as in $\Q$. 
\end{notation}

The following two results originate in \cite{Mitchell}; our formulation of them follows \cite{Krueger}.

\begin{lemma}\label{lemma:residueForceQuotient} Suppose $\Q\in N\prec H(\theta)$ and $q$ is a strongly $(N,\Q)$-generic condition. Then for all $r\geq q$ and $s\in \Q\cap N$, $s$ is a residue of $r$ to $N$ iff $s\Vdash_{\Q\cap N} r\in(\Q/q)/\dot{G}_{\Q\cap N}$.
\end{lemma}

\begin{lemma}\label{lemma:quotientFUN} Suppose $\Q\in N\prec H(\theta)$ and $q$ is a strongly $(N,\Q)$-generic condition. Then
\begin{enumerate}
\item if $r\geq q$, $s\in \Q\cap N$, and $r$ and $s$ are $\Q$-compatible, then there exists $t\geq_{\Q\cap N} s$ so that $t$ is a residue of $r$ to $N$;
\item if $D\seq\Q$ is dense above $q$, then $\Q\cap N$ forces that $D\cap (\Q/q)/\dot{G}_{\Q\cap N}$ is dense in $(\Q/q)/\dot{G}_{\Q\cap N}$.
\end{enumerate}
\end{lemma}

\subsection{Exact Residue Functions and {$\cal{F}$}-Strong Properness}\label{subsec:FwcStrongProper}

Following Neeman (\cite{GiltonNeeman}), we next isolate the properties of residue functions (see Definition \ref{def:residuestuff}) of interest. We will apply this in our work to the iteration $\ps * \dot{\C}$, consisting of the collapse poset $\ps$ followed by a Magidor-style, club-adding iteration $\dot{\C}$. {Neeman also connected this with countable closure of the quotient forcing (\cite{GiltonNeeman}, subsection 2.2). However, we were not able to apply this analysis to the final poset, which also includes the specializing iteration, as we do not know if the quotients involving the specializing iteration are even strategically closed. This will lead us later, in Section \ref{section:stationarypreservation}, to an ad-hoc proof that the quotients of the final poset preserve stationary sets consisting of countable cofinality ordinals, without having $\om_1$-closed quotients. }

Recalling that we are working with pre-orders (in anticipation of working with iterations later), we begin our discussion with the following definition.


\begin{definition} Let $\Q$ be a poset which is $\om_1$-closed with sups. $D\seq\Q$ is said to be \textbf{countably $=^*$-closed} if
\begin{enumerate}
\item[(a)] for each $q\in D$ and  $r\in\Q$, if  $r=^*q$, then $r\in D$;
\item[(b)] if $\la q_n:n\in\om\ra$ is an increasing sequence of conditions all of which are in $D$ and if $q^*$ is a sup of the sequence, then $q^*\in D$.
\end{enumerate}
\end{definition}

Such sets $D$ will arise later as the domains of \emph{exact} (see below) residue functions, whose domains need not in general be all of the poset under consideration, but only a dense, $=^*$-closed subset. We will construct such functions in Proposition \ref{prop:gettingESRFs}.

The following is Neeman's notion of an exact strong residue function for $N$ with dense domain {above} $q$ (\cite{GiltonNeeman}, Definitions 1.6, 2.10), but with the requirement of strategic continuity strengthened to continuity.

\begin{definition}\label{def:ESRF} Let $\Q$ be a poset which is $\om_1$-closed with sups, and fix $N$ with $\Q\in N\prec H(\theta)$. Let $q\in\Q$.

A partial function $f:\Q/q\rightharpoonup\Q\cap N$ is said to be an \textbf{exact, strong residue function for $N$ above $q$} if it satisfies the following properties:
\begin{enumerate}
\item (dense domain) the domain of $f$ is a dense, countably $=^*$-closed subset $D$ of $\Q/q$;
\item (projection) $r\geq f(r)$ for all $r\in D$;
\item (order preservation) for all $r^*,r\in D$, if $r^*\geq r$, then $f(r^*)\geq f(r)$;
\item (strong residue) for any $r\in D$ and any $u\in\Q\cap N$ so that $u\geq f(r)$, there exists $r^*\geq r$ with $r^*\in D$ so that $f(r^*)\geq u$;
\item (countable continuity) if $\la r_n:n\in\om\ra$ is an increasing sequence of conditions in $D$ with a sup $r^*$, then $f(r^*)$ is a sup of $\la f(r_n):n\in\om\ra$.\footnote{Note that $r^*$ is in $D$ by (1) and also that the sequence $\la f(r_n):n\in\om\ra$ is increasing by (3).}
\end{enumerate}
We call such a pair $\la q,f\ra$ a \textbf{residue pair for $(N,\Q)$}, or just a \textbf{residue pair for $N$} if $\Q$ is clear from context.
\end{definition}

The following appears in \cite{GiltonNeeman} (Lemma 2.11).

\begin{lemma}\label{lemma:sepOP}
Suppose that $\Q$ is separative, $q\in\Q$, and that $f:\Q/q\lra\Q\cap N$ is a function satisfying properties (2) and (4) of Definition \ref{def:ESRF}. Then $f$ is order-preserving on its domain.
\end{lemma}

\begin{example}\label{example:collapsetrivialresfunctions} Let $\al<\ka$ be inaccessible. Then the function $f:\ps\lra\ps\res\al$ given by $f(p)=p\res\al$ is an exact, strong residue function for any $M\prec H(\theta)$ with $M\cap\ka=\al$ above the condition $\es$ and has all of $\ps$ as its domain.
\end{example}

Our next task is to isolate the models which for us will play the role of ``$N$" in Definition \ref{def:ESRF}.  First some notation which we will fix for the remainder of the paper.

\begin{notation}\label{notation:wo}
Let $\lhd$ be a fixed well-order of $H(\ka^+)$.
\end{notation}

In the following definitions and claims we make a standard use of continuous sequences of elementary substructures $M_\alpha$, where $|M_\alpha|  = \alpha < \kappa$ and $M_\alpha \cap \kappa = \alpha$, to form natural restrictions $\ps^* \cap M_\alpha$ of posets $\ps^*$ which are members of the models on the chain.
For ease of notation in describing such chains, we use terminology similar to \cite{Krueger} and introduce the notion of a $P$-suitable sequence, for a parameter $P$.


\begin{definition}\label{def:suitable}
Let $P\in H(\ka^+)$ be a parameter. We say that a sequence $\la M_\al:\al\in A\ra$ is \textbf{$P$-suitable} if
\begin{enumerate}
\item $A\in{\cal{F}^+}$;
\item for each $\al\in A$, $\al$ is inaccessible, $M_\al\cap\ka=\al$, $\,^{<\al}M_\al\seq M_\al$, and $|M_\al|=\al$;
\item for each $\al\in A$, $M_\al\prec (H(\ka^+),\in,\lhd)$ and $P\in M_\al$;
\item if $\al<\be$ are in $A$, then $M_\al\in M_\be$, and also if $\ga\in A\cap\lim(A)$, then $M_\ga=\bigcup\lb M_\de:\de\in A\cap\ga\rb$.
\end{enumerate}
We refer to a single model $M$ satisfying (2) and (3) as a \textbf{$P$-suitable model}.
\end{definition}

It is clear from the definition that if $\vec{M}$ is $P$-suitable and $B\seq\dom(\vec{M})$ is in {$\cal{F}^+$}, then $\la M_\al:\al\in B\ra$ is also $P$-suitable. It is also clear that for any $P\in H(\ka^+)$, there exists a $P$-suitable sequence.

The next definition is the main item of this section; it specifies a class of posets which contains the Levy collapse $\ps$ and each of which can play the role of a preparatory forcing for a $\aleph_2$-c.c. iteration specializing Aronszajn trees. (The work in Section \ref{Section:specialize} is devoted to showing this.) It is helpful to recall Notation \ref{notation:FI}.

\begin{definition}\label{def:fksp}
Let $\ps^*$ be a poset in $H(\ka^+)$ which is $\om_1$-closed with sups and which collapses all cardinals in the interval $(\om_1,\ka)$. {In the case that $\ps^*=\ps$ (recall that $\ps=\col(\om_1,<\ka)$), we assume that $\ka$ is weakly compact, and in the case that $\ps^*$ is not $\ps$, we assume that $\ka$ is ineffable. Let $\cal{F}$ denote the following filter on $\ka$:}
$$
{\cal{F}=}\begin{cases}
{\cal{F}_{WC}} & {\text{if }\ps^*=\ps}\\
{\cal{F}_{in}} & {\text{otherwise,}}
\end{cases}
$$
{and let $\cal{I}$ denote the ideal dual to $\cal{F}$.}

We say that $\ps^*$ is \textbf{{$\cal{F}$}-strongly proper} if for any $\ps^*$-suitable sequence $\vec{M}$ {there exist an  $A\seq\dom(\vec{M})$ with $\dom(\vec{M})\bsl A\in\cal{I}$}, a sequence $\la p^*(M_\al):\al\in A\ra$ of conditions in $\ps^*$, and a sequence $\la\vp^{M_\al}:\al\in A\ra$ of functions satisfying the following properties, for each $\al\in A$:
\begin{enumerate}
\item $\vp^{M_\al}$ is an exact, strong residue function for $(M_\al,\ps^*)$ above $p^*(M_\al)$ {and $\vp^{M_\al}(p^*(M_\al))=0_{\ps^*}$;}
\item if $\be\in A$ is greater than $\al$, then $p^*(M_\al)$ and $\vp^{M_\al}$ are members of $M_\be$.
\end{enumerate}
We will refer to the sequence of pairs $\la\la p^*(M_\al),\vp^{M_\al}\ra:\al\in A\ra$ as a \textbf{residue system} for $\vec{M}\res A$ and $\ps^*$.
\end{definition}

\begin{remark}\hfill
\begin{enumerate}
\item {A corollary of (1) of the above definition is that $p^*(M_\al)$ is compatible with every condition in $\ps^*\cap M_\al$.} Such conditions were called \emph{universal} in \cite{CoxKrueger}.
\item {Note that any $\cal{F}$-strongly proper poset has size exactly $\ka$. It has size at least $\ka$ since the $\mathsf{GCH}$ holds and it collapses all cardinals in the interval $(\om_1,\ka)$ and has size no more than $\ka$ since it is a member of $H(\ka^+)$.}
\end{enumerate}
\end{remark}

\begin{example}\label{example:LevyCollapseFKSP} The Levy collapse poset $\ps$ is an example of a {$\cal{F}$}-strongly proper poset {(noting that $\cal{F}=\cal{F}_{WC}$ in this case)}. Indeed, letting $\vec{M}$ be any suitable sequence, {we may take the set $A$ in Definition \ref{def:fksp} to just be $\dom(\vec{M})$.} Then we define $p(M_\al):=\es$ for all $\al\in A$ and define $\vp^{M_\al}$ on the entire poset $\ps$ by $\vp^{M_\al}(p)=p\res\al$. As stated in Example \ref{example:collapsetrivialresfunctions}, each $\vp^{M_\al}$ is an exact, strong residue function for $M_\al$ above $\es$. The remaining properties of Definition \ref{def:fksp} are trivial.
\end{example}

In our intended applications, the posets playing the role of $\ps^*$ in Definition \ref{def:fksp} will be of the form $\ps\ast\dot{\C}$, where $\dot{\C}$ is a $\ps$-name for an iteration of club-adding with anticipation.

We now check, by a standard argument, that forcing with an {$\cal{F}$}-strongly proper poset preserves $\ka$.

\begin{lemma}\label{lemma:fkspPreserves} Suppose that $\ps^*$ is $\cal{F}$-strongly proper. Then forcing with $\ps^*$ preserves $\ka$.
\end{lemma}
\begin{proof} By Definition \ref{def:fksp}, we know that $\ps^*$ preserves $\om_1$ and collapses all cardinals in the interval $(\om_1,\ka)$. Thus if $\ps^*$ does not preserve $\ka$, then we may find a condition $p\in\ps^*$ and a $\ps^*$-name $\dot{f}$ for a function with domain $\om_1$ which $p$ forces is cofinal in $\ka$. Since $\ps^*$ has size $\ka$, we may assume that the name $\dot{f}$ is a member of $H(\ka^+)$. Let $\vec{M}$ be a $\lb\dot{f},p,\ps^*\rb$-suitable sequence. Also let $A\seq\dom(\vec{M})$ witness Definition \ref{def:fksp}, and let $N$ denote the least model on the sequence $\vec{M}\res A$. By Definition \ref{def:fksp}, we may find a condition $p^*(N)$ and an exact, strong residue function for $(N,\ps^*)$ above $p^*(N)$. Since $p\in N$, Definition \ref{def:fksp}(1) implies that $p^*(N)$ and $p$ are compatible. So let $q$ be an extension of them both. Then since $q$ is an $(N,\ps^*)$-strongly generic condition and $\dot{f}\in N$, $q$ forces that $\ran(\dot{f})\seq N\cap\ka<\ka$. But this contradicts the fact that $p$ forces that $\dot{f}$ is unbounded in $\ka$.
\end{proof}

The remainder of the subsection is dedicated to proving lemmas about how suitable sequences interact with the weak compactness of $\ka$.

\begin{lemma}\label{lemma:UsefulSuitable0} Let $P\in H(\ka^+)$ and $\vec{M}$ be $P$-suitable. Then $P\seq\bigcup_{\al\in\dom(\vec{M})}M_\al$.
\end{lemma}
\begin{proof} By definition of a suitable sequence, each model on the sequence is elementary with respect to the fixed well-order $\lhd$ on $H(\ka^+)$, and therefore each model contains the $\lhd$-least surjection $\psi$ from $\ka$ onto $P$. Then 
$$
P=\psi[\ka]=\bigcup_{\al\in\dom(\vec{M})}\psi[\al]\seq\bigcup_{\al\in\dom(\vec{M})}M_\al.
$$
\end{proof}

\begin{lemma}\label{lemma:UsefulSuitable} Let $P\in H(\ka^+)$, and let $\vec{M}$ be $P$-suitable. Suppose that $M^*$ is a $\ka$-model containing $\vec{M}$ { and that $U$ is an $M^*$-normal ultrafilter on $\ka$ so that $\dom(\vec{M})\in U$. Let $j:M^*\lra N$ be the ultrapower embedding.} Then
\begin{enumerate}
\item $\ka\in\dom(j(\vec{M}))$, and $j(\vec{M})(\ka)=\bigcup_{\al\in\dom(\vec{M})}j[M_\al]$;
\item $j(P)\cap j(\vec{M})(\ka)=j[P];$
\item $\bigcup_{\al\in\dom(\vec{M})}M_\al$ is transitive, and $j^{-1}\res M_\ka$ is the transitive collapse of $M_\ka$.
\end{enumerate}
\end{lemma}
\begin{proof} Let $B:=\dom(\vec{M})$. The first part of item (1) follows {since $B\in U=\lb X\in\cal{P}(\ka)\cap M^*:\ka\in j(X)\rb$.}  Thus $\ka\in\dom(j(\vec{M}))$, and so we may let $M_\ka:=j(\vec{M})(\ka)$. Additionally, $j(B)\cap\ka=B$, and so by Definition \ref{def:suitable}(4), $M_\ka=\bigcup_{\al\in B}j(M_\al)$. But $|M_\al|=\al<\ka$ for each $\al\in B$, and hence $j(M_\al)=j[M_\al]$. Thus
$$
M_\ka=\bigcup_{\al\in B}j[M_\al],
$$
completing the proof of (1). (2) follows immediately. 

For (3), observe that if $x\in M:=\bigcup_{\al\in\dom(\vec{M})}M_\al$, then a tail of the sequence $\vec{M}$ is $x$-suitable, and so $x\seq M$ by Lemma \ref{lemma:UsefulSuitable0}. Thus $M$ is transitive. Since $j^{-1}\res M_\ka$ is an $\in$-isomorphism whose range (namely $M$) is transitive, $j^{-1}$ is the transitive collapse.
\end{proof}

\begin{lemma}\label{lemma:refineToGetSP} Suppose that $\vec{M}$ is $\ps^*$-suitable and that $\la\la p^*(M_\al),\vp^{M_\al}\ra:\al\in\dom(\vec{M})\ra$ is a residue system for $\vec{M}$ and $\ps^*$. 
Then we may { find some $A\seq\dom(\vec{M})$ with $\dom(\vec{M})\bsl A\in\cal{I}$} so that
for all $\al\in A$, $\ps^*\cap M_\al\Vdash\check{\al}=\dot{\aleph}_2$.
\end{lemma}
\begin{proof}{ Suppose that $\vec{M}$ and $\la\la p^*(M_\al),\vp^{M_\al}\ra:\al\in\dom(\vec{M})\ra$ are as in the statement of the lemma. For a contradiction, assume that }
$$
{B:=\lb\al\in\dom(\vec{M}):\ps^*\cap M_\al\not\Vdash\check{\al}=\dot{\aleph}_2\rb\in\cal{F}^+.}
$$
{Since $\cal{F}_{WC}\seq\cal{F}$, $B$ is also in $\cal{F}_{WC}^+$.} 

{Let $M^*$ be a $\ka$-model containing $\vec{M}$, $B$, $\ps^*$, and the sequence $\la\la p^*(M_\al),\vp^{M_\al}\ra:\al\in\dom(\vec{M})\ra$. By Proposition \ref{proposition:characterizefwc}, since $\ka\bsl B\notin\cal{F}_{WC}$, we may find some $M^*$-normal measure $U$ so that, letting $j:M^*\lra N$ be the associated ultrapower embedding, $\ka\in j(B)$.} {Let $M_\kappa=j(\vec{M})(\kappa)$. Then}   {$N$ satisfies that $j(\ps^*)\cap M_\ka$ does not force that $\check{\ka}=\dot{\aleph}_2$. On the other hand, by the previous lemma, we know that $j(\ps^*)\cap M_\ka=j[\ps^*]$. Since $j[\ps^*]$ is isomorphic to $\ps^*$ and $\ps^*$ forces that $\ka$ becomes $\aleph_2$ (by Lemma \ref{lemma:fkspPreserves}), this implies that $j[\ps^*]$ forces that $\ka=\dot{\aleph}_2$. Thus $N$ also satisfies that $j[\ps^*]=j(\ps^*)\cap M_\ka$ forces that $\check{\ka}=\dot{\aleph}_2$, a contradiction.}

\end{proof}

We recall that in Definition \ref{def:fksp}(1), the condition $p^*(M_\al)$ is required to be compatible with every condition in $\ps^*\cap M_\al$. A practical corollary of this is that any generic for $\ps^*$ contains plenty of conditions of the form $p^*(M_\al)$.

\begin{lemma}\label{lemma:unbounded} Suppose that $\vec{M}$ is $\ps^*$-suitable, that $
\la\la p^*(M_\al),\vp^{M_\al}\ra:\al\in\dom(\vec{M})\ra$
is a residue system for $\vec{M}$ and $\ps^*$, and that $B\seq\dom(\vec{M})$ is in {$\cal{F}^+$}. Suppose that there is a condition $\bar{p}\in\ps^*$ satisfying that for each $\al\in B$, there is a condition $p_\al\in\dom(\vp^{M_\al})$ so that $\bar{p}=^*\vp^{M_\al}(p_\al)$. Then

$$
\bar{p}\Vdash \dot{X}:=\lb\al<\ka:p_\al\in \dot{G}_{\ps^*}\rb\text{ is unbounded in }\ka.
$$

In particular, taking $p_\al:=p^*(M_\al)$ {with $\bar{p}$ the trivial condition, and recalling Definition \ref{def:fksp}(1),}
$$
\ps^*\Vdash \lb\al<\ka:p^*(M_\al)\in \dot{G}_{\ps^*}\rb\text{ is unbounded in }\ka.
$$
Moreover, letting $E$ abbreviate $\dom(\vec{M})$, if $\al\in E\cap\lim(E)$, then 
$$
\Vdash_{\ps^*\cap M_\al}\lb\xi<\al:p^*(M_\xi)\in\dot{G}_{\ps^*\cap M_\al}\rb\text{ is unbounded in }\al.
$$

\end{lemma}
\begin{proof} Let $p\in\ps^*$ be a condition extending $\bar{p}$, and let $\ga<\ka$. We find an extension of $p$ which forces that $\dot{X}\bsl\ga\neq\es$. Since $B$ is unbounded in $\ka$ and $\vec{M}$ is $\ps^*$-suitable, Lemma \ref{lemma:UsefulSuitable0} implies that $p\in M_\be$ for some $\be\in B\bsl(\ga+1)$. By definition of an exact, strong residue function, $p_\be$ is compatible with $p\geq\bar{p}=^*\vp^{M_\be}(p_\be)$. {Therefore, let $p^*$ be a common extension; then $p^*\Vdash\be\in\dot{X}\bsl\ga$.}

The proof in the case $\al\in E\cap\lim(E)$ is identical, using the fact that, by Definition \ref{def:suitable}, $M_\al=\bigcup_{\xi\in B\cap\al}M_\xi$ in this case.
\end{proof}

\section{{$\cal{F}$}-Strongly Proper Posets and Specializing Aronszajn Trees on $\om_2$}\label{Section:specialize}\label{section:specialize}

In this section we will prove that if $\ps^*$ is an {$\cal{F}$}-strongly proper poset, then we can iterate to specialize Aronszajn trees on $\ka$ in the extension by $\ps^*$ {(see Definition \ref{def:fksp} for the definition of $\cal{F}$).} Recall from Lemma \ref{lemma:fkspPreserves} that $\ps^*$ forces that $\ka$ becomes $\aleph_2$ and also preserves the $\ch$; thus there are in fact Aronszajn trees on $\ka$ in any $\ps^*$-extension. By Example \ref{example:LevyCollapseFKSP}, the collapse poset $\ps$ is {$\cal{F}$}-strongly proper, and therefore, {in the case that $\ps^*=\ps$,} our results here generalize those of \cite{LS} {(recall that $\cal{F}=\cal{F}_{WC}$ when $\ps^*=\ps$).}

We consider a countable support iteration $\dot{\S} = \la \dot{\S}_\xi, \dot{\S}(\xi) : \xi < \kappa^+\ra$ of length $\kappa^+$, specializing Aronszajn trees on $\ka$ in the $\ps^*$-extension. More precisely, $\dot{\S}$ is a $\ps^*$-name for an iteration with countable support so that for any $\xi<\ka^+$, $\dot{\S}(\xi)$ is an $\dot{\S}_\xi$-name for the poset $\dot{\S}(\dot{T}_\xi)$, where $\dot{T}_\xi$ is a nice $\dot{\S}_\xi$-name for an Aronszajn tree on $\ka$; see Definition \ref{def:specposet} for the exact definition of posets of the form $\S(T)$. The selection of the names $\dot{T}_\xi$ - and hence the definition of the iteration - is determined by using the fixed well-order $\lhd$ of $H(\ka^+)$ from Notation \ref{notation:wo} as a bookkeeping function. In particular, for each $\xi<\ka^+$, the name $\dot{\S}_\xi$ is definable in $(H(\ka^+),\in,\lhd)$ from $\ps^*$ and $\xi$, and consequently it is a member of any model which is suitable with respect to $\ps^*$ and $\xi$. We will use $\R_\xi$ to abbreviate $\ps^*\ast\dot{\S}_\xi$ for each $\xi< \kappa^+$.

Since the poset $\R_\xi$ is $\om_1$-closed, it is straightforward to see that $\R_\xi$ has a dense set of determined conditions, i.e., conditions $(p,\dot{f})$ for which there is some function $f$ in $V$ so that $p\Vdash_{\ps^*}\dot{f}=\check{f}$. 
The dense set of determined conditions is also closed under sups of countable increasing sequences. Thus we will assume that all future conditions are determined.

\begin{notation}\label{notation:domain}
Strictly speaking, the domain of a (determined) condition $f$ in $\R_\xi$  is a countable subset of $\xi$, and for each $\zeta\in \dom(f)$, $f(\zeta)$ is itself a function whose domain is a countable subset of $\om_1$. However, we will often make an  abuse of notation and write $f(\zeta,\nu)$ to mean the countable set of tree nodes $f(\zeta)(\nu)$.
\end{notation}

The main goal of this section is to prove that $\dot{\S}$ is forced to be $\kappa$-c.c. For this it suffices to prove the following:

\begin{theorem}\label{thmSkappacc}
For every $\rho<\kappa^+$  it is forced by the trivial condition of $\ps^*$ that $\dot{\S}_\rho$ is $\kappa$-c.c. 
\end{theorem}

We will prove Theorem \ref{thmSkappacc} by induction on $\rho$. Doing so will require two induction hypotheses, the first of which is the following:
\vspace{3mm}

\textbf{Inductive Hypothesis I:} For each $\xi<\rho$, $\ps^*\Vdash\dot{\S}_\xi$ is $\ka$-c.c.
\vspace{3mm}

We will assume Inductive Hypothesis I throughout the entire section. Later in the section, after developing more of the theory, we will introduce a second, more technical inductive hypothesis; we state this after Remark \ref{remark:splittingfunction}. Though we assume the first inductive hypothesis throughout, we will only use the second inductive hypothesis once it is introduced, and the results prior to the statement thereof do not require it.

The rest of the section will proceed as follows. We will first establish, in Proposition \ref{proposition:ccreflect}, that for all $\xi<\rho$, there are plenty of intermediate generic extensions between $V$ and the full $\R_\xi$-extension in which various restrictions of Aronszajn trees are Aronszajn in the intermediate model. In light of this, we will define analogues of the ``hashtag" and ``star" principles from \cite{LS}; the former will say that two conditions have the same restriction to a given model, whereas the latter says that two conditions have a dual residue to a given model. Afterwards, we define the notion of a splitting pair of conditions, a notion which will play a key role in later amalgamation arguments. Next, we will state our second induction hypothesis, which describes the interplay between the star and hashtag principles. Using the second induction hypothesis, we will prove that splitting pairs exist and isolate sufficient conditions under which they can be amalgamated (see Lemma \ref{lemma:amalgamatesplittingpair}). Finally, we show that $\dot{\S}_\rho$ is forced to be $\ka$-c.c., and we verify that the second induction hypothesis holds at $\rho$. As mentioned in Remark \ref{rmk:iewc}(1), the only substantial use of the ineffability of $\ka$ is in verifying that the second induction hypothesis holds at $\ka$.

\begin{definition}\label{def:almostextend} Let $G^*$ be $V$-generic over $\ps^*$. If $\xi\leq\rho$, $f\in\S_\xi$, and $\bar{f}$ is a function (not necessarily a condition), we write $f\geq\bar{f}$ to mean that $\dom(\bar{f})\seq\dom(f)$ and for all $\la\zeta,\nu\ra\in\dom(\bar{f})$, $\bar{f}(\zeta,\nu)\seq f(\zeta,\nu)$.
\end{definition}

The next item establishes the existence of the desired intermediate generic extensions between $V$ and $V[\R_\xi]$ for $\xi<\rho$, and in turn the existence of plenty of residues. We recommend recalling Lemma \ref{lemma:UsefulSuitable} and Notation \ref{notation:FI} before reading the proof. {In the statement of the following proposition, we will assume that the various names are nice names for subsets of $H(\ka^+)$. Thus, in light of our discussion about determined conditions, the name $\dot{\S}_\zeta$ will be viewed as a union of sets of the form $\lb f\rb\times A_f$, where $f:\zeta\times\om_1\rightharpoonup\ka\times\om_1$ is a countable partial function and $A_f\seq\ps^*$ is an antichain. This will ensure, for instance, that $\dot{\S}_\zeta\cap M_\al$ is really a $(\ps^*\cap M_\al)$-name. Similar considerations apply to the $\R_\zeta$-name $\dot{T}_\zeta$.}

\begin{proposition}\label{proposition:ccreflect} Suppose that $\vec{M}$ is $\R_\rho$-suitable. Then there exists {$B^*\seq\dom(\vec{M})$ with $\dom(\vec{M})\bsl B^*\in\cal{I}$} so that for any $\al\in B^*$, {for any residue pair $\la p^*(M_\al), \vp^{M_\al}\ra$ for $(M_\al,\ps^*)$, and for any $\zeta\in M_\al\cap\rho$, }the following are true:
\begin{enumerate}
\item $(p^*(M_\al),0_{\dot{\S}_\zeta})$ forces that $\dot{G}_{\R_\zeta}\cap {M_\al}$ is a $V$-generic filter for $\R_\zeta\cap M_\al$;
\item $(p^*(M_\al),0_{\dot{\S}_\zeta})$ forces that $\dot{T}_\zeta\cap (\check{\om}_1\times\check{\al})=(\dot{T}_\zeta\cap M_\al)[\dot{G}_{\R_\zeta}\cap M_\al]$;
\item $(\ps^*\cap M_\al)\Vdash(\dot{\S}_\zeta\cap M_\al)$ is $\al$-c.c. Furthermore, $(\R_\zeta\cap M_\al)\Vdash\dot{T}_\zeta\cap {M_\al}$ is an Aronszajn tree on $\al=\dot{\aleph}_2$.
\end{enumerate}
\end{proposition}
\begin{proof} {Fix an $\R_\rho$-suitable sequence $\vec{M}$, and let $B:=\dom(\vec{M})$ so that $B\in\cal{F}^+$ by Definition \ref{def:suitable}. Since $\ps^*$ is $\cal{F}$-strongly proper, we may assume that $B$ satisfies the conclusion of Definition \ref{def:fksp}, by removing an $\cal{I}$-null set if necessary. Our goal is to show that each of (1)-(3) fail only on a set in $\cal{I}$. Thus we define $B_1$ to be the set of $\al\in B$ so that for some $\zeta\in M_\al$ and some residue pair $\la p^*(M_\al),\vp^{M_\al}\ra$, (1) fails for these objects; we define $B_2$ to be the set of $\al\in B$ so that for some $\zeta\in M_\al\cap\rho$ and some residue pair $\la p^*(M_\al),\vp^{M_\al}\ra$, (2) fails for $M_\al$ and $\zeta$; and we define $B_3$ similarly. We show that each of these is in $\cal{I}$.}

{Suppose for a contradiction that $B_i\in\cal{F}^+$ for some $i\in\lb 1,2,3\rb$. Then $B_i\in\cal{F}_{WC}^+$ since $\cal{F}_{WC}\seq \cal{F}$. Let $M^*$ be a $\ka$ model containing $B_i$ as well as the sequence $\vec{M}$. Applying Proposition \ref{proposition:characterizefwc}, we may fix an $M^*$-normal ultrafilter $U$ containing $B_i$, and we let $j:M^*\lra N$ be the induced ultrapower embedding. In particular $\ka\in j(B_i)$.}

Let $M_\ka:=j(\vec{M})(\ka)$. Fix $\zeta^*\in M_\ka\cap j(\rho)$ for the remainder of the proof which witnesses the relevant failure of (1), (2), or (3) on the $j$-side. Since $M_\ka=j\left[\bigcup_{\al\in B}M_\al\right]$, by Lemma \ref{lemma:UsefulSuitable}, we have that $\zeta^*\in\ran(j)$, and so $\zeta^*=j(\zeta)$ for some $\zeta<\rho$. Moreover, $M_\ka\cap j(\ps^*)=j[\ps^*]$.\\

{\underline{Case 1: $i=1$} Since $\ka\in j(B_1)$, we may fix a residue pair $\la p^*(M_\ka),\vp^{M_\ka}\ra$ for $(M_\ka,j(\ps^*))$ so that $(p^*(M_\ka),0_{j(\dot{\S}_\zeta)})$ does not force that $\dot{G}_{j(\R_\zeta)}\cap M_\ka$ is generic over $j(\R_\zeta)\cap M_\al$. }

{To obtain our contradiction, we show that $(p^*(M_\ka),0_{j(\dot{\S}_\zeta)})$ in fact \emph{does} force that $\dot{G}_{j(\R_\zeta)}\cap M_\ka$ is generic over $j(\R_\zeta)\cap M_\ka$.} Thus fix an extension $(q^*,\dot{g})$ of $(p^*(M_\ka),0_{j(\dot{\S}_\zeta)})$ in $j(\R_\zeta)$. Let $A^*\in N$ be a maximal antichain of $j(\R_\zeta)\cap M_\ka=j[\R_\zeta]$, and we will find some extension of $(q^*,\dot{g})$ which forces that $A^*\cap\dot{G}_{j(\R_\zeta)}\neq\es$. Since $q^*$ extends $p^*(M_\ka)$ in $j(\ps^*)$ and $\vp^{M_\ka}$ is an exact, strong residue function, we may extend and relabel, if necessary, to assume that $q^*\in\dom(\vp^{M_\ka})$. Then $\vp^{M_\ka}(q^*)\in j(\ps^*)\cap M_\ka=j[\ps^*]$. Therefore $\vp^{M_\ka}(q^*)=j(q)$ for some $q\in\ps^*$.

Now let $A:=j^{-1}[A^*]$; since $A^*$ is a maximal antichain in $j[\R_\zeta]$, $A$ is a maximal antichain in $\R_\zeta$. However, note that since $\R_\zeta$ is not necessarily $\ka$-c.c., $A$ could very well have size $\ka$, and therefore we cannot assume that it is an element of $M^*$. Until after the proof of the next claim, we will work in $V$, not $M^*$. Let $\dot{A}(1)$ be the $\ps^*$-name for
$
\lb f\in\dot{\S}_\zeta:(\exists p\in\dot{G}_{\ps^*})\;\;(p,f)\in A\rb.
$
Then $q\Vdash\dot{A}(1)$ is a maximal antichain in $\dot{\S}_\zeta$.

Since $\ps^*\Vdash\dot{\S}_\zeta$ is $\ka$-c.c., by Inductive Hypothesis I, we may extend $q$ in $\ps^*$ to some condition $q'$ and find an ordinal $\be<\ka$ and a sequence $\la\dot{f}_\ga:\ga<\be\ra$ of $\ps^*$-names so that
$$
q'\Vdash^V_{\ps^*}\dot{A}(1)=\lb\dot{f}_\ga:\ga<\be\rb.
$$
Since a given $\dot{f}_\ga$ needn't be a member of $M^*$, we show how to replace these names with ones that are in $M^*$, up to extending $q'$. 

\begin{claim}\label{claim1} There exist a condition $u\geq_{\ps^*}q'$
and a sequence of $\ps^*$-names $\la\dot{h}_\ga:\ga<\be\ra$ in $M^*$ so that 
$u\Vdash^V(\forall\ga<\be)\,\dot{h}_\ga=\dot{f}_\ga$.
\end{claim}
\noindent\emph{Proof.}
To find $u$, let $G^*$ be a $V$-generic filter over $\ps^*$ containing $q'$. Then $\S_\zeta:=\dot{\S}_\zeta[G^*]\seq\bigcup_{\eta\in B}M_\eta[G^*]$, since $\R_\zeta\seq\bigcup_{\eta\in B}M_\eta$.  Since $\ka=\aleph_2^{V[G^*]}$ and $\be<\ka$, there exists some $\eta\in B$ so that for all $\ga<\be$, $\dot{f}_\ga[G^*]\in M_\eta[G^*]$. By Lemma \ref{lemma:unbounded}, there exists a $\de\geq\eta$ so that $p^*(M_\de)\in G^*$. Now let $u\in G^*$ be an extension of $p^*(M_\de)$ and $q'$ so that $u\Vdash(\forall\ga<\be)\,\dot{f}_\ga\in M_\de[\dot{G}_{\ps^*}]$.

Back in $V$ we define new names $\dot{h}_\ga$ for each $\ga<\be$; recalling Notation \ref{notation:domain}, we view conditions in $\dot{\S}_\zeta$ as having a domain which is a countable subset of $\zeta\times\om_1$ so that each element in the range is a countable subset of $\ka\times\om_1$.

For each $\ga<\be$, {each} $\bar{\zeta}\in\zeta\cap M_\de$ (an iteration stage), {each} $\nu<\om_1$ (corresponding to the $\nu$th tree antichain), and {each} $\theta\in\de\times\om_1$ (a node with height below $\de$), let $A(\ga,\bar{\zeta},\nu,\theta)$ be a maximal antichain in $\ps^*\cap M_\de$ of conditions $p$ which decide whether or not $\theta$ is a member of $\dot{f}_\ga(\bar{\zeta},\nu)$. Let $\dot{h}_\ga$ be the $\ps^*$-name which is interpreted in an arbitrary generic extension via some $G^*$ as follows: $\theta\in {h}_\ga(\bar{\zeta},\nu)$ iff there is some $p\in A(\ga,\bar{\zeta},\nu,\theta)\cap G^*$ which forces that $\theta\in\dot{f}_\ga(\bar{\zeta},\nu)$. Otherwise $h_\ga$ is undefined.

We claim that $u\Vdash^V(\forall\ga<\be)\,\dot{h}_\ga=\dot{f}_\ga$. To see this, let $G^*$ be a $V$-generic filter containing $u$. Fix $\bar{\zeta}<\zeta$ and $\nu<\om_1$, and we verify that $f_\ga(\bar{\zeta},\nu)=h_\ga(\bar{\zeta},\nu)$. On the one hand, if $\theta\in h_\ga(\bar{\zeta},\nu)$, then by definition $f_\ga(\bar{\zeta},\nu)$ is defined and also contains the node $\theta$. 

On the other hand, if  $\tau\in f_\ga(\bar{\zeta},\nu)$ is a node, then since $f_\ga\in M_\de[G^*]$ has a countable domain, $\bar{\zeta}\in M_\de[G^*]$, and since $f_\ga(\bar{\zeta},\nu)\in M_\de[G^*]$ is countable, $\tau\in M_\de[G^*]$ too. But $M_\de[G^*]\cap V=M_\de$, since $u\in G^*$ is a (strongly) $(M_\de,\ps^*)$-generic condition (as it extends $p^*(M_\de)$). Hence $\bar{\zeta},\tau\in M_\de$. Thus $\bar{\zeta}\in M_\de\cap\zeta$ and $\tau$ has height below $M_\de\cap\ka=\de$. Next, since $u$ is a strongly $(M_\de,\ps^*)$-generic condition which is in $G^*$, and since $A(\ga,\bar{\zeta},\nu,\tau)$ is a maximal antichain $\ps^*\cap M_\de$, we know that $A(\ga,\bar{\zeta},\nu,\tau)\cap G^*\neq\es$, say with $\bar{u}$ in the intersection. But as $\tau\in f_\ga(\bar{\zeta},\nu)$, we must have that $\bar{u}$ forces that $\tau\in\dot{f}_\ga(\bar{\zeta},\nu)$, and hence $\tau\in h_\ga(\bar{\zeta},\nu)$. This completes the proof that $u\Vdash\dot{f}_\ga=\dot{h}_\ga$ for each $\ga<\be$.

Finally, since $M^*$ is $<\ka$-closed,
the sequence of antichains
$$
\la A(\ga,\bar{\zeta},\nu,\theta):\ga<\be,\bar{\zeta}\in M_\de\cap\zeta,\nu<\om_1,\theta\in(\de\times\om_1)\ra
$$
is a member of $M^*$. Therefore, the sequence $\la\dot{h}_\ga:\ga<\be\ra$ is a member of $M^*$ too.
\qed(Claim \ref{claim1})

\vspace{.2in}

Continuing with the main body of the argument, let $u\geq_{\ps^*}q'$ and $\la\dot{h}_\ga:\ga<\be\ra$ witness the above claim. Since $q'\Vdash_{\ps^*}\dot{A}(1)=\lb\dot{f}_\ga:\ga<\be\rb$ and $u\geq q'$, we have
$$
(*)\;\;u\Vdash_{\ps^*}\lb\dot{h}_\ga:\ga<\be\rb\;\text{ is a maximal antichain in }\;\dot{\S}_\zeta.
$$ 
Since $\la\dot{h}_\ga:\ga<\be\ra$ and $u$ are in $M^*$, $(*)$ is satisfied in $M^*$. Applying $j$, 
$$
j(u)\Vdash^N_{j(\ps^*)}\lb j(\dot{h}_\ga):\ga<\be\rb\text{ is a maximal antichain in }j(\dot{\S}_\zeta).
$$

Next, $u\geq_{\ps^*}q'\geq_{\ps^*} q$ so $j(u)\geq_{j(\ps^*)} j(q)=\vp^{M_\ka}(q^*)$, and $j(u)\in j[\ps^*]\seq M_\ka$ so $j(u)$ and $q^*$ are compatible in $j(\ps^*)$. Let $q^{**}$ be a condition extending both of them with $\vp^{M_\ka}(q^{**})\geq j(u)$. Since $q^{**}$ extends $q^*$, which forces that $\dot{g}$ is a condition in $j(\dot{\S}_\zeta)$, $q^{**}$ forces this too. As $q^{**}\geq j(u)$ also forces that $\lb j(\dot{h}_\ga):\ga<\be\rb$ is a maximal antichain in $j(\dot{\S}_\zeta)$, we may find an extension $r^*$ of $q^{**}$, a $j(\ps^*)$-name $\dot{g}^*$, and an ordinal $\ga<\be$ so that
$$
r^*\Vdash\dot{g}^*\geq\dot{g},j(\dot{h}_\ga).
$$
We may also extend, if necessary, to assume that $r^*\in\dom(\vp^{M_\ka})$, since $r^*\geq q^*\geq p^*(M_\ka)$. Let $r\in\ps^*$ so that $\vp^{M_\ka}(r^*)=j(r)$.

Now $r^*\geq_{j(\ps^*)}q^{**}$ are both in $\dom(\vp^{M_\ka})$. Since $\vp^{M_\ka}$ is order-preserving, $j(r)=\vp^{M_\ka}(r^*)\geq\vp^{M_\ka}(q^{**})\geq j(u)$. Then $r\geq u$. As a result, $r\Vdash^V\dot{h}_\ga=\dot{f}_\ga\in\dot{A}(1)$. By definition of $\dot{A}(1)$, we may find some $\ps^*$-extension $r'$ of $r$ so that $(r',\dot{h}_\ga)$ extends some element $(r'_0,\dot{f})$ of $A$. Then $j(r'_0,\dot{f})\in A^*$, as $A=j^{-1}[A^*]$. Since $r'$ extends $r$ in $\ps^*$, we get that $j(r')\geq_{j(\ps^*)}j(r)=\vp^{M_\ka}(r^*)$. So $j(r')$ and $r^*$ are compatible in $j(\ps^*)$. Let $r^{**}$ be a condition extending both of them. Then $(r^{**},\dot{g}^*)$ extends $j(r'_0,\dot{f})$. Indeed, $r^{**}$ extends $j(r')$ which extends $j(r'_0)$. Furthermore, $r^{**}$ extends $r^*$ which forces that $\dot{g}^*\geq j(\dot{h}_\ga)$, and $r^*$ extends $j(r')$ which forces that $j(\dot{h}_\ga)\geq j(\dot{f})$. Thus $(r^{**},\dot{g}^*)$ extends $j(r'_0,\dot{f})$, and therefore
$$
(r^{**},\dot{g}^*)\Vdash_{j(\R_\zeta)}j(r'_0,\dot{f})\in A^*\cap\dot{G}_{j(\R_\zeta)}\neq\es.
$$
{However, $(r^{**},\dot{g}^*)$ also extends the starting condition $(q^*,\dot{g})$. This finishes the proof that $\ka$ is not a member of $j(B_1)$, which contradicts our initial case assumption otherwise.}\\

{\underline{Case 2: $i=2$}  In this case, we are assuming that $\ka\in j(B_2)$, and we will derive a contradiction. Let $\la p^*(M_\ka),\vp^{M_\ka}\ra$ be a residue pair for $(M_\ka,j(\ps^*))$ so that $(p^*(M_\ka),0_{j(\dot{\S}_\zeta)})$ does not force the desired equality. We will show, however, that this residue pair does in fact force the desired equality.}

Towards this end, let $G^*$ be $V$-generic for $j(\R_\zeta)$ containing $(p^*(M_\ka),0_{j(\dot{\S}_\zeta)})$, and let {$\bar{G}^*:=G^*\cap M_\ka$ which, by item (1), is a $V$-generic filter over $j(\R_\zeta)\cap M_\ka=j[\R_\zeta]$. Let $G:=j^{-1}[\bar{G}^*]$, which is $V$-generic over $\R_\zeta$. Let $j(T_\zeta)$ denote $j(\dot{T}_\zeta)[G^*]$, let $T_\zeta:=\dot{T}_\zeta[G]$, and let $T'_\zeta:=(j(\dot{T}_\zeta)\cap M_\ka)[\bar{G}^*]=j[\dot{T}_\zeta][\bar{G}^*]$.}


Since $\dot{T}_\zeta$ is a nice $\R_\zeta$-name for a tree order on $\ka$, for each $\tau,\theta\in\ka\times\om_1$, there exists an antichain $B_{\theta,\tau}$ of $\R_\zeta$ so that, letting $\operatorname{op}(\theta,\tau)$ denote the canonical name for $\la\theta,\tau\ra$,
$$
<_{\dot{T}_\zeta}=\bigcup\lb\lb\operatorname{op}(\theta,\tau)\rb\times B_{\theta,\tau}:\theta,\tau<\ka\rb.
$$
{It is straightforward to see from this that $T_\zeta=T'_\zeta$. So we will show that $T_\zeta$ equals the restriction of $j(T_\zeta)$ to $\ka\times\om_1$. However, we know that $j:M^*\lra N$ lifts to $j:M^*[G]\lra N[G^*]$, since $j[G]=\bar{G}^*\seq G^*$ and since each of the filters is generic over the appropriate models. From this it follows that $T_\zeta=j(T_\zeta)\cap(\ka\times\om_1)$. Therefore, the equality in (2) is in fact satisfied, which contradicts the assumption that $\ka\in j(B_2)$.}\\


{\underline{Case 3: $i=3$} Let $j(q)\in j[\ps^*]=j(\ps^*)\cap M_\ka$ be a condition forcing that $\dot{A}$ is a name for a $\ka$-sized antichain in $j(\dot{\S}_\zeta)\cap M_\ka$. Since $B$ satisfies Definition \ref{def:fksp}, and since $B_1\seq B$ is in $U$, we may fix a residue pair $(p^*(M_\ka),\vp^{M_\ka})$ for $(M_\ka,j(\ps^*))$. Let $q^*\geq q,p^*(M_\ka)$ be a condition, and let $G^*$ be a $V$-generic filter over $j(\ps^*)$ containing $q^*$. Then $\bar{G}^*$ is a $V$-generic filter over $j(\ps^*)\cap M_\ka=j[\ps^*]$. We recall that $j^{-1}:M_\ka\lra\bigcup_{\al\in B_1}M_\al$ is the transitive collapse, and that $j^{-1}$ lifts in the standard way from $M_\ka[G^*]$ to $(\bigcup_{\al\in B_1}M_\al)[G]$. Now let $A:=\dot{A}[\bar{G}^*]$ so that $A$ is an antichain in $(j(\dot{\S}_\zeta)\cap M_\ka)[\bar{G}^*]= j[\dot{\S}_\zeta][\bar{G}^*]=j(\dot{\S}_\zeta[G^*])\cap M_\ka[G^*]$. Also, $A$ is a member of $V[\bar{G}^*]$ and has size $\ka$ there. Applying the elementarity of $j^{-1}$, we have that $j^{-1}[A]=:\bar{A}$ is an antichain in $\dot{\S}_\zeta[G]$, where $G:=j^{-1}[\bar{G}^*]$ is $V$-generic over $\ps^*$. Finally, $\bar{A}\in V[G]$ since $j^{-1}[\dot{A}]$ is a $\ps^*$-name in $V$ and $\bar{A}=j^{-1}[\dot{A}][G]$. Since $\bar{A}$ has size $\ka$ in $V[G]=V[\bar{G}^*]$, this contradicts the assumption that $\ps^*\Vdash\dot{\S}_\zeta$ is $\ka$-c.c.}

{For the ``furthermore" part of (3), suppose now that $\dot{b}'$ is a $j(\R_\zeta)\cap M_\ka=j[\R_\zeta]$-name which $j(q)$ forces is a branch through $j(\dot{T}_\zeta)\cap M_\ka$. Fix $q^*$ and $G^*$ as in the previous paragraph. Then by (2), we see that $j(\dot{T}_\zeta)[G^*]\cap(\ka\times\om_1)=\dot{T}_\zeta[G]$. Moreover, $j(\dot{T}_\zeta)[G^*]\cap(\ka\times\om_1)=j(\dot{T}_\zeta)[G^*]\cap M_\ka[G^*]=(j(\dot{T}_\zeta)\cap M_\ka)[\bar{G}^*]$. Thus $\dot{b}'[\bar{G}^*]$ is a branch through $T_\zeta:=\dot{T}_\zeta[G]$. But $\dot{b}'[\bar{G}^*]$ is a member of $V[\bar{G}^*]$ and $V[\bar{G}^*]=V[G]$. Thus $T_\zeta$ is not Aronszajn in $V[G]$, a contradiction.}

{ Thus we see that $\ka$ cannot be a member of $j(B_3)$, completing Case 3 and thereby the proof.}
\end{proof}

In both the previous result and Definition \ref{def:fksp}, there was the apparent necessity of refining the domain of a suitable sequence so that various desired behavior obtains on each level of the refined sequence. The next item amalgamates this into one definition which we use frequently throughout.

\begin{definition} Let $\vec{M}$ be an $\R_\rho$-suitable sequence. We say that $\vec{M}$ is in \textbf{pre-splitting configuration up to $\rho$} if there is a residue system $\la\la p^*(M_\al),\vp^{M_\al}\ra:\al\in \dom(\vec{M})\ra$ satisfying {items (1) and (2) of Definition \ref{def:fksp}} (with respect to $\ps^*$) as well as items (1)-(3) of Proposition \ref{proposition:ccreflect} for all $\al\in\dom(\vec{M})$ (with respect to $\R_\rho$).
\end{definition}

\begin{definition}\label{def:restriction} Suppose that $\vec{M}$ is in pre-splitting configuration up to $\rho$, $\xi\leq\rho$, and $\al\in\dom(\vec{M})$ so that $\xi \in M_\alpha$. 
Fix a residue pair $\la p^*(M_\al), \vp^{M_\al}\ra$ for $(M_\al,\ps^*)$.
\begin{enumerate}
\item For a (determined) condition $(p,f)$, we define $f\res M_\al$ to be the function $\bar{f}$ with domain $\dom(f)\cap M_\al$ so that for each $\la\zeta,\nu\ra\in\dom(f)\cap M_\al$, 
$$
\bar{f}(\zeta,\nu)=f(\zeta,\nu)\cap M_\al.
$$
\item We define $D(\vp^{M_\al},\xi)$ to be the set of conditions $( p,f)\in\R_\xi$ so that $p\in \dom(\vp^{M_\al})$ and $(\vp^{M_\al}(p),f\res M_\al)$ is a condition in $\R_\xi$ (and hence in $\R_\xi\cap M_\al$).

\item If $( p,f)\in D(\vp^{M_\al},\xi)$, we make a slight abuse of notation and define $( p,f )\res M_\al$ to be the pair $(\vp^{M_\al}(p),f\res M_\al)$, when $\vp^{M_\al}$ is clear from context.
\end{enumerate}
\end{definition} 

We observe that in general, for $(p,f) \in D(\vp^{M_\al},\xi)$, although $(p,f)\res M_\al \in \R_\xi\cap M_\al$ is a condition, it need not be {a} residue of $(p,f)$ to $M_\al$ in the sense that it is possible for some $(p',f') \in \R_\xi\cap M_\al$ which extends $(p,f)\res M_\al$ to not be compatible with $(p,f)$.
However, $(p,f)\res M_\alpha$ must have \emph{some} extension $(\bar{p},\bar{f}) \in \R_\xi\cap M_\al$ which is a residue of $(p,f)$. 
This is because if $G \subseteq \R_\xi$ is $V$-generic and contains $(p,f)$, then by Proposition \ref{proposition:ccreflect} {and the definition of pre-splitting configuration,} $\bar{G}:=G\cap M_\al$ is $V$-generic over $\R_\xi\cap M_\al$ and $(p,f)\res M_\al \in \bar{G}$. Since $(p,f)\in G$ is compatible with every condition in $\bar{G}$, there must be some $(\bar{p},\bar{f}) \in \bar{G}$ which extends $(p,f)\res M_\al$ and is a residue of $(p,f)$, by Lemma \ref{lemma:residueForceQuotient}.

\begin{lemma}\label{lemma:denserestriction} Suppose that $\vec{M}$ is in pre-splitting configuration up to $\rho$.
Then for each $\al\in\dom(\vec{M})$, each residue pair $\la p^*(M_\al),\vp^{M_\al}\ra$ for $(M_\al,\ps^*)$, and each $\xi\leq\rho$ with $\xi \in M_\al$, $D(\vp^{M_\al},\xi)$ is dense and $=^*$-countably closed in $(\ps^*/p^*(M_\al))\ast\dot{\S}_\xi$.
\end{lemma}
\begin{proof} We will first prove the result for $\xi<\rho$ and then use this to prove the result at $\rho$. Note that the $=^*$-countable closure of $D(\vp^{M_\al},\xi)$, in either case for $\xi$, follows from the continuity of $\vp^{M_\al}$ and the {$=^*$-countable} closure of the posets. We therefore concentrate on showing density.

Let $(p_0,f_0)\in\R_\xi$ be given with $p_0\geq p^*(M_\al)$. By the observations following Definition \ref{def:restriction}, we may build increasing sequences $\la (p_n,f_n):n\in\om\ra$ and $\la(\bar{p}_n,\bar{f}_n):n\in\om\ra$ so that 
\begin{enumerate}
\item[{(i)}] $(\bar{p}_n,\bar{f}_n)$ is a residue of $(p_n,f_n)$ to $M_\al$ {with $\bar{p}_n$ extending $\vp^{M_\al}(p_n)$ and with $\bar{f}_n$ extending the function $f_n\res M_\al$ in the sense of Definition \ref{def:almostextend} (note that $f_n\res M_\al$ needn't be forced by $\vp^{M_\al}(p_n)$ to be a condition); }
\item[{(ii)}] $(p_{n+1},f_{n+1})$ extends both $(p_n,f_n)$ and $(\bar{p}_n,\bar{f}_n)$; 
\item[{(iii)}] $p_{n+1}\in \dom(\vp^{M_\al})$ and  $\vp^{M_\al}(p_{n+1})\geq\bar{p}_n$.
\end{enumerate}

Now let $(p^*,f^*)$ be a sup of $\la (p_n,f_n):n\in\om\ra$. Note that $p^*\in \dom(\vp^{M_\al})$ since $p_n\in \dom(\vp^{M_\al})$ for each $n$ and also that $\vp^{M_\al}(p^*)$ is a sup of $\la\vp^{M_\al}(p_n):n\in\om\ra$. We claim that $(\vp^{M_\al}(p^*),f^*\res M_\al)$ is a condition. 
Now $\vp^{M_\al}(p^*)\geq\vp^{M_\al}(p_{n+1})\geq\bar{p}_n$ for each $n$, so $\vp^{M_\al}(p^*)$ forces that $\la\bar{f}_n:n\in\om\ra$ is an increasing sequence of conditions in $\dot{\S}_\xi$, and therefore forces that $\bigcup_n\bar{f}_n$ is a condition too. {However, $\bar{f}_{n+1}$ extends (in the sense of Definition \ref{def:almostextend}) $f_{n+1}\res M_\al$ which extends $\bar{f}_n$ for all $n$. Therefore $\bigcup_n\bar{f}_n=\bigcup_n(f_n\res M_\al)=(\bigcup_nf_n)\res M_\al=f^*\res M_\al$, which finishes the claim.}

Now we show that the lemma holds for $\xi=\rho$. We deal with $\rho$ limit first. If $\cf(\rho)>\om$, then the result holds since any $(p,f)\in\R_\rho$ is in $\R_\xi$ for some $\xi<\rho$. On the other hand, if $\cf(\rho)=\om$, then let $\la\xi_n:n\in\om\ra$ be an increasing sequence of ordinals in $M_\al$ which is cofinal in $\rho$. By applying the lemma below $\rho$, we define an increasing sequence of extensions $\la (p_n,f_n):n\in\om\ra$ of $(p,f)$ so that $f_n\res[\xi_n,\rho)=f\res[\xi_n,\rho)$ and so that $(p_n,f_n\res\xi_n)\in D(\vp^{M_\al},\xi_n)$. Now let $p^*$ be a sup of $\la p_n:n\in\om\ra$ and $f^*:=\bigcup_nf_n$. Then $(p^*,f^*)$ extends $(p,f)$ and is a member of $D(\vp^{M_\al},\rho)$.

Finally, assume that $\rho=\rho_0+1$ is a successor, and let $(p,h)\in\R_\rho$ be given. By the remarks after Definition \ref{def:restriction}, we may find a residue $(\bar{p},\bar{h}_0)$ of $(p,h\res\rho_0)$ to $M_\al$ {with respect to the poset $\R_{\rho_0}$.} Recalling Notation \ref{notation:domain}, we use $h(\rho_0)\cap M_\al$ in what follows as an abuse of notation for $\la h(\rho_0)(\nu)\cap M_\al:\nu\in\dom(h(\rho_0))\ra$. By the elementarity of $M_\al$ and the fact that $h(\rho_0)\cap M_\al$ is a member of $M_\al$, we may find an extension of $(\bar{p},\bar{h}_0)$, say $(\bar{p}',\bar{h}'_0)$, which either forces that $h(\rho_0)\cap M_\al\in\dot{\S}(\rho_0)$ or forces that $h(\rho_0)\cap M_\al\notin\dot{\S}(\rho_0)$. Since $(\bar{p}',\bar{h}'_0)$ extends $(\bar{p},\bar{h}_0)$, it is compatible with $(p,h\res\rho_0)$, and hence it must force that $h(\rho_0)\cap M_\al\in\dot{\S}(\rho_0)$. Finally, since $D(\vp^{M_\al},\rho_0)$ is dense and $\rho_0\in M_\al$, we may find an extension $(q,g_0)$ of $(\bar{p}',\bar{h}'_0)$ and $(p,h\res\rho_0)$ which is in $D(\vp^{M_\al},\rho_0)$ {and which satisfies that $\vp^{M_\al}(q)\geq\bar{p}'$.} Then $(q,g_0)\res M_\al$ is a condition
which forces that $h(\rho_0)\cap M_\al$ is a condition in $\dot{\S}(\rho_0)$. Thus $(\vp^{M_\al}(q),(g_0\res M_\al)^\frown\la h(\rho_0)\cap M_\al\ra)$ is a condition and equals $(q,g_0^\frown\la h(\rho_0)\ra)\res M_\al$.
\end{proof}

\begin{notation}\label{notation:horizontalvertical} We will often find it useful to denote conditions in $\R_\xi$ by the letters $u,v$ and $w$. If $u\in\R_\xi$, we write $p_u$ and $f_u$ to denote the objects so that $u=( p_u,f_u)$. Furthermore, if $\zeta\leq\xi$, then we write $u\res\zeta$ to denote the pair $(p_u,f_u\res\zeta)$, which restricts the length.
This should not be confused with $u\res M_\alpha = (\vp^{M_\alpha}(p),f\res M_\alpha)$ from Definition \ref{def:restriction}, which restricts the height.
\end{notation}

The following definitions of $\#$ and $*$ are taken from \cite{LS} and modified to the current presentation. {The dual residue property defined in (2) of the upcoming definition is the natural translation into the current situation of the statement that ``$\#$ implies $*$" at $\alpha$ from \cite{LS}.}

\begin{definition}\label{def:starhashtag} 
Suppose that $\vec{M}$ is in pre-splitting configuration up to $\rho$, that $\al\in\dom(\vec{M})$, and that $\zeta\leq\rho$ is in $M_\al$. 
\begin{enumerate}
\item 
Fix conditions $u,v\in\R_\zeta$ and $w\in\R_\zeta\cap M_\al$. Fix a residue pair $\la p^*(M_\al), \vp^{M_\al}\ra$ for $(M_\al,\ps^*)$. 
\begin{enumerate}
    \item We say that 
$$ 
\#^\zeta_{\vp^{M_\al}}(u,v,w) 
$$ 
holds if $u,v\in D(\vp^{M_\al},\zeta)$ and $u\res M_\al=^*w=^*v\res M_\al$.\footnote{Note that if $(p,f)$ and $(q,g)$ are (determined) conditions with $(p,f)=^*(q,g)$, then $f=g$.}

\item We say that 
$$
*^\zeta_{\vp^{M_\al}}(u,v,w)
$$ 
holds if $u,v\in D(\vp^{M_\al},\zeta)$ and if $w\geq u\res M_\al,v\res M_\al$ is a dual residue for $u$ and $v$ (see Definition \ref{def:residuestuff}).\footnote{One can in fact argue that if $w$ is a dual residue, then it follows that $w\geq u\res M_\al$ and $w\geq v\res M_\al$; however, we don't need this fact.}
\end{enumerate}
\item We say that $\R_\zeta$ satisfies the \textbf{dual residue property} at $M_\al$ if for any residue pair $\la p^*(M_\al),\vp^{M_\al}\ra$ for $(M_\al,\ps^*)$ and any conditions $u,v,w$ so that $\#^\zeta_{\vp^{M_\al}}(u,v,w)$ holds, there exists  $w^*\geq_{\R_\zeta\cap M_\al} w$ so that $*^\zeta_{\vp^{M_\al}}(u,v,w^*)$ holds.
\end{enumerate}
\end{definition}

\begin{lemma}\label{lemma:starimplieshashtag} Suppose that $*^\zeta_{\vp^{M_\al}}(u,v,w)$ holds and that $D$ is dense and countably $=^*$-closed in $\R_\zeta/(p^*(M_\al),0_{\dot{\S}_\zeta})$. Then:
\begin{enumerate}
\item there exist $u'\geq u$, $v'\geq v$ with $u',v'\in D$ and there exists $w'\geq w$ so that $u'\res M_\al\geq w$, $v'\res M_\al\geq w$, and $*^\zeta_{\vp^{M_\al}}(u',v',w')$ hold;
\item there exist $u^*\geq u$ and $v^*\geq v$ with $u^*,v^*\in D$, and there exists $w^*\geq w$ so that $\#^\zeta_{\vp^{M_\al}}(u^*,v^*,w^*)$ holds.
\end{enumerate}
\end{lemma}
\begin{proof} First define $E$ to be the set of conditions $s$ in $\R_\zeta/(p^*(M_\al),0_{\dot{\S}_\zeta})$ so that $s\in D(\vp^{M_\al},\zeta)\cap D$ and so that either $s\res M_\al\geq w$ or $s\res M_\al$ is incompatible with $w$; then $E$ is dense in $\R_\zeta/(p^*(M_\al),0_{\dot{\S}_\zeta})$. Now fix a $V$-generic filter $\bar{G}$ over $\R_\zeta\cap M_\al$ containing $w$, and note that $u$ and $v$ are in $(\R_\zeta/(p^*(M_\al),{0_{\dot{\S}_\zeta})})/\bar{G}$. By Lemma \ref{lemma:quotientFUN}(2), we can find $u'\geq u$ and $v'\geq v$ so that $u',v'$ are in $E$ as well as in $(\R_\zeta/(p^*(M_\al),0_{\dot{\S}_\zeta}))/\bar{G}$. We next observe that $u'\res M_\al\in\bar{G}$. Indeed, since $u'\in D(\vp^{M_\al},\zeta)$, $u'\res M_\al$ is a condition in $\R_\zeta\cap M_\al$. Additionally, since  $u'\in(\R_\zeta/(p^*(M_\al),0_{\dot{\S}_\zeta}))/\bar{G}$ and $u'\geq u'\res M_\al$, we have that $u'\res M_\al$ (a condition) is compatible with every condition in $\bar{G}$. Thus $u'\res M_\al\in\bar{G}$. However, by definition of $E$, and since $w\in\bar{G}$, $u'\res M_\al$ must extend $w$. A symmetric argument shows that $v'\res M_\al\geq w$.

Now let $w'\in\bar{G}$ be a condition extending $w$ which forces that $u',v'$ are in $(\R_\zeta/(p^*(M_\al),0_{\dot{\S}_\zeta}))/\dot{\bar{G}}$. By Lemma \ref{lemma:residueForceQuotient}, we have that $*^\zeta_{\vp^{M_\al}}(u',v',w')$ holds. Since $u'\res M_\al$ and $v'\res M_\al$ both extend $w$, this completes the proof of (1).

For (2), suppose that we are given conditions {$u_0,v_0$, and $w_0$} so that $*^\zeta_{\vp^{M_\al}}(u_0,v_0,w_0)$ holds. By repeatedly applying (1), we may define a {coordinate-wise} increasing sequence $\la\la u_n,v_n,w_n\ra:n\in\om\ra$ so that for all $n\in\om$, $u_{n+1}\geq u_n$ and $v_{n+1}\geq v_n$; $u_{n+1}\res M_\al\geq w_n$ and $v_{n+1}\res M_\al\geq w_n$; and $*^\zeta_{\vp^{M_\al}}(u_n,v_n,w_n)$ holds. Let $u^*$ be a sup of $\la u_n:n\in\om\ra$, and let $v^*$ and $w^*$ be defined similarly. Since $*^\zeta_{\vp^{M_\al}}(u_n,v_n,w_n)$ holds for each $n$, by definition we have that $w_n\geq u_n\res M_\al,v_n\res M_\al$. Therefore the sequences $\la u_n\res M_\al:n\in\om\ra$ and $\la v_n\res M_\al:n\in\om\ra$ are each intertwined with $\la w_n:n\in\om\ra$, and consequently, they have suprema {which are $=^*$-related.} It follows by the continuity of $\vp^{M_\al}$ that $\#^\zeta_{\vp^{M_\al}}(u^*,v^*,w^*)$ holds. 
\end{proof}

Suppose that $\vec{M}$ is in pre-splitting configuration up to $\rho$, that $\al\in\dom(\vec{M})$ and that $\zeta\in M_\al\cap\rho$. Fix $\theta \in (\kappa\setminus \alpha) \times \om_1$, a node in the tree $\dot{T}_\zeta$ of level greater than or equal to $\al$.

Let $\dot{b}_\zeta(\theta,\al)$ denote the $\R_\zeta$-name for $\lb\bar{\theta}\in\al\times\om_1:\bar{\theta}<_{\dot{T}_\zeta}\theta\rb$. 
By Proposition \ref{proposition:ccreflect}(2), the condition 
$(p^*(M_\al),0_{\dot{\S}_\zeta})$ forces that
$\dot{b}_\zeta(\theta,\al)$ is a cofinal branch through $(\dot{T}_\zeta\cap M_\al)[\dot{G}_{\R_\zeta\cap M_\al}]$. Note that by Proposition \ref{proposition:ccreflect} and the definition of pre-splitting configuration,
$(\dot{T}_\zeta\cap M_\al)[G_{\R_\zeta\cap M_\al}]$ is an Aronsjzan tree on $\alpha$ in the $V$-generic extension $V[G_{\R_\zeta\cap M_\al}]$ over $\R_\zeta \cap M_\alpha$. In light of this, we make the following definition:

\begin{definition}\label{def:splitbelow} 

Let $\zeta<\rho$, $\al<\ka$, and $\la\theta,\tau\ra$ be a pair of tree nodes (possibly equal) at or above level $\al$, which we view as nodes in the tree $\dot{T}_\zeta$. We say that two conditions $u$ and $v$ in $\R_\zeta$ \textbf{split $\la\theta,\tau\ra$ below $\alpha$ in $\dot{T}_\zeta$} if there exist a level $\bar{\al}<\al$ and \emph{distinct} nodes $\bar{\theta},\bar{\tau}$ on level $\bar{\al}$ so that $u\Vdash\bar{\theta}<_{\dot{T}_\zeta}\theta$ and $v\Vdash\bar{\tau}<_{\dot{T}_\zeta}\tau$.
 More generally, if $\zeta\leq\xi<\rho$ and $u',v' \in\R_\xi$,
 then we say that $u'$ and $v'$ split $\la\theta,\tau\ra$ below $\alpha$ in $\dot{T}_\zeta$ if $u = u'\res\zeta$ and $v = v'\res\zeta$ do.
\end{definition}

\begin{lemma}\label{lemma:1split} 
Suppose that $\zeta \leq \xi < \rho$ and $\R_\xi$ satisfies the dual residue property at some $M_\alpha$, where $\zeta\in M_\al$ (see Definition \ref{def:starhashtag}).
Fix $u,v\in\R_\xi$  so that for some $w\in\R_\xi\cap M_\al$, $\#^\xi_{\vp^{M_\al}}(u,v,w)$. Let $\la\theta,\tau\ra$ be a pair of tree nodes (possibly equal) 
each of which is at or above level $\al$.
Then there exist extensions $u^*\geq u$, $v^*\geq v$, and $w^*\geq w$ so that $\#^\xi_{\vp^{M_\al}}(u^*,v^*,w^*)$ and so that $u^*$ and $v^*$ split $\la\theta,\tau\ra$ below $\al$ in $\dot{T}_\zeta$.
\end{lemma}
\begin{proof} Since $\R_\xi$ satisfies the dual residue property at $M_\alpha$, $\R_\zeta$ does too, and so we may find some $w'\in\R_\zeta\cap M_\al$ so that $w'\geq_{\R_\zeta} w\res\zeta$ and  $*^\zeta_{\vp^{M_\al}}(u\res\zeta,v\res\zeta,w')$. Fix a $V$-generic filter $\bar{G}$ over $\R_\zeta\cap M_\al$ containing $w'$. As a result, $u\res\zeta$ and $v\res\zeta$ are in $(\R_\zeta/(p^*(M_\al),0_{\dot{\S}_\zeta}))/\bar{G}$. By the discussion preceding Definition \ref{def:splitbelow}, we know that $u\res\zeta$ forces in the quotient that $\dot{b}_\zeta(\theta,\al)$ is a cofinal branch through $\bar{T}:=(\dot{T}_\zeta\cap M_\al)[\bar{G}]$, which by Proposition \ref{proposition:ccreflect} is an Aronszajn tree on $\al$ in $V[\bar{G}]$. Consequently, we may find two conditions $u_0,u_1$ which extend $u\res\zeta$ in $(\R_\zeta/(p^*(M_\al),0_{\dot{\S}_\zeta}))/\bar{G}$, some level $\bar{\al}<\al$, and two \emph{distinct} nodes $\theta_0,\theta_1$ on level $\bar{\al}$ of $\bar{T}$ so that $u_i$ forces in $(\R_\zeta/(p^*(M_\al),0_{\dot{\S}_\zeta}))/\bar{G}$ that $\theta_i<_{\dot{T}_\zeta}\theta$. Since $v\res\zeta$ also forces that $\dot{b}_\zeta(\tau,\al)$ is a cofinal branch through $\bar{T}$, we may find some extension $v_0$ of $v\res\zeta$ in the quotient so that $v_0$ decides the $<_{\dot{T}_\zeta}$-predecessor, say $\bar{\tau}$, of $\tau$ on level $\bar{\al}$ of $\bar{T}$. As $\theta_0\neq\theta_1$, there exists some $i\in 2$ so that $\theta_i\neq\bar{\tau}$. Set $\bar{\theta}=\theta_i$.

Now fix an extension $w''$ of $w'$ in $\bar{G}$ so that $w''$ forces the following statements: (i) $u_i,v_0$ are in the quotient; (ii) $u_i$ forces in the quotient that $\bar{\theta}<_{\dot{T}_\zeta}\theta$; (iii) $v_0$ forces in the quotient that $\bar{\tau}<_{\dot{T}_\zeta}\tau$.

By two applications of Lemma \ref{lemma:quotientFUN}(3), we may find conditions $\bar{u},\bar{v}$ in the quotient 
so that $\bar{u}$ extends $u_i$ and $w''$, so that $\bar{v}$ extends $v_0$ and $w''$, and so that $\bar{u},\bar{v}\in D(\vp^{M_\al},\zeta)$.

We now see that $\bar{u}\Vdash_{\R_\zeta}\bar{\theta}<_{\dot{T}_\zeta}\theta$, since 
$$
w''\Vdash_{\R_\zeta\cap M_\al}\left(\;u_i\Vdash_{(\R_\zeta/(p^*(M_\al),0_{\dot{\S}_\zeta}))/\dot{\bar{G}}}\bar{\theta}<_{\dot{T}_\zeta}\theta\right)
$$
and since $\bar{u}\geq w'',u_i$. Similarly, $\bar{v}\Vdash_{\R_\zeta}\bar{\tau}<_{\dot{T}_\zeta}\tau$.

Finally, let $\bar{w}\geq w''$ be a condition in $\bar{G}$ which forces that $\bar{u}$ and $\bar{v}$ are in the quotient, so that by Lemma \ref{lemma:residueForceQuotient}, $*^\zeta_{\vp^{M_\al}}(\bar{u},\bar{v},\bar{w})$ holds. By Lemma \ref{lemma:starimplieshashtag}, we can find some $\bar{u}^*\geq\bar{u}$, $\bar{v}^*\geq\bar{v}$, and $\bar{w}^*\geq w$ so that $\#^\zeta_{\vp^{M_\al}}(\bar{u}^*,\bar{v}^*,\bar{w}^*)$ holds. Now let $u^*$ be the condition where $p_{u^*}=p_{\bar{u}^*}$, and where $f_{u^*}=f_{\bar{u}^*}\,^{\frown} f_u\res[\zeta,\xi)$. Let $v^*$ and $w^*$ be defined similarly. Then $\#^\xi_{\vp^{M_\al}}(u^*,v^*,w^*)$, and $u^*$ and $v^*$ split $\la\theta,\tau\ra$ below $\al$.

\end{proof}

One of the most important uses of the dual residue property is to obtain splitting pairs of conditions. Obtaining such conditions will also crucially use the ``exactness" conditions of Definition \ref{def:fksp}.

\begin{definition}\label{def:splittingpair/function} Suppose that $\vec{M}$ is in pre-splitting configuration up to $\rho$. 
\begin{enumerate}
    \item 
Let $\al\in\dom(\vec{M})$ and $\xi\in M_\al\cap\rho$. Fix a residue pair $\la p^*(M_\al),\vp^{M_\al}\ra$ for $(M_\al,\ps^*)$ and conditions $u,v\in\R_\xi$. We say that $u$ and $v$ are a \textbf{splitting pair for $(\vp^{M_\al},\xi)$} if
\begin{itemize}
\item for some $w\in\R_\xi\cap M_\al$, $\#^\xi_{\vp^{M_\al}}(u,v,w)$;
\item for any $\la\zeta,\nu\ra\in\dom(f_u)\cap\dom(f_v)\cap M_\al$ and any
$
\la\theta,\tau\ra\in (f_u(\zeta,\nu))\times (f_v(\zeta,\nu)),
$ 
both at or above level $\al$, $u$ and $v$ split $\la\theta,\tau\ra$ below $\al$ in $\dot{T}_\zeta$.
\end{itemize}
\item 
Given fixed enumerations $f_u(\zeta,\nu)\bsl(\al\times\om_1) = \{ \theta_n \mid n < \omega\}$ and $f_v(\zeta,\nu)\bsl(\al\times\om_1) = \{ \tau_m \mid m < \omega\}$ (possibly with repetitions in the case the sets are finite, nonempty), we define a \textbf{splitting function} to be a function $\Si$ with domain\footnote{Recall our convention from Notation \ref{notation:domain} regarding conditions $f_u$ and their domains.}
$$
\dom(\Si)=(\dom(f_u)\cap\dom(f_v)\cap M_\al)\times\om\times\om,
$$
so that for any $\la\zeta,\nu,m,n\ra\in\dom(\Si)$, $\Si(\zeta,\nu,m,n)$ is a pair $\la\bar{\theta},\bar{\tau}\ra$ of tree nodes satisfying Definition \ref{def:splitbelow} with respect to $\la\theta_m,\tau_n\ra$. We will denote $\bar{\theta}$ by $\Si(\zeta,\nu,m,n)(L)$ and $\bar{\tau}$ by $\Si(\zeta,\nu,m,n)(R)$.

\end{enumerate}
\end{definition}

\begin{remark}\label{remark:splittingfunction} Let $\Si$ be as in Definition \ref{def:splittingpair/function}.
\begin{enumerate}
\item We emphasize the fact that if $\la\zeta,\nu,m,n\ra\in\dom(\Si)$, then 
$$
\Si(\zeta,\nu,m,n)(L)\neq\Si(\zeta,\nu,m,n)(R)
$$ 
are two \emph{distinct} tree nodes \emph{on the same level}. We will usually suppress explicit mention of the level.
\item Any splitting function $\Si$ is a member of $M_\al$ since $M_\al$ is countably closed and since $\Si$ maps from a countable subset of $M_\al$ into $M_\al$.
\item We only require the splitting pair in Definition \ref{def:splittingpair/function} to split nodes coming from coordinates which are members of $M_\al$. As we will see from Lemma \ref{lemma:amalgamatesplittingpair} and later thinning out arguments, this is sufficient for our purposes.
\end{enumerate}
\end{remark}

Now we are ready to state our second inductive hypothesis, the point of which is to provide plenty of instances of the dual residue property. We will assume the second inductive hypothesis for the rest of the section. {We again recall the filter $\cal{F}$ and its dual ideal $\cal{I}$ from Definition \ref{def:fksp}.}
\vspace{3mm}

\textbf{Inductive Hypothesis II}: Let $\xi<\rho$, and suppose that $\vec{M}$ is $\R_\xi$-suitable. Then {there is an $A\seq\dom(\vec{M})$ with $A\in\cal{I}$ so that for all $\al\in\dom(\vec{M})\bsl A$, $\R_\xi$ satisfies the dual residue property at $M_\al$.}\\

Now we move to the main part of the proof that $\ps^*$ satisfies inductive hypotheses I and II with respect to $\rho$. After a bit more set-up, we will verify inductive hypothesis I for $\rho$ and then use this
to verify inductive hypothesis II at $\rho$. 

{The following lemma amalgamates instance of the second induction hypothesis below $\rho$, stating that if $\vec{M}$ is $\rho$-suitable, then for almost all $\al\in\dom(\vec{M})$ and all $\xi\in M_\al\cap\rho$, $\R_\xi$ satisfies the dual residue property at $M_\al$. However, we note that the following lemma is far from showing that the second induction hypothesis holds at $\rho$ itself, only asserting (roughly) that it holds up to $\rho$. The proof involves a standard diagonal union, using the normality of $\cal{I}$. }

\begin{lemma}\label{lemma:DRPBelow} {Suppose that $\vec{M}$ is $\R_\rho$-suitable. Then there is an $A\seq\dom(\vec{M})$ with $A\in\cal{I}$ so that for all $\al\in\dom(\vec{M})\bsl A$ and all $\xi\in M_\al\cap\rho$, $\R_\xi$ satisfies the dual residue property at $M_\al$.}
\end{lemma}
\begin{proof}
{Fix $\vec{M}$ which is $\R_\rho$-suitable. If $\rho<\ka$, then the lemma follows by taking the union of $<\ka$-many sets in $\cal{I}$ which witness the second induction hypothesis below $\rho$. Suppose, then, that $\rho\geq\ka$, and let $h:\ka\lra\rho$ be the $\lhd$-least bijection from $\ka$ onto $\rho$. Note that $h$ is in $M_\al$ for all $\al\in\dom(\vec{M})$  (since $\rho$ is an element of $M_\al$) and that for each such $\al$, $M_\al\cap\rho=h[M_\al\cap\ka]$. Next observe that for all $\xi<\rho$, a tail of the sequence $\vec{M}$ is $\R_\xi$-suitable, since a tail of this sequence contains $\xi$ as an element and since each model on $\vec{M}$ is $\R_\rho$-suitable.}

{For each $\xi<\rho$, we may then find $A_\xi\in\cal{I}$ so that for all $\al\in\dom(\vec{M})\bsl A_\xi$, $M_\al$ is $\R_\xi$-suitable and so that $\R_\xi$ satisfies the dual residue property at $M_\al$. For each $\nu<\ka$, let $B_\nu:=A_{h(\nu)}$, and let $B:=\nabla_{\nu<\ka}B_\nu=\lb\be<\ka:(\exists\nu<\be)\,[\be\in B_\nu]\rb$. $B\in\cal{I}$ since $\cal{I}$ is a normal ideal.}

{We now claim that for all $\al\in\dom(\vec{M})$ and all $\xi\in M_\al\cap\rho$, $M_\al$ is $\R_\xi$-suitable and $\R_\xi$ satisfies the dual residue property at $M_\al$. So let such $\al$ and $\xi$ be given. $\xi\in M_\al\cap\rho$, and hence $\xi=h(\bar{\nu})$ for some $\bar{\nu}<\al$. However, $\al\notin B$, and therefore for all $\nu<\al$, $\al\notin B_\nu=A_{h(\nu)}$. In particular, $\al\notin A_{h(\bar{\nu})}=A_\xi$. Thus by choice of $A_\xi$, $M_\al$ is $\R_\xi$-suitable, and $\R_\xi$ satisfies the dual residue property at $M_\al$.}
\end{proof}

At important parts of the following proofs we will need to understand the circumstances under which we can amalgamate conditions in $\R_\rho$, and in particular, in $\dot{\S}_\rho.$ We will often be interested in the following strong sense of amalgamation:

\begin{definition}
\label{def:stronglycompatibleover} Let $u,v\in\R_\rho$ so that $p_u$ and $p_v$ are compatible in $\ps^*$. We say that $f_u$ and $f_v$ are \textbf{strongly compatible over $p_u$ and $p_v$} if for any condition $q\in\ps^*$ which extends $p_u$ and $p_v$, $q$ forces that $\check{f}_u\cup \check{f}_v\in\dot{\S}_\rho$.
\end{definition}

The next lemma gives sufficient conditions under which we may amalgamate conditions in $\R_\rho$ whose specializing parts are strongly compatible as above.

\begin{lemma}\label{lemma:amalgamatesplittingpair} Suppose that $\vec{M}$ is in pre-splitting configuration up to $\rho$, that $\al<\be$ are in $\dom(\vec{M})$, and that $\la p^*(M_\al),\vp^{M_\al}\ra$ and $\la p^*(M_\be),\vp^{M_\be}\ra$ are residue pairs for $(M_\al,\ps^*)$ and $(M_\be,\ps^*)$, respectively. Let $\la u_\al,v_\al\ra$ and $\la u_\be,v_\be\ra$ be two pairs of conditions in $\R_\rho$ which satisfy the following:
\begin{enumerate}
\item $\la u_\al,v_\al\ra$ is a splitting pair for $(\vp^{M_\al},\rho)$ with splitting function $\Si_\al$;
\item $\la u_\be,v_\be\ra$ is a splitting pair for $(\vp^{M_\be},\rho)$ with splitting function $\Si_\be$;
\item $\Si_\al=\Si_\be$;
\item there exists $w\in M_\al$ so that $\#^\rho_{\vp^{M_\al}}(u_\al,v_\al,w)$ and $\#^\rho_{\vp^{M_\be}}(u_\be,v_\be,w)$ both hold; and
\item $u_\al,v_\al\in M_\be$.
 \end{enumerate}
Then $u_\al$ and $v_\be$ are compatible in $\R_\rho$; in fact, $f_{u_\al}$ is strongly compatible with $f_{v_\be}$ over $p_{u_\al}$ and $p_{v_\be}$.
\end{lemma}
\begin{proof} We first observe that $p_{u_\al}$ and $p_{v_\be}$ are compatible in $\ps^*$. Indeed, by (4), $\vp^{M_\al}(p_{u_\al})=^* p_w=^*\vp^{M_\be}(p_{v_\be})$, and by (5), $p_{u_\al}\in M_\be$. Thus as $p_{u_\al}\geq\vp^{M_\al}(p_{u_\al})\geq\vp^{M_\be}(p_{v_\be})$, and as $\vp^{M_\be}$ is a residue function, $p_{u_\al}$ is compatible with $p_{v_\be}$.

Now let $q\in\ps^*$ be any common extension of $p_{u_\al}$ and $p_{v_\be}$. We will argue by induction on $\zeta\leq\rho$ that $q\Vdash (\check{f}_{u_\al}\cup \check{f}_{v_\be})\res\zeta\in\dot{\S}_\zeta$. Limit stages are immediate. For the successor stage, suppose that $\la\zeta,\nu\ra\in\dom(f_{u_\al})\cap\dom(f_{v_\be})$ and that we have proven that $q\Vdash (\check{f}_{u_\al}\cup \check{f}_{v_\be})\res\zeta\in\dot{\S}_\zeta$. Since $f_{u_\al}\in M_\be$ by (5), $\la\zeta,\nu\ra\in M_\be$. Thus $\la\zeta,\nu\ra\in\dom(f_{v_\be})\cap M_\be=\dom(f_w)$, since $w=^*v_\be\res M_\be$.
Since we also have that $w=^*u_\be\res M_\be$, it follows that $\la\zeta,\nu\ra\in\dom(f_{u_\be})$. Thus $\la\zeta,\nu\ra\in\dom(f_{u_\be})\cap\dom(f_{v_\be})\cap M_\be=\dom(f_{u_\al})\cap\dom(f_{v_\al})\cap M_\al$, with equality holding by (3) and the definition of a splitting function. Moreover, $\zeta\in M_\al$ since $\la\zeta,\nu\ra\in\dom(f_w)\seq M_\al$.

Now pick a pair of distinct nodes $\la\theta,\tau\ra\in f_{u_\al}(\zeta,\nu)\times f_{v_\be}(\zeta,\nu)$, and we will show that $(q,(f_{u_\al}\cup f_{v_\be})\res\zeta)$ forces in $\R_\zeta$ that $\theta$ and $\tau$ are $\dot{T}_\zeta$-incompatible. If $\theta$ is below level $\al$, then $\theta\in (f_{u_\al}\res M_\al)(\zeta,\nu)=f_w(\zeta,\nu)\seq f_{v_\be}(\zeta,\nu)$. Thus $(q,f_{v_\be}\res\zeta)\Vdash\theta,\tau$ are $\dot{T}_\zeta$-incompatible, and so $(q,(f_{u_\al}\cup f_{v_\be})\res\zeta)$ forces this too. A similar argument applies if $\tau$ is below level $\be$.

We therefore assume that $\theta$ is at or above level $\al$ and $\tau$ is at or above level $\be$. Let $m$ and $n$ be chosen so that $\theta$ is the $m$th node in $f_{u_\al}(\zeta,\nu)\bsl(\al\times\omega_1)$ and $\tau$ is the $n$th node in $f_{v_\be}(\zeta,\nu)\bsl(\be\times\om_1)$. By assumption (3), letting $\Si:=\Si_\al=\Si_\be$, we know that $\Sigma(\zeta,\nu,m,n)(L)$ and $\Sigma(\zeta,\nu,m,n)(R)$ are two distinct nodes on the same level and also that
$$
(p_{u_\al},f_{u_\al}\res\zeta)\Vdash\Sigma(\zeta,\nu,m,n)(L)<_{\dot{T}_\zeta}\theta
$$
and
$$
(p_{v_\be},f_{v_\be}\res\zeta)\Vdash\Sigma(\zeta,\nu,m,n)(R)<_{\dot{T}_\zeta}\tau.
$$
Therefore $(q,(f_{u_\al}\cup f_{v_\be})\res\zeta)$ forces that $\tau$ and $\theta$ are incompatible in $\dot{T}_\zeta$, as we intended to show.
\end{proof}

The following item shows how we can obtain the desired splitting pairs of conditions.

\begin{lemma}\label{lemma:splittingpairs}  Suppose that $\vec{M}$ is in pre-splitting configuration up to $\rho$ {and that $\dom(\vec{M})$ satisfies the conclusion of Lemma \ref{lemma:DRPBelow}.} Fix $\al\in\dom(\vec{M})$, and suppose that $\la p^*(M_\al),\vp^{M_\al}\ra$ is a residue pair for $(M_\al,\ps^*)$. Finally, fix $u,v,w$ so that $\#^\rho_{\vp^{M_\al}}(u,v,w)$. Then there exist extensions $u^*\geq u$ and $v^*\geq v$ so that $u^*,v^*$ are a splitting pair for $(\vp^{M_\al},\rho)$.
\end{lemma}
\begin{proof} Fix $u,v,w$ as in the statement of the lemma. We define a {coordinate-wise} increasing sequence of triples $\la\la u_n,v_n,w_n\ra:n\in\om\ra$ of conditions and a sequence $\la\la\zeta_n,\nu_n,\theta_n,\tau_n\ra:n\in\om\ra$ of tuples of ordinals and tree nodes so that $\la u_0,v_0,w_0\ra=\la u,v,w\ra$ and so that for each $n$, 
\begin{itemize}
\item $\#^\rho_{\vp^{M_\al}}(u_n,v_n,w_n)$ holds;
\item $\la\zeta_n,\nu_n\ra\in\dom(f_{u_n})\cap\dom(f_{v_n})\cap M_\al$ and $\la\theta_n,\tau_n\ra\in (f_{u_n}(\zeta_n,\nu_n)\bsl(\al\times\om_1))\times(f_{v_n}(\zeta_n,\nu_n)\bsl(\al\times\om_1))$; and
\item $u_{n+1}$ and $v_{n+1}$ split $\la\theta_n,\tau_n\ra$ below $\al$.
\end{itemize}
This is done with respect to some bookkeeping device in such a way that if $u^*$ is a sup of $\la u_n:n\in\om\ra$ (and similarly for $v^*$), then for each $\la\zeta,\nu\ra\in\dom(f_{u^*})\cap\dom(f_{v^*})\cap M_\al$ and each $\la\theta,\tau\ra\in(f_{u^*}(\zeta,\nu)\bsl(\al\times\om_1))\times(f_{v^*}(\zeta,\nu)\bsl(\al\times\om_1))$, $\la\zeta,\nu,\theta,\tau\ra$ appears as the $n$th tuple for some $n$. 

To show the successor step, suppose that $u_n,v_n$ and $w_n$ are given, and consider $\la\zeta_n,\nu_n,\theta_n,\tau_n\ra$. Note that $\#^{\zeta_n}_{\vp^{M_\al}}(u_n\res\zeta_n,v_n\res\zeta_n,w_n\res\zeta_n)$ also holds. {Then Lemma \ref{lemma:1split} applies since $\zeta_n\in M_\al$ and since $\R_\eta$ satisfies the dual residue property at $M_\al$ for all $\eta\in M_\al\cap\rho$.} Thus we may find conditions $u'_n\geq u_n\res\zeta_n$, $v'_n\geq v_n\res\zeta_n$, and $w'_n\geq w_n\res\zeta_n$ so that $\#^{\zeta_n}_{\vp^{M_\al}}(u'_n,v'_n,w'_n)$ and so that $u'_n$ and $v'_n$ split $\la\theta_n,\tau_n\ra$ below $\al$. Now define $f_{u_{n+1}}$ to be the function {which equals $f_{u'_n}$ on $\zeta_n$ and which equals $f_{u_n}$ on $[\zeta_n,\rho)$.} Also, let $u_{n+1}$ be the pair $(p_{u'_{n}},f_{u_{n+1}})$. Let $v_{n+1}$ and $w_{n+1}$ be defined similarly. Then $\#^\rho_{\vp^{M_\al}}(u_{n+1},v_{n+1},w_{n+1})$ holds and $u_{n+1}$ and $v_{n+1}$ split $\la\theta_n,\tau_n\ra$ below $\al$.

This completes the construction of the sequence. Fix sups $u^*,v^*,w^*$. Since $\#^\rho_{\vp^{M_\al}}(u_n,v_n,w_n)$ holds for all $n$, $\#^\rho_{\vp^{M_\al}}(u^*,v^*,w^*)$ also holds. By the choice of bookkeeping, $u^*,v^*$ is a splitting pair for $(\vp^{M_\al},\rho)$, completing the proof.
\end{proof}


\begin{lemma}\label{lemma:crosscompatible} Suppose that $\vec{M}$ is in pre-splitting configuration up to $\rho$. Suppose that for each $\al\in \dom(\vec{M})$, there exist $u_\al,v_\al$ which are a splitting pair for $(\vp^{M_\al},\rho)$, where $\la p^*(M_\al),\vp^{M_\al}\ra$ is a residue pair for $(M_\al,\ps^*)$. {Then there exists $B\seq\dom(\vec{M})$ in $\cal{F}^+$ so that} for any $\al<\be$ in $B$, $u_\al,v_\al\in M_\be$, $u_\al\res M_\al=^*v_\be\res M_\be$, and $u_\al$ is compatible with $v_\be$.
\end{lemma}
\begin{proof}
Suppose that for each {$\al\in \dom(\vec{M})$}, we have a splitting pair $u_\al,v_\al$ for $(\vp^{M_\al},\rho)$; we also let $w_\al\in \R_\rho\cap M_\al$ be a condition witnessing $\#^\rho_{\vp^{M_\al}}(u_\al,v_\al,w_\al)$. Let $\Si_\al$ be a splitting function for $(u_\al,v_\al)$ with respect to $M_\al$, as in Definition \ref{def:splittingpair/function}. By Remark \ref{remark:splittingfunction}, $\Si_\al\in M_\al$. Now the function on $\dom(\vec{M})$ defined by $\al\mapsto \la w_\al,\Si_\al\ra$ is regressive (since the pair can be coded by an ordinal below $\al$). Since $\dom(\vec{M})\in{\cal{F}^+}$ and ${\cal{F}}$ is normal, there exists some ${B}\seq\dom(\vec{M})$ which is also in ${\cal{F}^+}$ on which that function takes a constant value, say $\la\bar{w},\Si\ra$. Moreover, by intersecting with a club and relabelling if necessary, we may assume that if $\al<\be$ are in ${B}$, then $u_\al,v_\al\in M_\be$. But then for any  $\al<\be$ in {$B$}, we have that $u_\al\res M_\al=^*\bar{w}=^*v_\be\res M_\be$. Therefore, for all $\al<\be$ in ${B}$, the assumptions of Lemma \ref{lemma:amalgamatesplittingpair} are satisfied, and consequently $u_\al$ and $v_\be$ are compatible.
\end{proof}

\begin{proposition}\label{prop:kappacc} $\ps^*\Vdash\dot{\S}_\rho$ is $\ka$-c.c.
\end{proposition}
\begin{proof} Let $p\in\ps^*$ be a condition, and suppose that $p\Vdash\la\dot{f}_\ga:\ga<\ka\ra$ is a sequence of conditions in $\dot{\S}_\rho$. We will find some extension $p^*$ of $p$ which forces that this sequence does not enumerate an antichain.

Let $\vec{M}$ be a sequence which is suitable with respect to the three parameters $\R_\rho$, $p$ and $\la\dot{f}_\ga:\ga<\ka\ra$, and which is in pre-splitting configuration up to $\rho$. {By removing an $\cal{I}$-null set, we may assume that $\dom(\vec{M})$ satisfies the conclusion of Lemma \ref{lemma:DRPBelow}.}

{Let $B:=\dom(\vec{M})$. Since $\vec{M}$ is in pre-splitting configuration up to $\rho$, let $\la\la p^*(M_\al),\vp^{M_\al}\ra:\al\in B\ra$ be a residue system.} For each $\al\in B$, $p\in \ps^*\cap M_\al$, and therefore we may find some extension $p_\al$ of $p$ so that $p_\al\in \dom(\vp^{M_\al})$. We may also assume that for some function $f_\al$ in $V$, $p_\al\Vdash_{\ps^*}\dot{f}_\al=\check{f}_\al$. Now extend $\la p_\al,f_\al\ra$ to a condition $u_\al$ in $D(\vp^{M_\al},\rho)$. By Lemma \ref{lemma:splittingpairs}, we may further extend $u_\al$ to a splitting pair $\la u^*_\al,v^*_\al\ra$ for $(\vp^{M_\al},\rho)$. By Lemma \ref{lemma:crosscompatible}, we may find some $B^*\seq B$ with $B^*\in{\cal{F}^+}$ so that for all $\al<\be$ in $B^{*}$, $u^*_\al$ and $v^*_\be$ are compatible. Let $w$ be a condition extending them both. Then $p_w$ forces that $\check{f}_w$ extends both $\check{f}_{u^*_\al}$ and $\check{f}_{v^*_\be}$ and hence extends $\dot{f}_\al$ and $\dot{f}_\be$. Therefore $p_w$ forces that $\dot{f}_\al$ and $\dot{f}_\be$ are compatible in $\dot{\S}_\rho$.
\end{proof}

{We are now ready to verify that the second induction hypothesis holds at $\rho$. We again remark that this proposition (and the later results which build off of it)  is the only place in our work where we need the ineffability of $\ka$. In all other cases, the weak compactness of $\ka$ suffices.}

\begin{proposition}\label{prop:InductiveIandII} Suppose that $\vec{M}$ is in pre-splitting configuration up to $\rho$. {Then there is an $A\in\cal{I}$ so that for all $\al\in\dom(\vec{M})\bsl A$, $\R_\rho$ satisfies the dual residue property at $M_\al$. }
\end{proposition}
\begin{proof} {We only deal with the case when $\ps^*$ is not just the collapse poset $\ps$ (and hence we're in the case where $\ka$ is ineffable, and $\cal{F}=\cal{F}_{in}$). The case when $\ps^*$ is the collapse $\ps$ is simpler and taken care of in \cite{LS}.}

Suppose otherwise, for a contradiction. Then
$$
B:=\lb\al\in\dom(\vec{M}):\R_\rho \text{ does not satisfy the dual residue property at }M_\al\rb
$$
is in ${\cal{F}^+}$. {Moreover, by removing an $\cal{I}$-null set if necessary, we may assume that $B$ satisfies the conclusion of Lemma \ref{lemma:DRPBelow}. We will derive our contradiction by creating a $\ka$-sized antichain in $\R_\rho$ for which we can amalgamate many of the $\ps^*$-parts. This will then lead to a $\ka$-sized antichain in $\S_\rho$ in some $V$-generic extension over $\ps^*$.}

For each $\al\in B$, we fix a residue pair $\la p^*(M_\al),\vp^{M_\al}\ra$ for $(M_\al,\ps^*)$ and a triple $\la u_\al,v_\al,w_\al\ra$ which witnesses that $\R_\rho$  does not satisfy the dual residue property at $M_\al$. Thus $\#^\rho_{\vp^{M_\al}}(u_\al,v_\al,w_\al)$ holds, but for any $w^*\geq_{\R_\rho\cap M_\al}w_\al$, either $w^*$ is not a residue for $u_\al$ to $M_\al$ or $w^*$ is not a residue for $v_\al$ to $M_\al$. In particular, for each such $w^*$, we may find a further extension in $\R_\rho\cap M_\al$ which is either incompatible with $u_\al$ or incompatible with $v_\al$ in $\R_\rho$.

By Lemma \ref{lemma:splittingpairs}, we may extend $\la u_\al,v_\al,w_\al\ra$ to another triple $\la u^*_\al,v^*_\al,w^*_\al\ra$ so that $u^*_\al$ and $v^*_\al$ are a splitting pair for $M_\al$. By Lemma \ref{lemma:crosscompatible}, we may find some $B^*\seq B$ with $B^*\in\cal{F}^+$ so that for all $\al,\be\in B^*$ with $\al<\be$, $w^*_\al=^*w^*_\be$, $u^*_\al$ and $v^*_\al$ are in $M_\be$, and $u_\al^*$ is compatible with $v_\be^*$. We let $\bar{w}^*$ denote a condition which is $=^*$ equal to $w^*_\al$ for $\al\in B^*$.

{Next, for each $\al<\be$ both in $B^*$, we define a condition $w^*_{\al,\be}$. Fix such $\al$ and $\be$. Since $u^*_\al\in M_\be$ is compatible with $v^*_\be$, there is an extension $w^*_{\al,\be}$ of $u^*_\al$ in $\R_\rho\cap M_\be$ which is a residue for $v_\be^*$ to $\R_\rho\cap M_\be$. Since $\be\in B$ and since $w^*_{\al,\be}$ is a residue for $v^*_\be$, we may further extend (and relabel if necessary) to assume that $w^*_{\al,\be}$ is incompatible with $u^*_\be$. Since $p_{w^*_{\al,\be}}\geq p_{u^*_\al}\geq p^*(M_\al)$, we may further assume that $p_{w^*_{\al,\be}}$ is in the domain of $\vp^{M_\al}$.}

{We now set up an application of the ineffability of $\ka$. For each $\be\in B^*$, we define the function $E_\be$ by}
$$
{E_\be:=\lb (\al,w^*_{\al,\be}):\al\in B^*\cap\be\rb\seq\be\times (M_\be\cap\R_\rho).}
$$
{Formally, we ought to apply the ineffability of $\ka$ to a sequence $\vec{A}$ where the $\al$-th element on the sequence is a subset of $\al$. However, we will work with the sets $E_\be$; this poses no loss of generality since, by using the $\lhd$-least bijection from $\R_\rho$ onto $\ka$ and the G{\"o}del pairing function, we can code $E_\be$ as a subset of $\be$.}

{Since $\ka$ is ineffable and $B^*\in\cal{F}^+$, we can find a subset $E$ of $\ka\times\R_\rho$ and a stationary $S\seq B^*$ so that for all $\be\in S$, $E\cap(\be\times(M_\be\cap\R_\rho))=E_\be$. We observe that $E$ is a function: if $(\al,w)$ and $(\al,w')$ are both in $E$, fix some $\be\in S$ large enough so that $w,w'\in M_\be\cap\R_\rho$. Then $(\al,w)$ and $(\al,w')$ are in $E\cap(\be\times(M_\be\cap\R_\rho))=E_\be$. Since $E_\be$ is a function, $w=w'$. We can now rephrase the coherence as follows: if $\be\in S$, then $E\res\be=E_\be$, since $E\res\be$ and $E_\be$ are both functions with domain $B^*\cap\be$ and $E_\be\seq E\res\be$.}

{Next, $B^*\seq\dom(E)$. Indeed, for each $\be\in S$, $E\res\be=E_\be$, and the domain of $E_\be$ is $B^*\cap\be$. Since $S$ is unbounded (in fact stationary) in $\ka$, there are unboundedly-many $\be$ so that $E\res\be=E_\be$, from which the conclusion follows. And finally, if $\al\in B^*$ then for any $\be\in S\bsl (\al+1)$, $E(\al)=w^*_{\al,\be}$, since $E(\al)=E_\be(\al)=w^*_{\al,\be}$.}

{Now we press down residues for conditions indexed by $S$. Since $S\seq B^*=\dom(E)$, $E(\al)$ is defined for each $\al\in S$. And moreover, $p_{E(\al)}$ is in the domain of $\vp^{M_\al}$ since it equals $w^*_{\al,\be}$ for some/any $\be\in S\bsl(\al+1)$, and since $w^*_{\al,\be}$ is in the domain of $\vp^{M_\al}$. Then $\vp^{M_\al}(p_{E(\al)})$ is a condition in $M_\al\cap\R_\rho$, and each such condition can be coded by an element of $\al$, using the $\lhd$-least bijection from $\R_\rho$ onto $\ka$. Thus the function $\al\mapsto\vp^{M_\al}(p_{E(\al)})$ on $S$ is regressive, and so we can find a stationary $S^*\seq S$ so that it has a constant value, say the condition $p^{**}$.}

{Now fix $\al<\be$ in $S^*$, and we will show that $p_{E(\al)}$ and $p_{E(\be)}$ are compatible in $\ps^*$. Indeed, $E(\al)=w^*_{\al,\be}$ is an element of $M_\be\cap\R_\rho$. Additionally, $p_{w^*_{\al,\be}}\geq\vp^{M_\al}(p_{w^*_{\al,\be}})=p^{**}=\vp^{M_\be}(p_{E(\be)})$. Thus $p_{E(\al)}=p_{w^*_{\al,\be}}$ extends, inside of $M_\be$, the residue of $p_{E(\be)}$ to $M_\be$. $p_{E(\al)}$ and $p_{E(\be)}$ are therefore compatible in $\ps^*$.}

{However, for such $\al<\be$, we also know that $E(\al)$ and $E(\be)$ are incompatible conditions in $\R_\rho$, since $E(\be)$ extends $u^*_\be$ and since $E(\al)=w^*_{\al,\be}$ is incompatible with $u^*_\be$. Thus if $q$ is any condition which extends $p_{E(\al)}$ and $p_{E(\be)}$ in $\ps^*$, then $q$ must force that $f_{E(\al)}$ and $f_{E(\be)}$ are incompatible conditions in $\dot{\S}_\rho$.}

{Now we can create our $\ka$-sized antichain of specializing conditions. Let $G^*$ be a $V$-generic filter over $\ps^*$ which contains the condition $p^{**}$, and recall that $p^{**}=\vp^{M_\al}(p_{E(\al)})$ for all $\al\in S^*$. By Lemma \ref{lemma:unbounded}, the set }
$$
{X:=\lb\al\in S^*:p_{E(\al)}\in G^*\rb}
$$ 
{is unbounded in $\ka$. Therefore if $\al<\be$ are in $X$, then $f_{E(\al)}$ and $f_{E(\be)}$ are incompatible conditions in $\dot{\S}_\rho[G]$. Since $\ka$ is a cardinal after forcing with $\ps^*$ and since $X$ has size $\ka$, this gives a $\ka$-sized antichain in $\dot{\S}_\rho[G]$. This contradicts Proposition \ref{prop:kappacc} and completes the proof.}
\end{proof}

{Here we comment on the use of the ineffability of $\ka$. In the original Laver-Shelah argument, $\ps^*$ is just the collapse forcing. Thus their entire forcing is $\ka$-c.c. However, in our set-up, $\ps^*$ will in general fail to be $\ka$-c.c., and consequently it is not enough to find a $\ka$-sized antichain in $\ps^*\ast\dot{\S}_\rho=\R_\rho$. Rather, we need to arrange that there is a $\ka$-sized antichain in $\R_\rho$ for which we can amalgamate plenty of the $\ps^*$-parts of the conditions. This in turn requires that we be able to press down on the residues of the $\ps^*$-parts. Considering the  array $\la w^*_{\al,\be}:\al,\be\in B^*\we\al<\be\ra$ from the proof of the previous result, we need to find a stationary $S\seq B^*$ on which, for each $\al\in S$, the function on $S\bsl(\al+1)$ taking $\be$ to $w^*_{\al,\be}$ is independent of $\be$, say taking value $w^{**}_\al$. Then using the stationarity of $S$, we pressed down on the residue of $w^{**}_\al$.  The ineffability of $\ka$ allowed us to create a function, namely $E$, out of the above array with $\dom(E)$ containing a stationary set on which the approximations (the $E_\be$) cohere. We were not able to create this function and set up an application of pressing down just assuming that $\ka$ is weakly compact.}
 
{However, in the case that $\ps^*$ is just the collapse, then a weakly compact cardinal suffices for the entirety of the argument, since the entire poset $\ps\ast\dot{\S}_\rho$ is then $\ka$-c.c. In this case, we only need to create an unbounded $Z\seq B^*$ on which the function $\be\mapsto w^*_{\al,\be}$ is independent of $\be$, for each $\al\in Z$ and $\be\in Z\bsl(\al+1)$; this is because $\la w^{**}_\al:\al\in Z\ra$ would then be a $\ka$-sized antichain in $\ps\ast\dot{\S}_\rho$, a contradiction. $Z$ can be constructed by working inside a $\ka$-model $M^*$ containing all of the relevant information, for which there exists an $M^*$-normal ultrafilter containing $B$ as an element.}\\
 

We have now completed the proof of Theorem \ref{thmSkappacc}. We conclude with a corollary which adds to that theorem an additional clause about the dual residue property; this will be useful later.

\begin{corollary}\label{cor:IH1and2}
Suppose that $\ps^*$ is ${\cal{F}}$-strongly proper and that $\dot{\S}_{\ka^+}$ is a $\ps^*$-name for a  $\ka^+$-length, countable support iteration specializing Aronszajn trees on $\ka$. Then for all $\rho<\ka^+$,
\begin{enumerate}
    \item $\ps^*$ forces that $\dot{\S}_\rho$ is $\ka$-c.c.; and
    \item if $\vec{M}$ is in pre-splitting configuration up to $\rho$, then {there is some $A\seq\dom(\vec{M})$ with $A\in\cal{I}$ so that for all $\al\in \dom(\vec{M})\bsl A$, $\R_\rho$ satisfies the dual residue property at $M_\al$. Hence, for all $\zeta\in M_\al\cap(\rho+1)$, $\R_\zeta$ satisfies the dual residue property at $M_\al$.}
\end{enumerate}
\end{corollary}
\begin{proof}
If the corollary is false, let $\rho$ be the least such that it fails at $\rho$. Then Induction Hypotheses I and II hold below $\rho$, so Propositions \ref{prop:kappacc} and \ref{prop:InductiveIandII} show that (1) and (2) hold at $\rho$, a contradiction.
\end{proof}

\section{${\cal{F}}$-Strongly Proper Posets and Preserving Stationary Sets}\label{section:stationarypreservation}

In this section, we will prove that the appropriate quotients preserve stationary sets of cofinality $\om$ ordinals. We will apply this result in Section \ref{section:clubscompletelyproper} when we show that our intended club-adding iteration is ${\cal{F}}$-completely proper (see Definition \ref{def:fkcp}). In the first part of this section, we will prove some helpful lemmas which we use in the second part to complete proof of the preservation of the relevant stationary sets.

For the remainder of this section, we fix a ${\cal{F}}$-strongly proper poset $\ps^*$ and an iteration $\dot{\bb{S}}_\rho$ of length $\rho<\ka^+$ specializing Aronszajn trees in the extension by $\ps^*$; see the beginning of Section \ref{section:specialize} for a more precise definition and relevant notation. Note that the conclusions of Corollary \ref{cor:IH1and2} hold.

We first prove two lemmas which describe how the residue functions with respect to two models on a suitable sequence interact. More precisely, suppose we have a suitable sequence $\vec{M}$, where $\al<\be$ are both in $\dom(\vec{M})$ and $M_\al$ and $M_\be$ have respective residue pairs $\la p^*(M_\al),\vp^{M_\al}\ra$ and $\la p^*(M_\be),\vp^{M_\be}\ra$ with respect to $\ps^*$. A natural question is whether, on a dense set, $\vp^{M_\al}(\vp^{M_\be}(q))=^*\vp^{M_\al}(q)$, i.e., whether the $M_\al$-residue of the $M_\be$-residue is equivalent to the $M_\al$-residue. Proposition \ref{prop:twomasterconditions} below shows that this is the case.

\begin{lemma}\label{lemma:firstresiduelemma} Suppose that $\vec{M}$ is $\ps^*$-suitable with residue system 
$$
\la \la p^*(M_\ga),\vp^{M_\ga}\ra:\ga\in\dom(\vec{M})\ra
$$
and that $\al<\be$ are in $\dom(\vec{M})$. 
\begin{enumerate}
    \item For every  $p\in\ps^*$ that extends both $p^*(M_\al)$ and $p^*(M_\be)$, there is an extension $p^*\geq_{\ps^*}p$ with $p^*\in \dom(\vp^{M_\al})\cap \dom(\vp^{M_\be})$.
    \item $D_0(\vp^{M_\al},\vp^{M_\be}):=\lb q\in \dom(\vp^{M_\al})\cap \dom(\vp^{M_\be}):\vp^{M_\be}(q)\in \dom(\vp^{M_\al})\rb$ is $=^*$-countably closed and dense in $\ps^*/\lb p^*(M_\al),p^*(M_\be) \rb$.\footnote{This denotes the set of $r\in\ps^*$ which extend both $p^*(M_\al)$ and $p^*(M_\be)$.} 
\end{enumerate}
\end{lemma}
\begin{proof}
For (1), we apply a dovetailing construction using the properties of the residue functions. Define, by recursion, an increasing sequence $\la p_n:n\in\om\ra$ of extensions of $p$ so that $p_{2n+1}\in\dom(\vp^{M_\al})$ and, for $n>0$, $p_{2n}\in\dom(\vp^{M_\be})$. Let $p^*$ be a sup of $\la p_n:n\in\om\ra$. Then $p^*\in\dom(\vp^{M_\al})$ since it is also a sup of $\la p_{2n+1}:n\in\om\ra$, and $p^*\in\dom(\vp^{M_\be})$ since it is a sup of $\la p_{2n}:n>0\ra$.

For (2), fix a condition $q_{-1}\in\ps^*$ which extends both $p^*(M_\al)$ and $p^*(M_\be)$, where {by (1)} we may assume that $q_{-1}\in\dom(\vp^{M_\al})\cap\dom(\vp^{M_\be})$. We first make a cosmetic improvement to $q_{-1}$ before the main construction. Since $q_{-1}$ extends both $\vp^{M_\be}(q_{-1})$ and $p^*(M_\al)$ and since both of these conditions are in $M_\be$ (using Definition \ref{def:fksp}(2) to see that $p^*(M_\al)\in M_\be$), we may apply the elementarity of $M_\be$ to find a condition $s_{-1}\in M_\be$ which extends $\vp^{M_\be}(q_{-1})$ and $p^*(M_\al)$. Now find an extension $q_0\geq q_{-1}$ so that $\vp^{M_\be}(q_0)\geq s_{-1}$, noting that we may assume that $q_0\in\dom(\vp^{M_\al})\cap\dom(\vp^{M_\be})$.

Having completed this modification, we now define, by recursion, an increasing sequence of conditions $\la q_n:n\in\om\ra$ in $\dom(\vp^{M_\al})\cap\dom(\vp^{M_\be})$ and an increasing sequence $\la s_n:n\in\om\ra$ of conditions in $M_\be\cap\dom(\vp^{M_\al})$ so that for each $n$, $\vp^{M_\be}(q_{n+1})\geq s_n\geq\vp^{M_\be}(q_n)$. So assume that $q_n$ is defined. Since $\vp^{M_\be}(q_n)\geq p^*(M_\al)$ {(using the previous paragraph for the case $n=0$)}, we may find an extension $s_n$ of $\vp^{M_\be}(q_n)$ which is a member of $M_\be\cap\dom(\vp^{M_\al})$. Then let $q_{n+1}\geq q_n$ be a condition with $\vp^{M_\be}(q_{n+1})\geq s_n$. Finally, let $q^*$ be a sup of $\la q_n:n\in\om\ra$, and let $s^*$ be a sup of $\la s_n:n\in\om\ra$, noting by the intertwined construction that $s^*$ is also a sup of $\la\vp^{M_\be}(q_n):n\in\om\ra$. By the countable continuity of the residue functions, we have $\vp^{M_\be}(q^*)=^*s^*$. But $s^*\in\dom(\vp^{M_\al})$, since it is the sup of the increasing sequence $\la s_n:n\in\om\ra$ of conditions in $\dom(\vp^{M_\al})$. Consequently, $\vp^{M_\be}(q^*)$ is also in $\dom(\vp^{M_\al})$. Since $q_{-1}$ in $\ps^*/\lb p^*(M_\al),p^*(M_\be) \rb$ was arbitrary, this completes the proof of (2).
\end{proof}

\begin{proposition}\label{prop:twomasterconditions} Suppose that $\vec{M}$ is $\ps^*$-suitable with residue system 
$$
\la \la p^*(M_\ga),\vp^{M_\ga}\ra:\ga\in\dom(\vec{M})\ra,
$$
and let $\al<\be$ be in $\dom(\vec{M})$. Then 
$$
E(\vp^{M_\al},\vp^{M_\be}):=\lb p\in\ps^*:\vp^{M_\be}(p)\in \dom(\vp^{M_\al})\;\wedge\;\vp^{M_\al}(\vp^{M_\be}(p))=^*\vp^{M_\al}(p)\rb
$$
is $=^*$-countably closed and dense in $\ps^*/\lb p^*(M_\al),p^*(M_\be) \rb$.
\end{proposition}

\begin{proof} We begin by observing that if $q\in D_0(\vp^{M_\al},\vp^{M_\be})$, then $\vp^{M_\al}(q)$ extends $\vp^{M_\al}(\vp^{M_\be}(q))$. Indeed, since $q\in\dom(\vp^{M_\be})$, $q\geq\vp^{M_\be}(q)$, and since $\vp^{M_\al}$ is order-preserving and both $q$ and $\vp^{M_\be}(q)$ are in $\dom(\vp^{M_\al})$, we conclude that $\vp^{M_\al}(q)\geq\vp^{M_\al}(\vp^{M_\be}(q))$.

With this observation in mind, let $p\in\ps^*$ extend both $p^*(M_\al)$ and $p^*(M_\be)$, and by extending further if necessary, we may assume that $p$ is in $D_0(\vp^{M_\al},\vp^{M_\be})$. We will define by recursion an increasing sequence of conditions $\la p_n:n\in\om\ra$ in $D_0(\vp^{M_\al},\vp^{M_\be})$ with $p_0=p$ so that for all $n$, 
$$
\vp^{M_\al}(\vp^{M_\be}(p_{n+1}))\geq \vp^{M_\al}(p_{n})\geq\vp^{M_\al}(\vp^{M_\be}(p_{n}));
$$
note that all of the above items are defined, by definition of $D_0(\vp^{M_\al},\vp^{M_\be})$.

Suppose we are given $p_{n}$. As observed earlier, {since $p_n\in D_0(\vp^{M_\al},\vp^{M_\be})$, we have}  $\vp^{M_\al}(p_n)\geq\vp^{M_\al}(\vp^{M_\be}(p_n))$. {Since $\vp^{M_\al}(p_n)$ extends, in $M_\al$, the residue of $\vp^{M_\be}(p_n)$ to $M_\al$, we may find a condition $q\in M_\be$ extending $\vp^{M_\be}(p_n)$ so that $\vp^{M_\al}(q)\geq\vp^{M_\al}(p_n)$.} Since $q\in M_\be$ extends $\vp^{M_\be}(p_n)$, there is an $r\geq p_n$ so that $\vp^{M_\be}(r)\geq q$. Finally, let $p_{n+1}\geq r$ be a condition in $D_0(\vp^{M_\al},\vp^{M_\be})$. Then $\vp^{M_\be}(p_{n+1})\geq\vp^{M_\be}(r)\geq q$, and hence $\vp^{M_\al}(\vp^{M_\be}(p_{n+1}))\geq\vp^{M_\al}(q)\geq\vp^{M_\al}(p_n)$. This completes the construction of the desired sequence.


Let $p^*$ be a sup of $\la p_n:n\in\om\ra$. It is straightforward to verify that it witnesses the lemma. 
\end{proof}

The last lemma that we will need before turning to the main result of this section is a technical refinement of Lemma \ref{lemma:splittingpairs} which isolates circumstances in which for $\alpha < \beta < \kappa$ as above, we can find splitting pairs $u,v$ for $(M_\al,\rho)$ with the additional property that $u\res M_\be$ and $v\res M_\be$ also form a splitting pair for $(M_\al,\rho)$. Moreover, $u\res M_\be$ and $v\res M_\be$ will split the nodes on levels between $\al$ and $\be$ in the same way that $u$ and $v$ do. For the statement of the next result, recall the way we denote restriction of (iteration) length $p\uhr \xi$, and restriction in the poset height, $p\uhr M_\al$, from Notation \ref{notation:horizontalvertical}.

\begin{lemma}\label{lemma:splittingpairswithresidues} Suppose that $\vec{M}$ is in pre-splitting configuration up to $\rho$ and that $\dom(\vec{M})$ satisfies the conclusion of Corollary \ref{cor:IH1and2}(2). Suppose $\al<\be$ are both in $\dom(\vec{M})$ and that $\la p^*(M_\al),\vp^{M_\al}\ra$ and $\la p^*(M_\be),\vp^{M_\be}\ra$  are residue pairs for $(M_\al,\ps^*)$ and $(M_\be,\ps^*)$ respectively. Finally, fix a condition $u\in\R_\rho$ with $p_u\in\dom(\vp^{M_\al})\cap\dom(\vp^{M_\be})$. Then there exist a splitting pair $(u^*,v^*)$ for $(\vp^{M_\al},\rho)$ extending $u$ and a splitting function $\Si$ satisfying the following: 
\begin{enumerate}

\item $p_{u^*}$ and $p_{v^*}$ are both in $E(\vp^{M_\al},\vp^{M_\be})$ (see Proposition \ref{prop:twomasterconditions});
\item $(u^*\uhr M_\beta,v^*\uhr M_\beta)$ is also an 
  $(M_\al,\rho)$-splitting pair, and for any tuple $(\zeta,\nu,m,n)\in\dom(\Si)$ so that the $m$-th node of $f_{u^*}(\zeta,\nu)\bsl(\al\times\om_1)$ and the $n$-th node of $f_{v^*}(\zeta,\nu)\bsl(\al\times\om_1)$ are both in $M_\be$,
$$
(\vp^{M_\be}(p_{u^*}),(f_{u^*}\res M_\be)\res\zeta)\Vdash_{\R_\zeta}\Si(\zeta,\nu,m,n)(L)<_{\dot{T}_\zeta}\theta
$$
and
$$
(\vp^{M_\be}(p_{v^*}),(f_{v^*}\res M_\be)\res\zeta)\Vdash_{\R_\zeta}\Si(\zeta,\nu,m,n)(R)<_{\dot{T}_\zeta}\tau.
$$
\end{enumerate}
\end{lemma} 

\begin{proof} By Lemma \ref{lemma:denserestriction}, we know that $D(\vp^{M_\al},\rho)\cap D(\vp^{M_\be},\rho)$ is dense and $=^*$-countably closed in $\R_\rho/\lb p^*(M_\al),p^*(M_\be)\rb$. Moreover, by Proposition \ref{prop:twomasterconditions} (with the notation from the statement thereof), $E(\vp^{M_\al},\vp^{M_\be})$ is dense and countably $=^*$-closed in $\ps^*/\lb p^*(M_\al),p^*(M_\be)\rb$. Consequently, 
$$
E^*(\vp^{M_\al},\vp^{M_\be},\rho):=\lb v\in\R_\rho:v\in D(\vp^{M_\al},\rho)\cap D(\vp^{M_\be},\rho)\we p_v\in E(\vp^{M_\al},\vp^{M_\be})\rb
$$ 
is dense and countably $=^*$-closed in $\R_\rho/\lb p^*(M_\al),p^*(M_\be)\rb$. For use later, we also let $E^*(\vp^{M_\al},\vp^{M_\be},\zeta)$ be defined similarly, with $\zeta$ replacing $\rho$ in the above definition.

Let $u$ be as in the assumption of the current lemma. We may extend and relabel if necessary to assume that $u\in E^*(\vp^{M_\al},\vp^{M_\be},\rho)$.  We set $u_0:=v_0:=u$ and $w_0:=u\res M_\al$. We will now define a {coordinate-wise} increasing sequence of triples $\la\la u_n,v_n,w_n\ra:n\in\om\ra$ and a sequence of tuples $\la\la\zeta_n,\nu_n,\theta_n,\tau_n\ra:n\in\om\ra$ (with respect to some bookkeeping device) of tree nodes and ordinals so that the following conditions are satisfied {for all $n$:}
\begin{enumerate}
\item $\#^\rho_{\vp^{M_\al}}(u_n,v_n,w_n)$;
\item $u_n,v_n\in D(\vp^{M_\be},\rho)$;
\item $\la\zeta_n,\nu_n\ra\in\dom(f_{u_n})\cap\dom(f_{v_n})\cap M_\al$, and $\la\theta_n,\tau_n\ra\in (f_{u_n}(\zeta_n,\nu_n)\bsl(\al\times\om_1))\times(f_{v_n}(\zeta_n,v_n)\bsl(\al\times\om_1))$;
\item $u_{n+1}$ and $v_{n+1}$ split $\la\theta_n,\tau_n\ra$ below $\al$ in $\dot{T}_{\zeta_n}$, and if $\theta_n$ and $\tau_n$ are both below level $\be$, then in fact $u_{n+1}\res M_\be$ and $v_{n+1}\res M_\be$ split $\la\theta_n,\tau_n\ra$ below $\al$ in $\dot{T}_{\zeta_n}$. Moreover, in this case, there is a pair of nodes $\la\bar{\theta}_n,\bar{\tau}_n\ra$ below level $\al$ which witnesses the splitting for both $u_{n+1}$ and $v_{n+1}$ as well as their restrictions to $M_\be$;
\item $p_{u_n}$ and $p_{v_n}$ are in $E(\vp^{M_\al},\vp^{M_\be})$.
\end{enumerate}
For $n=0$, we have that (1), (2), and (5) hold because $u\in E^*(\vp^{M_\al},\vp^{M_\be},\rho)$. (3) holds by definition and (4) is vacuous.

Suppose, then, that we have defined $u_n,v_n$, and $w_n$. By Lemma \ref{lemma:1split}, we may find extensions $u'_n\geq u_n$, $v'_n\geq v_n$, and $w'_n\geq w_n$ so that $\#^\rho_{\vp^{M_\al}}(u'_n,v'_n,w'_n)$ holds and so that $u'_n$ and $v'_n$ split $\la\theta_n,\tau_n\ra$ below $\al$ in $\dot{T}_{\zeta_n}$. Let $\bar{\theta}_n$ and $\bar{\tau}_n$ be nodes below level $\al$ which witness the splitting.

We now define conditions $u^{***}_n$, $v^{***}_n$, and $w^{***}_n$ (the superscript for later notational purposes) which extend, respectively, $u'_n\res\zeta_n$, $v'_n\res\zeta_n$, and $w'_n\res\zeta_n$. If either $\theta_n$ or $\tau_n$ are at or above level $\be$ (namely, outside of $M_\beta$), then we simply set $u^{***}_n:=u'\res\zeta_n$, $v^{***}_n:=u'\res\zeta_n$. Since $\dom(\vec{M})$ satisfies the conclusion of Corollary \ref{cor:IH1and2} (2), and since $\zeta_n\in M_\al$, we may find a dual residue $w^{***}_n$ of $u^{***}_n$ and $v^{***}_n$ to $M_\al$. This completes the definition of the triple $(u^{***}_n,v^{***}_n,w^{***}_n)$ in the case that either $\theta_n$ or $\tau_n$ are at or above level $\be$.

Suppose on the other hand that $\theta_n$ and $\tau_n$ are both below level $\be$ and therefore are in $M_\be$.  Since $\dom(\vec{M})$ satisfies the conclusion of Corollary \ref{cor:IH1and2} (2), and since $\zeta_n\in M_\al\cap\rho$ and $\#^{\zeta_n}_{\vp^{M_\al}}(u'_n\res\zeta_n,v'_n\res\zeta_n,w'_n\res\zeta_n)$, we may find a condition $w^*_n\in M_\al\cap\R_{\zeta_n}$ which is a dual residue of $u'_n\res\zeta_n$ and $v'_n\res\zeta_n$ to $M_\al$. We next extend $u'_n\res\zeta_n$ to a condition which extends not only $w^*_n$ but also some residue to $M_\be$.

Let $u^{**}_n$ be an extension in $D(\vp^{M_\be},\zeta_n)$ of $u'_n\res\zeta_n$ and $w^*_n$. By the remarks before Lemma \ref{lemma:denserestriction}, we may let $\bar{u}^{**}_n\in M_\be$ be a residue of $u^{**}_n$ to $M_\be$ in $\R_{\zeta_n}$. Finally, let $u^{***}_n$ be a condition in $D(\vp^{M_\be},\zeta_n)\cap D(\vp^{M_\al},\zeta_n)$ which extends $u_n^{**}$ and $\bar{u}^{**}_n$ and which satisfies that $u^{***}_n\res M_\be\geq_{\R_{\zeta_n}}\bar{u}^{**}_n$.

By definition of $\bar{\theta}_n$ above, we know that $u'_n\res\zeta_n\Vdash_{\R_{\zeta_n}}\bar{\theta}_n<_{\dot{T}_{\zeta_n}}\theta_n$, and hence the extension $u^{**}_n$ of $u'\res\zeta_n$ forces this too. Since $\bar{u}^{**}_n$ is a residue of $u^{**}_n$ to $M_\be$ in $\R_{\zeta_n}$ and $\bar{\theta}_n$, $\theta_n$ are nodes in $M_\be$, we conclude that $\bar{u}^{**}_n$ also forces that $\bar{\theta}_n<_{\dot{T}_{\zeta_n}}\theta_n$. Finally, since $u^{***}_n\res M_\be$ is a condition (because $u^{***}_n\in D(\vp^{M_\be},\zeta_n)$) which extends $\bar{u}^{**}_n$, we conclude that $u^{***}_n\res M_\be$ forces that $\bar{\theta}_n<_{\dot{T}_{\zeta_n}}\theta_n$. This completes the first round of extensions of $u'_n\res\zeta_n$.

We now turn to extending $v'_n\res\zeta_n$. Since $u_n^{***}\in D(\vp^{M_\al},\zeta_n)$, we may let $w^{**}_n$ be a residue of $u^{***}_n$ to $M_\al$ which extends $w_n^*$. Since $w^{**}_n\geq w_n^*$ and $w_n^*$ is a residue of $v'_n\res\zeta_n$ to $M_\al$, $w^{**}_n$ is also a residue of $v'_n\res\zeta_n$ to $M_\al$. Applying the same argument as in the previous two paragraphs to $v'_n\res\zeta_n$ with $w^{**}_n$ playing the role of $w^*_n$ and with $\bar{\tau}_n$ and $\tau_n$ playing the respective roles of $\bar{\theta}_n$ and $\theta_n$, we may find an extension $v^{***}_n$ of $v'_n\res\zeta_n$ in $D(\vp^{M_\be},\zeta_n)\cap D(\vp^{M_\al},\zeta_n)$ and a condition $w^{***}_n$ so that $v^{***}_n\res M_\be$ forces $\bar{\tau}_n<_{\dot{T}_{\zeta_n}}\tau_n$ and so that $w^{***}_n$ is a residue of $v^{***}_n$ to $M_\al$ in $\R_{\zeta_n}$ which extends $w^{**}_n$. Note that since $w^{***}_n\geq w^{**}_n$, $w^{***}_n$ is also a residue of $u^{***}_n$ to $M_\al$ in $\R_{\zeta_n}$.

To summarize, we now have extensions $u^{***}_n$ and $v^{***}_n$ of $u'_n\res\zeta_n$ and $v'_n\res\zeta_n$ respectively which are both in $D(\vp^{M_\be},\rho)$ and which satisfy that $u^{***}_n\res M_\be$ and $v^{***}_n\res M_\be$ split $\la\theta_n,\tau_n\ra$ as witnessed by $\la\bar{\theta}_n,\bar{\tau}_n\ra$. Moreover, $w^{***}_n$ is a {dual} residue of $u^{***}_n$ and $v^{***}_n$ to $M_\al$ in $\R_{\zeta_n}$, i.e., $*^{\zeta_n}_{\vp^{M_\al}}(u^{***}_n,v^{***}_n,w^{***}_n)$. This completes the definition of the triple $(u^{***}_n,v^{***}_n,w^{***}_n)$ in the case that $\theta_n$ and $\tau_n$ are below level $\be$.

We now apply Lemma \ref{lemma:starimplieshashtag}, to find $u_{n+1}\res\zeta_n\geq u^{***}_n$, $v_{n+1}\res\zeta_n\geq v^{***}_n$ and $w_{n+1}\res\zeta_n\geq w^{***}_n$ so that $\#^{\zeta_n}_{\vp^{M_\al}}(u_{n+1}\res\zeta_n,v_{n+1}\res\zeta_n,w_{n+1}\res\zeta_n)$ and so that $u_{n+1}\res\zeta_n,v_{n+1}\res\zeta_n\in E^*(\vp^{M_\al},\vp^{M_\be},\zeta_n)$. Define $u_{n+1}:=(u_{n+1}\res\zeta_n)^\frown (u'_n\res[\zeta_n,\rho))$, with $v_{n+1}$ and $w_{n+1}$ defined similarly. $u_{n+1},v_{n+1}$, and $w_{n+1}$ then satisfy (1)-(5), completing the successor step of the construction.

If we now let $u^*,v^*,w^*$ be sups of their respective sequences, it is straightforward to see that they satisfy the lemma, using (4) to secure the desired splitting function.
\end{proof}

Having laid the groundwork in the previous results, we next turn to analyzing when quotients of $\R_\rho$ preserve stationary sets of cofinality $\om$ ordinals. We will prove the following proposition:

\begin{proposition}\label{prop:statsetpreservation} Suppose that $\vec{M}$ is in pre-splitting configuration up to $\rho$ and that $\dom(\vec{M})$ satisfies Corollary \ref{cor:IH1and2}(2). {Then there exists some $B^*\seq\dom(\vec{M})$ with $\dom(\vec{M})\bsl B^*\in\cal{I}$ so that for any }$\al\in B^*$, any $(\R_\rho\cap M_\al)$-name $\dot{S}$ for a stationary subset of $\al\cap\cof(\om)$, and any residue pair $\la p^*(M_\al),\vp^{M_\al}\ra$ for $(M_\al,\ps^*)$, the poset $\R_\rho/(p^*(M_\al),0_{\dot{\S}_\rho})$ forces that $\dot{S}$ remains stationary. 
\end{proposition}

Thus the quotient forcing of $\R_\rho$ above the condition $(p^*(M_\al),0_{\dot{\S}_\rho})$ preserves the stationarity of $\dot{S}$. The remainder of the section is devoted to the proof.
\vspace{1mm}

\begin{proof}
To begin, we define the set $B^*:=\tr(B)\cap B$, where $B=\dom(\vec{M})$. {Since $B\in\cal{F}^+$, Lemma \ref{lemma:SetAndTrace} implies that $B\bsl B^*\in\cal{I}$.}

Now fix, for the rest of the proof, an ordinal $\al\in B^*$ and a residue pair $\la p^*(M_\al),\vp^{M_\al}\ra$ for $(M_\al,\ps^*)$; since $\vec{M}$ is in pre-splitting configuration up to $\rho$, we may also fix, for each $\ga\in B\cap\al$, a residue pair $\la p^*(M_\ga),\vp^{M_\ga}\ra$ for $(M_\ga,\ps^*)$.

Next, fix a condition $(p,f)$ in $\R_\rho/(p^*(M_\al),0_{\dot{\S}_\rho})$ and an $\R_\rho$-name $\dot{C}$ for a closed unbounded subset of $\al$. We will find some extension $(p^*,f^*)$ of $(p,f)$ which forces in $\R_\rho$ that $ \dot{C}\cap\dot{S}\neq\es$.
By Lemma \ref{lemma:denserestriction}, we may assume that $(p,f)\in D(\vp^{M_\al},\rho)$.

In $V$, let $\theta>\ka^+$ be a large enough regular cardinal, and let $K\prec H(\theta)$ be chosen so that $|K|=\ka$, $\,^{<\ka}K\seq K$, and so that $K$ has the following parameters as elements: 
\begin{enumerate}[label=(\roman*)]
\item the sequences $\vec{M}$ and $\la\la p^*(M_\ga),\vp^{M_\ga}\ra:\ga\in B\cap (\al+1)\ra$, the set $B^*$, the poset $\R_\rho$, the $\R_\rho$-condition $(p,f)$, the $\R_\rho$-name $\dot{C}$, and the $(\R_\rho\cap M_\al)$-name $\dot{S}$;
\item the fixed well-order $\lhd$ of $H(\ka^+)$ from Notation \ref{notation:wo}.
\end{enumerate}
Finally, let $\bb{K}$ denote the tuple 
$(K,\in,\al,\vec{M},B^*,\R_\rho,(p,f),\dot{C},\dot{S},\lhd).
$

Define $E_0$ to be the club of $\be<\al$ so that $\Sk^{\bb{K}}(\be)\cap\al=\be$. Since $\al\in B^*$, we know that $B\cap\al$ is stationary in $\al$. Thus {$B\cap\lim(B)\cap\al$ is stationary in $\al$}, and therefore $E:=\lim(E_0\cap B)$ is a club in $\al$.
\vspace{3mm}

Recalling that $(p,f)\in D(\vp^{M_\al},\rho)$, we can find a residue $(\bar{p},\bar{f})$ of $(p,f)$ to $M_\al$ which extends the condition  $(\vp^{M_\al}(p),f\res M_\al)$.
 Let $\dot{X}$ be the $(\R_\rho\cap M_\al)$-name for 
 $$
  \lb\be\in E_0\cap B{\cap\lim(B)}\cap\al:(p^*(M_\be),0_{\dot{\S}_\rho})\in\dot{G}_{\R_\rho\cap M_\al}\rb;
  $$
 by Lemma \ref{lemma:unbounded} we know that $\dot{X}$ is forced by $\R_\rho\cap M_\al$ to be unbounded in $\al$.
 Since $\dot{S}$ is an $(\R_\rho\cap M_\al)$-name of a stationary subset of $\al\cap\cof(\omega)$, {$\dot{S}$} is forced to contain a limit point of $\dot{X}$.
 This, combined with the fact that $\R_\rho\cap M_\al$ does not add new $\omega$-sequences, implies that we can find an extension $(q,g)\geq_{\R_\rho\cap M_\al}(\bar{p},\bar{f})$ and
 an increasing sequence $\la\be_n:n\in\om\ra$ in $E_0\cap B{\cap\lim(B)}$ with $\sup_n\beta_n = \nu \in E \cap \cof(\omega)$, so that 
 (i) $(q,g)\Vdash_{\R_\rho\cap M_\al}\nu\in\dot{S}$, and 
 (ii) for all $n\in\om$, $q\geq_{\ps^*} p^*(M_{\be_n})$.

For the rest of the proof, we will fix $(q,g) \in \R_\rho \cap M_\alpha$, $\la \beta_n \mid n< \omega\ra$, and $\nu$ with the above properties. Define
$ 
K_{\be_n}:=\Sk^{\bb{K}}(\be_n),
$  
noting that $K_{\be_n}\cap\al=\be_n$ because  $\be_n\in E_0$. {It will be helpful for later to see that $M_{\be_n}\seq K_{\be_n}$ for each $n$. Indeed, $\be_n\in B\cap\lim(B)$ which implies that $M_{\be_n}=\bigcup_{\eta\in B\cap\be_n}M_\eta$. Moreover, $\be_n\seq K_{\be_n}$, and therefore applying the elementarity of $K_{\be_n}$, we see that for all $\eta\in B\cap\be_n$, $M_\eta\in K_{\be_n}$. Since $\be_n\seq K_{\be_n}$, $\eta\seq K_{\be_n}$ too. Thus $M_\eta\seq K_{\be_n}$ since $K_{\be_n}$ sees a bijection between $M_\eta$ and $\eta$. Combining all of this, we see that $M_{\be_n}=\bigcup_{\eta\in B\cap\be_n}M_\eta\seq K_{\be_n}$.}

We proceed to find an  extension $(p^*,f^*)$ of $(p,f)$ which is compatible with $(q,g)$ and forces that $\nu\in\dot{C}$. We will secure this by building two increasing $\om$-sequences of conditions, one above $(p,f)$ and another above $(q,g)$, in such a way that the limits of each sequence can be amalgamated; the resulting condition will then force $\nu$ into $\dot{S}\cap\dot{C}$. Let $(p_0,f_0):=(p,f)$ and $(q_0,g_0):=(q,g)$.

\begin{claim}\label{claim:sequences} There exists an increasing sequence $\la (p_n,f_n):n\in\om\ra$ of conditions in $\R_\rho$ and an increasing sequence $\la (q_n,g_n):n\in\om\ra$ of conditions in $\R_\rho\cap M_\al$ so that for each $n\in\om$,
\begin{enumerate}

\item $(p_n,f_n)\in K_{\be_n}$;
\item $(p_{n+1},f_{n+1})\Vdash_{\R_\rho}\dot{C}\cap(\be_n,\nu)\neq\es$;
\item $(p_n,f_n)\in D(\vp^{M_\al},\rho)$;
\item $q_n\geq_{{\ps^*}\cap M_\al}\vp^{M_\al}(p_n)$;
\item $f_{n+1}$ and $g_{n+1}$ are strongly compatible (Definition \ref{def:stronglycompatibleover}) over $p_{n+1}$ and $q_{n+1}$.
\end{enumerate}
\end{claim}

Before we prove this claim, we show that proving it suffices to obtain the desired condition $(p^*,f^*)$. So suppose that Claim \ref{claim:sequences} is true. Let $(p^*,f^*)$ be a sup of $\la(p_n,f_n):n\in\om\ra$, and let $(q^*,g^*)$ be a sup of  $\la(q_n,g_n):n\in\om\ra$. 

Observe that by item (2) of Claim \ref{claim:sequences} and the fact that the sequence $\la\be_n:n\in\om\ra$ is cofinal in $\nu$,  we have that $(p^*,f^*)\Vdash_{\R_\rho}\nu\in\lim(\dot{C})$ and hence forces that $\nu\in\dot{C}$ as $\dot{C}$ names a club. Also, since $(q^*,g^*)\geq (q_0,g_0)$ and since $(q_0,g_0)=(q,g)$ forces that $\nu\in\dot{S}$,  $(q^*,g^*)$ forces that $\nu\in\dot{S}$ too.
We claim that $p^*$ and $q^*$ are compatible in $\ps^*$, from which it follows by item (5) of Claim \ref{claim:sequences} that $f^*$ and $g^*$ are strongly compatible over $p^*$ and $q^*$. Indeed, $(p^*,f^*)\in D(\vp^{M_\al},\rho)$ since this set is closed under sups of increasing $\om$-sequences by Lemma \ref{lemma:denserestriction}. Furthermore,  by the countable continuity of $\vp^{M_\al}$, $\vp^{M_\al}(p^*)$ is a sup of the increasing sequence $\la\vp^{M_\al}(p_n):n\in\om\ra$. Thus to show that $p^*$ and $q^*$ are compatible, since $q^*\in\ps^*\cap M_\al$, it suffices to show that $q^*\geq_{\ps^*\cap M_\al}\vp^{M_\al}(p^*)$. However, we know that $q^*\geq q_n$ for all $n$ and so by (4) of Claim \ref{claim:sequences}, $q^*\geq\vp^{M_\al}(p_n)$ for all $n$. Therefore $q^*$ extends $\vp^{M_\al}(p^*)$, by definition of a supremum.

Now let $(p^{**},f^{**})$ be a condition in $\R_\rho$ above both $(p^*,f^*)$ and $(q^*,g^*)$. Then because $(p^{**},f^{**})$ extends $(p^*(M_\al),0_{\dot{\S}_\rho})$ as well as $(q^*,g^*)$, which in turn forces in $\R_\rho\cap M_\al$ that $\nu\in\dot{S}$, we have that $(p^{**},f^{**})\Vdash_{\R_\rho}\nu\in\dot{S}$. And finally, as $(p^{**},f^{**})$ extends $(p^*,f^*)$ which forces in $\R_\rho$ that $\nu\in\dot{C}$, we conclude that $(p^{**},f^{**})\Vdash_{\R_\rho}\nu\in\dot{S}\cap\dot{C}$. Thus it suffices to prove Claim \ref{claim:sequences} in order to finish the proof of Proposition \ref{prop:statsetpreservation}.
\vspace{3mm}

\noi\emph{Proof.} (of Claim 4.5) We will construct the sequences satisfying (1)-(5) of Claim \ref{claim:sequences} recursively. For the base case $n=0$, items (2) and (5) hold vacuously. For item (1), we have that $(p_0,f_0)=(p,f)\in K_{\be_0}$ as $(p,f)=(p_0,f_0)$ was chosen to be definable by a constant in the language of $\bb{K}$. We also ensured that $(p_0,f_0)\in D(\vp^{M_\al},\rho)$, which establishes (3). Finally, 
$
q_0\geq_{\ps^*\cap M_\al}\bar{p}\geq_{\ps^*\cap M_\al}\vp^{M_\al}(p_0),
$
which establishes (4).

Suppose, then, that we have defined $(p_n,f_n)$ and $(q_n,g_n)$ satisfying (1)-(5). We first observe that $(p_n,f_n)$ and $(q_n,g_n)$ are compatible. If $n=0$, this holds since $(q_0,g_0)$ is in $\R_\rho\cap M_\al$ and extends $(\bar{p},\bar{f})$, which is a residue of $(p_0,f_0)$ to $\R_\rho\cap M_\al$. If $n>0$, then we have that $q_n\geq_{\ps^*\cap M_\al}\vp^{M_\al}(p_n)$, and therefore $p_n$ and $q_n$ are $\ps^*$-compatible. Moreover, $f_n$ and $g_n$ are strongly compatible over the compatible conditions $p_n$ and $q_n$, and therefore $(p_n,f_n)$ and $(q_n,g_n)$ are compatible in $\R_\rho$.

Next choose some condition $(r,h)$ in $\R_\rho$ which extends $(p_n,f_n)$ and $(q_n,g_n)$, and by extending if necessary, we may assume that there is some ordinal $\mu>\be_n$ so that $(r,h)\Vdash_{\R_\rho}\mu\in\dot{C}\bsl(\be_n+1)$. Since $r\geq p_n\geq p^*(M_\al)$ and since $r\geq q_n\geq q_0\geq p^*(M_{\be_{n+1}})$, we may also extend if necessary to assume, by Lemma \ref{lemma:firstresiduelemma}(1), that $r\in\dom(\vp^{M_{\be_{n+1}}})\cap\dom(\vp^{M_\al})$ {and also that $\vp^{M_\al}(r)\geq q_n$}.

We now apply Lemma \ref{lemma:splittingpairswithresidues}, with $\al$ and $\be_{n+1}$ playing the respective roles of ``$\be$" and ``$\al$" in the statement thereof,  to find extensions $(r_L,h_L)$ and $(r_R,h_R)$ of $(r,h)$ which satisfy the conclusion of that lemma. We let $(\bar{r},\bar{h})$ be a condition so that $\#^\rho_{\vp^{M_{\be_{n+1}}}}((r_L,h_L),(r_R,h_R),(\bar{r},\bar{h}))$.

Let $\Si$ be a splitting function for $(r_L,h_L)$ and $(r_R,h_R)$ with respect to the model $M_{\be_{n+1}}$ which satisfies  Lemma \ref{lemma:splittingpairswithresidues}. For $Z\in\lb L,R\rb$, set 
$$
x_Z:=\dom(h_Z)\cap M_{\be_{n+1}};
$$
this is a countable subset of $M_{\be_{n+1}}$ and therefore is a member of $M_{\be_{n+1}}$. {Since $M_{\be_{n+1}}\seq K_{\be_{n+1}}$, as shown earlier,}  $x_Z$ is also an element of $K_{\be_{n+1}}$.

We are now in a position to reflect into the model $K_{\be_{n+1}}$. We observe that in $H(\theta)$ the following statement is true in the following parameters $\beta_n,\Sigma,\bar{r},\bar{h}, \R_\rho, (p_n,f_n), \alpha$, $B, \dot{C}, x_L$, $x_R$, {$\vec{M}$, and $\la\la p^*(M_\ga),\vp^{M_\ga}\ra:\ga\in B\cap (\al+1)\ra$}, all of which are in
$ K_{\be_{n+1}}$: there exists a condition $(r^*,h^*)$ in $\R_\rho$ and a  pair $(r^*_Z,h^*_Z)_{Z\in\lb L,R\rb}$ of conditions above $(r^*,h^*)$ in $\R_\rho$ as well as ordinals $\mu^*,\eta$, so that
\begin{enumerate}[label=(\roman*)]
\item $(r^*,h^*)\geq_{\R_\rho} (p_n,f_n)$;
\item $\eta\in B$;
\item $\#^\rho_{\vp^{M_\eta}} ((r^*_L,h^*_L),(r^*_R,h^*_R),(\bar{r},\bar{h}))$;
\item for each $Z\in\lb L,R\rb$, $(r^*_Z,h^*_Z)$ and $(r^*,h^*)$ are in $E^*(\vp^{M_\eta},\vp^{M_\al})$ (see Proposition \ref{prop:twomasterconditions});
\item $(r^*,h^*)\Vdash_{\R_\rho}\mu^*\in\dot{C}\bsl(\be_n+1)$;
\item $\dom(h^*_Z)\cap M_\eta=x_Z$;
\item $(r^*_L,h^*_L)$ and $(r^*_R,h^*_R)$ are an $(M_\eta,\rho)$-splitting pair, and $\Si$ is a splitting function for $(r^*_L,h^*_L)$ and $(r^*_R,h^*_R)$ with respect to the model $M_\eta$.
\end{enumerate}
This statement is true in $H(\theta)$ as witnessed by the conditions $(r_Z,h_Z)_{Z\in\lb L,R\rb}$ and $(r,h)$, the ordinal $\mu$ playing the role of $\mu^*$, and the ordinal $\be_{n+1}$ playing the role of $\eta$. Since the parameters of this statement are in $K_{\be_{n+1}}$,
we may therefore find, in $K_{\be_{n+1}}$, conditions $(r^*_Z,h^*_Z)_{Z\in\lb L,R\rb}$ extending some $(r^*,h^*)\geq (p_n,f_n)$, an ordinal $\mu^*$, and an ordinal $\eta\in B$ so that (i)-(vii) above are satisfied of these objects.

We now define $(p_{n+1},f_{n+1}):=(r^*_L,h^*_L)$. We need to extend the condition $(\vp^{M_{\al}}(r_R),h_R\res M_\al)$ a bit more before defining $(q_{n+1},g_{n+1})$. The following claim will help us do this:

\begin{subclaim}\label{sc:meow} $\vp^{M_\al}(p_{n+1})$ and $\vp^{M_\al}(r_R)$ are compatible in $\ps^*\cap M_\al$.
\end{subclaim}

\noi\emph{Proof.}
Both the condition $p_{n+1}$ and the function $\vp^{M_\al}$ are members of  $K_{\be_{n+1}}$. Therefore $\vp^{M_\al}(p_{n+1})\in K_{\be_{n+1}}\cap M_\al\cap\ps^*$. Recall that $\bb{K}$ contained the fixed well-order $\lhd$ of $H(\ka^+)$ and that all suitable models are elementary in $H(\ka^+)$ with respect to $\lhd$. Thus if we let $e^{\ps^*}$ denote the $\lhd$-least bijection from $\ka$ onto $\ps^*$, then we have that $e^{\ps^*}$ is in $M_\al$ and in $K_{\be_{n+1}}$. Since $M_\al$ is elementary and contains $e^{\ps^*}$, we see that $\vp^{M_\al}(p_{n+1})=e^{\ps^*}(\zeta)$ for some $\zeta<\al$. But then by the elementarity of $K_{\be_{n+1}}$, we see that $\zeta\in K_{\be_{n+1}}\cap\al=\be_{n+1}$. Therefore $\vp^{M_\al}(p_{n+1})=e^{\ps^*}(\zeta)\in M_{\be_{n+1}}$. Furthermore, we know that
$$
\bar{r}=^*\vp^{M_{\be_{n+1}}}(r_R)=^*\vp^{M_{\be_{n+1}}}(\vp^{M_\al}(r_R))
$$
where the first equality holds by definition of $\bar{r}$ and the second because $r_R$ satisfies Lemma \ref{lemma:splittingpairswithresidues}. Applying (iii) and (iv) above we also have that,
$$
\bar{r}=^*\vp^{M_\eta}(p_{n+1})=^*\vp^{M_\eta}(\vp^{M_\al}(p_{n+1})).
$$
Additionally, since $\vp^{M_\eta}$ is an exact, strong residue function and $\vp^{M_\al}(p_{n+1})\in\dom(\vp^{M_\eta})$, we know that 
$$
\vp^{M_\al}(p_{n+1})\geq \vp^{M_\eta}(\vp^{M_\al}(p_{n+1}))=^*\bar{r}.
$$
Therefore, as $\vp^{M_\al}(p_{n+1})\in M_{\be_{n+1}}$ extends $\bar{r}$, which is a residue of $\vp^{M_\al}(r_R)$ to $M_{\be_{n+1}}$, we conclude that $\vp^{M_\al}(p_{n+1})$ is compatible with $\vp^{M_\al}(r_R)$ in $\ps^*\cap M_\al$.
\qed(Subclaim \ref{sc:meow})

\vspace{.1in}

Using Subclaim \ref{sc:meow}, we may fix some condition $q_{n+1}$ in $\ps^*\cap M_\al$ which is above both $\vp^{M_\al}(p_{n+1})$ and $\vp^{M_\al}(r_R)$. We finally set $g_{n+1}:=h_R\res M_\al$, noting that $(q_{n+1},g_{n+1})\in M_\al$.

We next verify that items (1)-(5) of Claim \ref{claim:sequences} hold for $n+1$. We have that $(p_{n+1},f_{n+1})\geq (p_n,f_n)$ by (i) of the reflection, more precisely, since 
$$
(p_{n+1},f_{n+1})=(r^*_L,h^*_L)\geq (r^*,h^*)\geq (p_n,f_n).
$$
Additionally, because  
$$
q_{n+1}\geq\vp^{M_\al}(r_R)\geq\vp^{M_\al}(r)\geq q_n,
$$
we have that $q_{n+1}\geq q_n$. Moreover, $g_{n+1}$ extends $g_n$ as a function: {$g_{n+1}=h_R\res M_\al$, $g_n\in M_\al$, and $(r_R,h_R)\geq (r,h)\geq (q_n,g_n)$.} Thus $(q_{n+1},g_{n+1})$ extends $(q_n,g_n)$. (1) of Claim \ref{claim:sequences} holds because we found the witnesses in the model $K_{\be_{n+1}}$. For (2), $\mu^*\in K_{\be_{n+1}}\cap\al=\be_{n+1}\seq\nu$, and since $\mu^*>\be_n$, we have that $(p_{n+1},f_{n+1})\Vdash\mu^*\in\dot{C}\cap(\be_n,\nu)$. For (3), we have $(p_{n+1},f_{n+1})\in D(\vp^{M_\al},\rho)$ by (iii) and the definition of $\#^\rho_{\vp^{M_\eta}}$ (see Definition \ref{def:starhashtag}). For (4), we have that $q_{n+1}\geq\vp^{M_\al}(p_{n+1})$ by choice of $q_{n+1}$.

It remains therefore to check that item (5) of Claim \ref{claim:sequences} holds. Since $q_{n+1}\geq_{\ps^*\cap M_\al}\vp^{M_\al}(p_{n+1})$, we know that $p_{n+1}$ and $q_{n+1}$ are compatible in $\ps^*$; let $p^*\in\ps^*$ be any condition extending both. We claim that 
$$
p^* \Vdash_{\ps^*}\check{f}_{n+1}\cup \check{g}_{n+1}\in\dot{\S}_\rho.
$$

Suppose by induction on $\de<\rho$ that $\la\de,\nu\ra\in\dom(f_{n+1})\cap\dom(g_{n+1})$ and that {$p^*$} forces that the union of $\check{f}_{n+1}\res\de$ and $\check{g}_{n+1}\res\de$ is a condition in $\dot{\S}_\de$. Again using the fixed well-order $\lhd$ of $H(\ka^+)$, we may let $\psi$ be the $\lhd$-least bijection from $\ka$ onto $\rho$, so that $\psi$ is a member of $K_{\be_{n+1}}$ as well as every model on the $\R_\rho$-suitable sequence $\vec{M}$. Since $\la\de,\nu\ra\in\dom(g_{n+1})$ and $g_{n+1}\in M_\al$, $\de\in M_\al\cap\rho=\psi[\al]$. Furthermore, since $\la\de,\nu\ra\in\dom(f_{n+1})$ and $f_{n+1}=h^*_L\in K_{\be_{n+1}}$, we have that $\de\in K_{\be_{n+1}}$. Thus 
$$
\de\in\psi[\al]\cap K_{\be_{n+1}}=\psi[K_{\be_{n+1}}\cap\al]=\psi[\be_{n+1}]\seq M_{\be_{n+1}}.
$$
Therefore (recalling that $g_{n+1}=h_R\res M_\al$), 
$$
\la\de,\nu\ra\in\dom(h_R)\cap M_{\be_{n+1}}= x_R=\dom(h^*_R)\cap M_\eta.
$$

Continuing, fix a pair 
$$
\la\theta,\tau\ra\in f_{n+1}(\de,\nu)\times g_{n+1}(\de,\nu)
$$ 
with $\theta\neq\tau$. We need to show that $\theta$ and $\tau$ are forced to be incompatible nodes in the tree $\dot{T}_\de$ by the condition $\big(p^*,(f_{n+1}\cup g_{n+1})\res\de\big)$. Recall, going forward, that $(\bar{r},\bar{h})$ equals both $(r_R,h_R)\res M_{\be_{n+1}}$ and $(p_{n+1},f_{n+1})\res M_\eta$; in particular, $f_{n+1}$ and $h_R\res M_\al=g_{n+1}$ both extend $\bar{h}$. Continuing, if $\tau$ is below level $\be_{n+1}$, then $\tau\in\bar{h}(\de,\nu)\seq f_{n+1}(\de,\nu)$ and we are done. Furthermore, if $\theta$ is below level $\eta$, then $\theta\in\bar{h}(\de,\nu)\seq g_{n+1}(\de,\nu)$ and we are done in this case too. Thus we assume that $\theta$ is at or above level $\eta$ and that $\tau$ is at or above level $\be_{n+1}$. With respect to the fixed enumerations, let $k$ and $m$ be chosen so that $\theta$ is the $k$th element of $f_{n+1}(\de,\nu)\bsl(\eta\times\om_1)$ and $\tau$ is the $m$th element of $h_R(\de,\nu)\bsl(\be_{n+1}\times\om_1)$. Then because the function $\Si$ is the same for both pairs of splitting conditions, we know that
$$
(p_{n+1},f_{n+1}\res\de)\Vdash_{\R_\de}\Si(\de,\nu,k,m)(L)<_{\dot{T}_\de}\theta
$$
and that 
$$
(r_R,h_R\res\de)\Vdash_{\R_\de}\Si(\de,\nu,k,m)(R)<_{\dot{T}_\de}\tau.
$$
However, $\tau\in g_{n+1}(\de,\nu)=(h_R\res M_\al)(\de,\nu)$, and therefore $\tau$ is below level $\al$ of the tree $\dot{T}_\de$. Therefore by Lemma \ref{lemma:splittingpairswithresidues}, we have that
$$
\Big(\vp^{M_\al}(r_R),(h_R\res M_\al)\res\de\Big)\Vdash_{\R_\de}\Si(\de,\nu,k,m)(R)<_{\dot{T}_\de}\tau.
$$
Since $q_{n+1}\geq\vp^{M_\al}(r_R)$ and $g_{n+1}=h_R\res M_\al$, we conclude that
$$
(q_{n+1},g_{n+1}\res\de)\Vdash_{\R_\de}\Si(\de,\nu,k,m)(R)<_{\dot{T}_\de}\tau.
$$
Finally, since $\big(p^*,(f_{n+1}\cup g_{n+1})\res\de\big)$ is above both $(p_{n+1},f_{n+1}\res\de)$ and $(q_{n+1},g_{n+1}\res\de)$, it follows that
$$
\big(p^*,(f_{n+1}\cup g_{n+1})\res\de\big)\Vdash\Si(\de,\nu,k,m)(R)<_{\dot{T}_\de}\tau\we\Si(\de,\nu,k,m)(L)<_{\dot{T}_\de}\theta.
$$
Since the distinct nodes $\Si(\de,\nu,k,m)(L)$ and $\Si(\de,\nu,k,m)(R)$ are on the same level, $\big(p^*,(f_{n+1}\cup g_{n+1})\res\de\big)$ therefore forces that $\theta$ and $\tau$ are incompatible nodes in the tree $\dot{T}_\de$. 

This completes the proof that $f_{n+1}$ and $g_{n+1}$ are strongly compatible over $p_{n+1}$ and $q_{n+1}$. Therefore the proof of Claim \ref{claim:sequences} is now complete.
\qed(Claim \ref{claim:sequences})

\vspace{.1in}

As remarked earlier, this completes the proof of Proposition \ref{prop:statsetpreservation}.
\end{proof}

\begin{remark}\label{remark:LSWC} {As we've noted before, if $\ps^*$ is just equal to the collapse poset $\ps$,  then the results from Section 3 which are needed for this section hold only assuming that $\ka$ is weakly compact (since then $\cal{F}=\cal{F}_{WC}$; see Definition \ref{def:fksp}). We then see that if $\ps^*$ is just $\ps$, then the arguments in this section can also be carried out only using a weakly compact.}
\end{remark}

As a corollary of Proposition \ref{prop:statsetpreservation}, we can now prove Theorem \ref{theorem:OriginalLS}.
\begin{proof}
Recall that the Laver-Shelah model $V[G*F]$ is obtained by starting from a ground model $V$ with a weakly compact cardinal $\ka$, and forcing with the Levy collase $\ps$ followed by a countable support iteration $\S = \la \S_\tau,\S(\tau)\mid \tau < \ka^+\ra$ of specializing posets $\S(\tau) = \S(\name{T_\tau})$ of Aronszajn trees on $\ka$, chosen by a bookkeeping function. Since $\S$ satisfies the $\ka$-c.c, every sequence of stationary sets $\la S_\al \mid \al < \ka\ra$ as in the statement of Theorem \ref{theorem:OriginalLS}, belongs to an intermediate extension $V[G*F_\tau]$, where  $F_\tau := F \cap \S_\tau$, for some $\tau<\ka^+$. Now work in $V$, and take a $(\po*\dot\S_\tau)$-name $\la \name{S}_\al \mid \al < \ka\ra$ for the sequence of stationary sets. Let $\vec{M}$ be suitable with respect to these parameters. By the weak compactness of $\ka$, let $\be\in\dom(\vec{M})$ be such that $\la \name{S}_\al \cap V_\beta \mid \al < \be\ra$ are names for stationary subsets of $\beta$ in the restricted poset $(\po * \dot\S_\tau)\cap M_\beta$. {Recalling Remark \ref{remark:LSWC}, we see that  Proposition \ref{prop:statsetpreservation} can be applied to $\ps$, and in this case, $\cal{F}=\cal{F}_{WC}$.} We conclude that each $S_\al \cap \beta$ remains stationary in the full $\po * \dot\S_\tau$ generic extension $V[G*F_\tau]$ and hence in $V[G\ast F]$.

To see that $\mathsf{CSR}(\om_2)$ fails {in the Laver-Shelah model}, observe that in the ground model, there exist stationary sets $S\seq\ka\cap\cof(\om)$ and $T\seq\ka\cap\cof(\om_1)$ so that $S$ does not reflect at any point in $T$ (see Proposition 1.1 of \cite{JechShelah}). The stationarity of $S$ and $T$ are preserved by the $\ka$-c.c. forcing of Laver and Shelah, and since $\om_1$ is also preserved, we have that $S$ and $T$ witness the failure of $\mathsf{CSR}(\om_2)$ in the final model.
\end{proof}

\section{{$\cal{F}$}-completely proper posets}\label{sec:FWCcompletelyproper}

In this section, we will specify what $\ps$-names $\dot{\C}$ for posets are such that $\ps\ast\dot{\C}$ is ${\cal{F}}$-strongly proper, and we will draw some conclusions from this. {Since $\ps\ast\dot{\C}$ will be playing the role of $\ps^*$ in Definition \ref{def:fksp}, and since $\ps\ast\dot{\C}$ is not merely the collapse, we are in the case when $\ka$ is ineffable and $\cal{F}=\cal{F}_{in}$. However, we only use the ineffability of $\ka$ in the following section when applying Proposition \ref{prop:InductiveIandII} through its use in Proposition \ref{prop:statsetpreservation}.} Recall Definition \ref{def:fksp} for the definition of ${\cal{F}}$, and also recall that $\ps$ denotes the Levy collapse poset $\col(\om_1,<\ka)$, where $\ka$ is either ineffable or weakly compact. Two main ideas come into play in this section. The first is an axiomatization of various properties of the iterated club adding $\dot{\C}_{\text{Magidor}}$ from \cite{Magidor}, which will allow us to place upper bounds on various ``local" filters added by the Levy collapse. We then couple this axiomatization with a generalization of a result of Abraham's (\cite{Abraham}) that, in current language, if $\dot{\Q}$ is an $\Add(\om,\om_1)$-name for an $\om_1$-closed poset, then $\Add(\om,\om_1)\ast\dot{\Q}$ is strongly proper; see \cite{GiltonNeeman} for a proof of this fact as stated here. The strong properness results from using  so-called ``guiding reals."

We recall that in \cite{Magidor}, to show that an iteration $\dot{\C}_{\text{Magidor}}$ of length $<\ka^+$ adding the desired clubs is $\ka$-distributive, Magidor argued, in part, as follows: let $j:M\lra N$ be a weakly compact embedding, where $M$ has the relevant parameters. Let $G^*$ be $N$-generic over $j(\ps)$ and $G:=G^*\cap\ps$, so that in $N[G^*]$, we may construct an $M[G]$-generic filter $H$ for $\C_{\text{Magidor}}$. Moreover, $j[H]$ has a least upper bound in $j(\C_{\text{Magidor}})$, namely, the function obtained by placing $\ka$ on top of each coordinate in the domain of $j[H]$; by the closure of the quotient (which implies the preservation of the stationary sets appearing along the way in the definition of $\C_{\text{Magidor}}$), this is indeed a condition.

The property of $\dot{\C}_{\text{Magidor}}$ which we will axiomatize is a reflection of the above to an {$\cal{F}$-positive} set of $\al<\ka$. Roughly, we want to say that for many $\al$, if you ``cut off" $\ps\ast\dot{\C}$ at $\al$, then many generics added by the tail of the collapse for ``$\dot{\C}$ cut off at $\al$'' have upper bounds in the full poset $\dot{\C}$. More precisely, given a $\ps$-name $\dot{\C}$ in $H(\ka^+)$ for a poset which is $\om_1$-closed with sups and given a $\dot{\C}$-suitable model (see Definition \ref{def:suitable}) $M$, say with $M\cap\ka=\al<\ka$, we consider the poset $\pi_M(\dot{\C})$, where $\pi_M$ denotes the transitive collapse of $M$ to $\bar{M}$. An easy absoluteness argument shows that $\pi_M(\dot{\C})$ is a name in $\pi_M(\ps)=\ps\res\al$. Appealing to the closure of $\bar{M}$ under $<\al$-sequences, and hence $\om$-sequences, we see that $\pi_M(\dot{\C})$ is forced by $\ps\res\al$ to be $\om_1$-closed with sups. The desired condition on $\ps$-names $\dot{\C}$ can now be stated a bit more precisely: we will demand that after forcing with $\ps$, say to add the generic $G$, for many $\al$ as above and many $V[G_{\ps\res\al}]$-generics $H$ for $\pi_M(\dot{\C})[G_{\ps\res\al}]$ in $V[G]$, $\pi^{-1}_{M[G]}[H]$ has an upper bound in $\dot{\C}[G]$. Note that we are implicitly appealing to the properness of $\ps$ with respect to $M$ to see that $\pi_M:M\lra\bar{M}$ lifts to $\pi_{M[G]}:M[G]\lra\bar{M}[G_{\ps\res\al}]$; we discuss this more later.

The first step to making this work is to isolate exactly which filters we will use; for reasons related to building strong, exact residue functions later, we will not consider all filters added by the tail of the collapse for $\pi_M(\dot{\C})$. The definition is meant to capture the behavior of filters generated by using the generic surjections to guide choices of conditions, similar to how Abraham used guiding reals in \cite{Abraham}.

\subsection{Residue Functions from Local Filters}

\begin{definition}\label{def:filtersfromcollapse} Let $\al<\ka$ be inaccessible, and let $\dot{\Q}$ be a $(\ps\res\al)$-name for a poset of size $\al$ which is $\om_1$-closed with sups. Since $\po\res\alpha$ is $\alpha$-c.c there exists a list $\la\dot{\ga}_i:i<\al\ra$ of $(\po\res \alpha)$-names, which is forced to enumerate all conditions in $\dot{\Q}$. We say that a sequence $\dot{s}=\la\dot{d}_\nu:\nu<\om_1\ra$ of $\ps$-names (not $(\ps\res\al)$-names)  for conditions in $\dot{\Q}$ is \textbf{guided by the collapse at $\al$ for $\dot{\Q}$} if the following conditions are satisfied:
\begin{enumerate}
\item $\Vdash_\ps\la\dot{d}_\nu:\nu<\om_1\ra$ is $\leq_{\dot{\Q}}$-increasing, $\dot{d}_0$ is the weakest  condition in $\dot{\Q}$, and if $\nu$ is limit, then $\dot{d}_\nu$ is a sup of $\la\dot{d}_\mu:\mu<\nu\ra$;
\item if $p\in\ps$ and $\dom(p(\al))$ is an ordinal $\nu<\om_1$, then there exists $p'\geq p$ with $p'\res[\al,\ka)=p\res[\al,\ka)$ and a sequence $\la\be(\mu):\mu\leq\nu\ra$ of ordinals in $V$ so that $p'\Vdash\dot{d}_\mu=^*\dot{\ga}_{\be(\mu)}$ for all $\mu\leq\nu$. In this case we will say that $p'$ \textbf{determines an initial segment of $\dot{s}$};
\item if $p'$ as in (2) determines an initial segment of $\dot{s}$ and if $\dot{\ga}$ is a $(\ps\res\al)$-name for a $\dot{\Q}$-extension of $\dot{\ga}_{\be(\nu)}$, then there exists $p^*\geq p'$ so that $p^*\Vdash\dot{d}_{\nu+1}\geq\dot{\ga}$.
\end{enumerate}
\end{definition}

\begin{lemma}\label{lemma:itgeneratesthegoods} Suppose that $\dot{s}:=\la\dot{d}_\nu:\nu<\om_1\ra$ is guided by the collapse at $\al$. Let $H(\dot{s})$ be the $\ps$-name for the filter on $\dot{\Q}$ generated by $\dot{s}$. Then $\ps$ forces that $H(\dot{s})$ is $V[\dot{G}\res\al]$-generic over $\dot{\Q}$.
\end{lemma}
\begin{proof} Fix a condition $p\in\ps$ and a $\ps$-name $\dot{D}$ for a dense subset of $\dot{\Q}$ which is a member of $V[\dot{G}\res\al]$. We find an extension of $p$ which forces that $\dot{D}\cap H(\dot{s})\neq\es$.  By extending $p$ and applying (2) of Definition \ref{def:filtersfromcollapse} if necessary, we may assume the following:
\begin{enumerate}
\item  there is a $(\ps\res\al)$-name $\dot{D}_0$ for a dense subset of $\dot{\Q}$ so that $p\Vdash\dot{D}=\dot{D}_0$;
\item $\dom(p(\al))$ is an ordinal $\nu$, and there is a sequence $\la\be(\mu):\mu\leq\nu\ra$ of ordinals in $V$ so that $p\Vdash\dot{d}_\mu=^*\dot{\ga}_{\be(\mu)}$ for all $\mu\leq\nu$. 
\end{enumerate}
Let $\dot{\ga}$ be a $(\ps\res\al)$-name for a condition in $\dot{D}_0$ forced to extend $\dot{\ga}_{\be(\nu)}$. By item (3) of Definition \ref{def:filtersfromcollapse}, we may find an extension $p^*$ of $p$ so that $p^*\Vdash\dot{d}_{\nu+1}\geq\dot{\ga}$. Then $p^*\Vdash\dot{\ga}\in\dot{D}\cap H(\dot{s})$, finishing the proof.
\end{proof}

\begin{definition}\label{def:genbycoll} Let $\dot{H}$ be a $\ps$-name for a filter on $\dot{\Q}$. We say that $\dot{H}$ is \textbf{guided by the collapse at $\al$} if there is a sequence $\dot{s}=\la\dot{d}_\nu:\nu<\om_1\ra$ of conditions guided by the collapse at $\al$ so that $\dot{H}=H(\dot{s})$.
\end{definition}

Suppose that $M$ is a suitable model. Let $\al:=M\cap\ka$, and let $\pi_M: M \to \bar{M}$ be the transitive collapse map of $M$. 
Let $G \subseteq \ps$ be generic over $V$, and set $G_\alpha = G \cap (\po\res\alpha)$.
We have that 
$\po\res\alpha = \pi_M(\ps) \in \bar{M}$, 
and $G_\alpha \subset \pi_M(\ps)$ is generic for $\bar{M}$. 
Moreover, setting $M[G] = \{ \name{x}[G] \mid \name{x} \in M \text{ is a }{\ps}\text{-name}\}$, we have that 
$\bar{M}[G_\alpha]$ is the transitive collapse of $M[G]$, with the transitive collapse map $\pi_{M[G]}$ being the natural extension of $\pi_M$, given by 
$\pi_{M[G]}(\name{x}[G]) = \pi_M(\name{x})[G_\alpha]$.

\begin{lemma}\label{lemma:getCGConditions}
Suppose that $M$ is a $(\ps\ast\dot{\C})$-suitable model, where $\dot{\C}$ is a $\ps$-name for a poset on $\ka$ which is $\om_1$-closed with sups. Let $\al:=M\cap\ka$ and $\pi_{M}$ be the transitive collapse map of $M$.  Suppose that $\dot{H}$ is a $\ps$-name
for a subset of $\pi_M(\dot{\C})$ which is guided by the collapse at $\al$ for $\pi_M(\dot{\C})$, and further suppose that there is a $\ps$-name $\dot{c}$ for a condition in $\dot{\C}$ which is forced to be an upper bound for $\pi^{-1}_{M[\name{G}]}[\dot{H}]$. Then $\ps$ forces that $\dot{c}$ is an $(M[\dot{G}],\dot{\C})$-completely-generic condition.
\end{lemma}
\begin{proof} 
Fix $G$. To see that $c = {\dot{c}[G]}$ is $(M[G],\C)$-completely generic, fix a dense, open $E\seq\C$ with $E\in M[G]$. We show that $c$ extends some condition in $E$.

By the elementarity of $\pi_{M[G]} : M[G] \to \bar{M}[G_\al]$, we know that $\pi_{M[G]}(E)$ is dense in $\pi_{M[G]}(\C)=\pi_M(\dot{\C})[G_\al]$. Since $H:=\dot{H}[G]$ is a $V[{G_\al}]$-generic filter, by Lemma \ref{lemma:itgeneratesthegoods}, $H\cap\pi_{M[G]}(E)\neq\es$, and thus $\pi^{-1}_{M[G]}[H] \cap E \neq \es$.
\end{proof}

There are two particularly useful properties of this class of names for generic filters. On the one hand, filters in this class will allow us to generate strong, exact residue functions by isolating the information which a given conditions determines about the filters. On the other hand, the class of such filters for one poset, such as a two-step iteration, often projects to the class of such filters for another poset, such as the first step in a two-step iteration. This property will be particularly useful in Section \ref{section:clubscompletelyproper}  when we want to show, by induction, that our club adding poset is well-behaved.

It is straightforward to verify that the notion of
$\ps$-names {of} filters, which are guided by the collapse at a given cardinal, factor well in iterations.
\begin{lemma}\label{lemma:projectcollapsefilter} Suppose that $\al<\ka$ is inaccessible and that $\dot{\Q_0}\ast\dot{\Q}_1$ is a $(\ps\res\al)$-name for a two-step poset of size $\al$ which is $\om_1$-closed with sups. 
Let $\la\dot{d}_\nu:\nu<\om_1\ra$ be a sequence of $\ps$-names which is guided by the collapse at $\al$ for $\dot{\Q}_0\ast\dot{\Q}_1$. Then the sequence $\la\dot{d}_\nu(0):\nu<\om_1\ra$  of $\ps$-names of conditions in $\dot{\Q}_0$ is guided by the collapse at $\al$ for $\dot{\Q}_0$.
\end{lemma}

The following proposition shows how to generate exact, strong residue functions from the filters discussed above.

\begin{proposition}\label{prop:gettingESRFs}
Suppose that $\dot{\C}$ is a $\ps$-name in $H(\ka^+)$ for a poset of size $\ka$ which is $\om_1$-closed with sups, and set $\ps^*:=\ps\ast\dot{\C}$. Let $M$ be a $(\ps\ast\dot{\C})$-suitable model, say with $\al:=M\cap\ka<\ka$, and let $\pi_M$ denote the transitive collapse of $M$. Also let
 $\dot{s}=\la\dot{d}_\nu:\nu<\om_1\ra$ be a sequence of $\ps$-names guided by the collapse at $\al$ for $\pi_M(\dot{\C})$. Additionally, suppose that there is a $\ps$-name $\dot{d}^*$ for a condition in $\dot{\C}$ which is forced to be an upper bound for the sequence $\pi_{M[\name{G}]}^{-1}[\dot{s}]$.

Define $p^*(M)$ to be the condition 
$$
p^*(M):=(0_{\ps},\dot{d}^*),
$$
and let
$$
D(M):=\lb (p,\dot{d})\geq p^*(M):p\emph{ determines an initial segment of }\dot{s}\rb.
$$
Finally, define $\vp^M$ on $D(M)$ by
$$
\vp^M(p,\dot{d})=(p\res\al,\pi_M^{-1}(\dot{\ga}_\be)),
$$
where $\be<\al$ is the least so that $p\Vdash\dot{d}_{\dom(p(\al))}=^*\dot{\ga}_\be$ ($\be$ exists by definition of ``$p$ determines an initial segment of $\dot{s}$"). Then
\begin{enumerate}
\item[(a)]  $D(M)$ is a dense, countably $=^*$-closed subset of $\ps^*/p^*(M)$;
\item[(b)] $p^*(M)$ is compatible with every condition in $\ps^*\cap M$; and 
\item[(c)] $\vp^M$ is an exact, strong residue function from $D(M)$ to $M\cap \ps^*$.
\end{enumerate}
\end{proposition}
\begin{proof} 
Let $\la\dot{\ga}_i:i<\al\ra$ be a sequence of  $(\ps\res\al)$-names which is forced to enumerate all conditions in $\pi_M(\dot{\C})$, and with $\dot{d}^*$, satisfies Definition \ref{def:filtersfromcollapse}, witnessing that $\dot{s}=\la\dot{d}_\nu:\nu<\om_1\ra$ is guided by the collapse at $\al$ for $\pi_M(\dot{\C})$.

We first prove item (a). Given a condition $(p,\dot{d})$ in $\ps^*/p^*(M)$, by item (2) of Definition \ref{def:filtersfromcollapse}, we may find an extension $p'$ of $p$ so that $p$ determines an initial segment of $\dot{s}$. Then $(p',\dot{d})\geq (p,\dot{d})$ is in $D(M)$, proving density. Similarly, $D(M)$ is $=^*$-closed: if $(p_1,\dot{d}_1)\in D(M)$ and $(p_1,\dot{d}_1)=^*(p_2,\dot{d}_2)$, then $p_2$ determines an initial segment of $\dot{s}$, because $p_2=p_1$ (recall these are collapse conditions) and because $(p_2,\dot{d}_2)\geq(p_1,\dot{d}_1)\geq(0_\ps,\dot{d}^*)$. 

To see that $D(M)$ is closed under sups of increasing $\om$-sequences, suppose that $\la(p_n,\dot{c}_n):n\in\om\ra$ is an increasing sequence of conditions in $D(M)$, and let $(p^*,\dot{c}^*)$ be a sup. Set $\nu_n:=\dom(p_n(\al))$ and $\nu^*:=\dom(p^*(\al))$. If $\nu^*=\nu_m$ for some $m\in\om$, then because $p_m$ determines $\dot{s}$ up to $\nu_m$, we have that $p^*$ determines $\dot{s}$ up to $\nu^*=\nu_m$. Thus $p^*$ determines an initial segment of $\dot{s}$ in this case. So consider the case that $\nu^*>\nu_m$ for all $m$; in particular, $\nu^*$ is a limit ordinal. Since for all $n\in\om$, $p_n$ determines an initial segment of $\dot{s}$ and $p^*\geq p_n$, we may find a sequence $\la\be(\mu):\mu<\nu^*\ra$ in $V$ so that $p^*\Vdash\dot{d}_\mu=^*\dot{\ga}_{\be(\mu)}$ for all $\mu<\nu^*$. Now let $\be(\nu^*)$ be chosen so that $\dot{\ga}_{\be(\nu^*)}$ is forced to be a sup of $\la\dot{\ga}_{\be(\mu)}:\mu<\nu\ra$, if this sequence is increasing, and equals the trivial condition otherwise. Since $\nu^*$ is a limit, $\Vdash_{\ps}\dot{d}_{\nu^*}$ is a sup of $\la\dot{d}_\mu:\mu<\nu^*\ra$. But $p^*$ forces that $\dot{d}_\mu=^*\dot{\ga}_{\be(\mu)}$ for all $\mu<\nu^*$, and therefore $p^*$ forces that $\dot{d}_{\nu^*}=^*\dot{\ga}_{\be(\nu^*)}$. Thus, in either case, $p^*$ determines an initial segment of $\dot{s}$, which finishes the proof of (a).

Now we verify item (b).  
Fix a condition $(u,\dot{c}_0)$ in $\ps^*\cap M$, and we will show that it is compatible with $p^*(M)$. We observe that, trivially, $u$ determines an initial segment of $\dot{s}$ since $\dom(u(\al))=0$ and $\Vdash_\ps\dot{d}_0$ is the trivial condition in $\pi_M(\dot{\C})$, by (1) of Definition \ref{def:filtersfromcollapse}. By (3) of the same definition, we may find an extension $p\geq u$ s.t. $p\Vdash\dot{d}_1\geq\pi_M(\dot{c}_0)$. Then $(p,\dot{d}^*)$ extends $(u,\dot{c}_0)$ since $p$ forces that $\dot{d}^*$ is an upper bound for the sequence $\pi_M^{-1}[\dot{s}]$ and that $\pi^{-1}_M(\dot{d}_1)\geq\dot{c}_0$.

It therefore remains to verify that $\vp^M$ is an exact, strong residue function. Condition (1) of Definition \ref{def:ESRF} holds since, by (a), $D(M)$ is dense and countably $=^*$-closed in $\ps^*/p^*(M)$. For the projection condition of Definition \ref{def:ESRF}, fix $(p,\dot{c})\in D(M)$, and let $\dot{\ga}$ be the $(\ps\res\al)$-name so that $\vp^M(p,\dot{c})=(p\res\al,\pi_M^{-1}(\dot{\ga}))$. Since $(p,\dot{c})\in D(M)$, $p$ determines an initial segment of $\dot{s}$, and therefore $p\Vdash\dot{d}_{\dom(p(\al))}=^*\dot{\ga}$. Since $p$ also forces that $\dot{c}\geq\dot{d}^*$, $p$ forces that $\dot{c}$ is an upper bound for $\pi_M^{-1}[\dot{s}]$ and therefore that $\dot{c}$ extends $\pi_M^{-1}(\dot{d}_{\dom(p(\al))})=^*\pi_M^{-1}(\dot{\ga})$. Therefore $(p,\dot{c})\geq(p\res\al,\pi_M^{-1}(\dot{\ga}))=\vp^M(p,\dot{c})$.

{It is straightforward to verify that $\vp^M$ is order preserving, i.e., condition (3) of Definition \ref{def:ESRF}.} So we prove that $\vp^M$ has the strong residue property (condition (4) of Definition \ref{def:ESRF}). Thus fix $(p,\dot{c})\in D(M)$, where we let $\nu:=\dom(p(\al))$ and $\dot{\ga}$ so that $\vp^M(p,\dot{c})=(p\res\al,\pi_M^{-1}(\dot{\ga}))$. Fix a condition $(u,\dot{\de})$ in $\ps^*\cap M$ with $(u,\dot{\de})\geq(p\res\al,\pi_M^{-1}(\dot{\ga}))$, and we will verify that $(u,\dot{\de})$ is compatible with $(p,\dot{c})$. Let $p':=u\cup p$, a condition in $\ps$, and observe that $p'$ still determines an initial segment of $\dot{s}$ and $\nu=\dom(p'(\al))$. By item (3) of Definition \ref{def:filtersfromcollapse}, we may find some $p^*\geq p'$ so that $p^*\Vdash\dot{d}_{\nu+1}\geq\pi_M(\dot{\de})$. Then $(p^*,\dot{c})$ extends both $(p,\dot{c})$ and $(u,\dot{\de})$.

We now check that $\vp^M$ is $\om$-continuous, which will finish the proof of (c) and thereby the proof of the proposition. Fix an increasing sequence of conditions $\la(p_n,\dot{c}_n):n\in\om\ra$ in $D(M)$, and let $(p^*,\dot{c}^*)$ be a supremum of this sequence. Then $p^*:=\bigcup_np_n$, and $\dot{c}^*$ is forced by $p^*$ to be a sup of $\la\dot{c}_n:n\in\om\ra$. By item (a) of the proposition, $(p^*,\dot{c}^*)\in D(M)$. We need to show that $\vp^M(p^*,\dot{c}^*)$ is a sup of $\la\vp^M(p_n,\dot{c}_n):n\in\om\ra$. For each $n<\om$, set $\nu_n:=\dom(p_n(\al))$, and also set $\nu^*:=\dom(p^*(\al))$. Additionally, for each $n$, fix the least ordinal $\be(\nu_n)$ with $p_n\Vdash\dot{d}_{\nu_n}=^*\dot{\ga}_{\be(\nu_n)}$, so that $\vp^M(p_n,\dot{c}_n)=(p_n\res\al,\pi_M^{-1}(\dot{\ga}_{\be(\nu_n)}))$. Finally let $\be(\nu^*)$ be least so that $\vp^M(p^*,\dot{c}^*)=(p^*\res\al,\pi_M^{-1}(\dot{\ga}_{\be(\nu^*)}))$. 

We claim that $p^*\Vdash\dot{\ga}_{\be(\nu^*)}$ is a sup of $\la\dot{\ga}_{\be(\nu_n)}:n\in\om\ra$. Note that proving this claim suffices: indeed, then $p^*\res\al$ forces that $\dot{\ga}_{\be(\nu^*)}$ is a sup of $\la\dot{\ga}_{\be(\nu_n)}:n\in\om\ra$, and as a result $\vp^M(p^*,\dot{c}^*)=(p^*\res\al,\pi_M^{-1}(\dot{\ga}_{\be(\nu^*)}))$ is a sup of $\la\vp^M(p_n,\dot{c}_n):n\in\om\ra$.

To prove the claim, we have two cases on $\nu^*$. Since $\nu^*=\sup_m\nu_m$, either $\nu^*>\nu_m$ for all $m$, or $\nu^*=\nu_m$ for almost all $m$. In the first case, $\nu^*$ is a limit, and so $\ps$ forces that $\dot{d}_{\nu^*}$ is a sup of $\la\dot{d}_\nu:\nu<\nu^*\ra$. Since $p_n\Vdash\dot{d}_{\nu_n}=^*\dot{\ga}_{\be(\nu_n)}$ for each $n$, $p^*$ forces that $\la\dot{\ga}_{\be(\nu_n)}:n\in\om\ra$ is cofinal in $\la\dot{d}_\nu:\nu<\nu^*\ra$. Therefore $p^*$ forces that these two sequences have the same sups. Consequently, $p^*$ forces that $\dot{\ga}_{\be(\nu^*)}=^*\dot{d}_{\nu^*}$ is a sup of $\la\dot{\ga}_{\be(\nu_n)}:n\in\om\ra$. 
For the second case, $\nu^*=\nu_m$ for all $m$ above some $k$. Then $\ps$ forces that $\la\dot{d}_{\nu_n}:n\in\om\ra$ is eventually equal to $\dot{d}_{\nu^*}$. Because $p_n\Vdash\dot{d}_{\nu_n}=^*\dot{\ga}_{\be(\nu_n)}$ for all $n$ and $p^*\geq p_n$, we have that $p^*\Vdash\la\dot{\ga}_{\be(\nu_n)}:n\in\om\ra$ is eventually equal to $\dot{d}_{\nu^*}$. Finally, $p^*\Vdash\dot{\ga}_{\be(\nu^*)}=^*\dot{d}_{\nu^*}$, and therefore $p^*$ forces that $\la\dot{\ga}_{\be(\nu_n)}:n\in\om\ra$ is eventually $=^*$-equal to $\dot{\ga}_{\be(\nu^*)}$, and therefore that $\dot{\ga}_{\be(\nu^*)}$ is a sup. This finishes the proof of the claim and thereby the proof that $\vp^M$ is $\om$-continuous.
\end{proof}

The final result in this subsection shows that we can create filters which are guided by the collapse at $\al$ by using the generic surjection from $\om_1$ onto $\al$ to guide extensions in the second coordinate. This combines ideas of collapse absorption with, as previously mentioned, Abraham's use of guiding reals.

\begin{lemma}\label{lemma:gettingthefilters} Suppose that $\dot{\Q}$ is a $(\ps\res\al)$-name for a poset of size $\al$, which is $\om_1$-closed with sups, and let $\la\dot{\ga}_i:i<\al\ra$ be forced to enumerate all  conditions in $\dot{\Q}$. Let $\dot{f}_\al$ be the $\ps$-name for the standard surjection added from $\om_1$ onto $\al$.

Suppose that $\dot{s}=\la\dot{d}_\nu:\nu<\om_1\ra$ is a sequence of $\ps$-names for conditions in $\dot{\Q}$ forced by $\ps$ to satisfy the following properties:
\begin{enumerate}
\item $\dot{s}$ is $\leq_{\dot{\Q}}$-increasing, $\dot{d}_0$ names the trivial condition, and if $\nu$ is a limit, then $\dot{d}_\nu$ is a $\leq_{\dot{\Q}}$-sup of $\la\dot{d}_\mu:\mu<\nu\ra$;
\item if $\nu<\om_1$ and $\dot{\ga}_{\dot{f}_\al(\nu)}$ extends $\dot{d}_\nu$ in $\dot{\Q}$, then $\dot{d}_{\nu+1}$ extends $\dot{\ga}_{\dot{f}_\al(\nu)}$;
\item for each $\nu<\om_1$, the sequence $\la\dot{d}_\mu:\mu\leq\nu\ra$ is definable in $V[\dot{G}\res\al]$ from $\dot{f}_\al\res\nu$. 
\end{enumerate}
Then $\dot{s}$ is guided by the collapse at $\al$ for $\dot{\Q}$.

In particular, there exists a $\ps$-name for a sequence which is guided by the collapse at $\al$ for $\dot{\Q}$.
\end{lemma}
\begin{proof} 

We will verify that items (1)-(3) of Definition \ref{def:filtersfromcollapse} hold. Item (1) of the definition is immediate from assumption (1) of the lemma.

For item (2) of Definition \ref{def:filtersfromcollapse}, suppose that $p\in\ps$ is a condition where $\dom(p(\al))$ is an ordinal $\nu<\om_1$. Let $G$ be $V$-generic over $\ps$ containing $p$, and let $\bar{G}:=G\cap (\ps\res\al)$. For each $\mu\leq\nu$, let $d_\mu:=\dot{d}_\mu[G]$, and let $\be(\mu)$ be an ordinal $<\al$ so that $d_\mu=\dot{\ga}_{\be(\mu)}[\bar{G}]$. By assumption (3) of the lemma, the sequence $\la d_\mu:\mu\leq\nu\ra$ is definable in $V[\bar{G}]$ from $f_\al\res\nu=p(\al)$. Therefore, there exists a condition $\bar{p}\geq p\res\al$ with $\bar{p}\in\bar{G}$ so that $\bar{p}$ forces that $\la\dot{\ga}_{\be(\mu)}:\mu\leq\nu\ra$ satisfies the definition with respect to $p(\al)$. Then $p':=p\cup\bar{p}$ witnesses item (2) of Definition \ref{def:filtersfromcollapse}, since $p'$ forces that $\la\dot{\ga}_{\be(\mu)}:\mu\leq\nu\ra$ and $\la\dot{d}_\mu:\mu\leq\nu\ra$ both satisfy the same definition in $V[\dot{\bar{G}}]$ with the parameter $\dot{f}_\al\res\nu=p(\al)$.

Turning to item (3) of Definition \ref{def:filtersfromcollapse}, let $p'$ be a condition as in the previous paragraph. Fix a $(\ps\res\al)$-name $\dot{\ga}$ for a $\dot{\Q}$-extension of $\dot{\ga}_{\be(\nu)}$, {and let $\de<\al$ and $\bar{p}^*\geq p'\res\al$ so that $\bar{p}^*\Vdash_{\ps\res\al}\dot{\ga}=\dot{\ga}_\de$.} Define $p^*$ to be the minimal extension of $\bar{p}^*\cup p'\res[\al,\ka)$ so that $p^*\Vdash\dot{f}_\al(\nu)=\de$. Then $p^*\Vdash\dot{\ga}_{\dot{f}_\al(\nu)}\geq\dot{\ga}_{\be(\nu)}=^*\dot{d}_\nu$, so there exists an extension $p^{**}$ of $p^*$ so that $p^{**}$ forces $\dot{d}_{\nu+1}\geq\dot{\ga}_{\dot{f}_\al(\mu)}=\dot{\ga}$. 

{For the ``in particular" claim of the lemma}, define a sequence $\dot{s}$ by recursion so that it satisfies (1) and so that if $\nu<\om_1$, then $\dot{d}_{\nu+1}$ is forced to be equal to $\ga_{\dot{f}_\al(\nu)}$ if this extends $\dot{d}_\nu$, and otherwise equals $\dot{d}_\nu$. Then (2) and (3) are also satisfied, so $\dot{s}$ is guided by the collapse at $\al$ for $\dot{\Q}$.
\end{proof}

\subsection{${\cal{F}}$-Complete Properness}

We are now ready to isolate a sufficient condition on names $\dot{\C}$ so that $\ps\ast\dot{\C}$ is ${\cal{F}}$-strongly proper.

\begin{definition}\label{def:fkcp} Let $\dot{\C}$ be a $\ps$-name in $H(\ka^+)$ for a poset forced to be $\om_1$-closed with sups.  We say that $\dot{\C}$ is ${\cal{F}}$-\textbf{Completely Proper} if for any $(\ps\ast\dot{\C})$-suitable sequence $\vec{M}$ {there is some $A\seq\dom(\vec{M})$ with $A\in\cal{I}$ so that for each $\al\in \dom(\vec{M})\bsl A$ and each} $\ps$-name $\dot{H}$ for filter over $\pi_{M_\al}(\dot{\C})$ which is guided by the collapse at $\al$, there exists a $\ps$-name $\dot{c}_{\dot{H}}$ for a condition in $\dot{\C}$ which is forced to be a least upper bound for $\pi^{-1}_{M_\alpha[\dot{G}_{\ps}]}[\dot{H}]$.
\end{definition}

We recall that a poset $\bb{U}$ is \emph{$\lam$-distributive} if forcing with $\bb{U}$ adds no sequences of ordinals of length less than $\lam$. {A sufficient condition to guarantee this is that the intersection of fewer than $\lam$-many dense, open subsets of $\bb{U}$ is dense, open. They are equivalent if $\bb{U}$ is separative.}

\begin{lemma}\label{lemma:cpimpliesdist} Suppose that $\dot{\C}$ is ${\cal{F}}$-completely proper. Then $\ps$ forces that {the intersection of fewer than $\ka$-many dense, open subsets of $\dot{\C}$ is dense, open. Hence} $\dot{\C}$ is forced to be $\ka$-distributive.
\end{lemma}
\begin{proof} Let $\vec{D}=\la\dot{D}_i:i<\om_1\ra$ be a sequence of $\ps$-names for dense, open subsets of $\dot{\C}$. We show that the intersection is forced to be non-empty. Fix a sequence $\vec{M}$ which is suitable with respect to $\ps\ast\dot{\C}$ and $\vec{D}$ so that $\dom(\vec{M})$  satisfies the conclusion of Definition 
\ref{def:fkcp}. Let $\al\in\dom(\vec{M})$. By Lemma \ref{lemma:gettingthefilters}, there exists a $\ps$-name $\dot{H}$ for a filter which is guided by the collapse at $\al$ for $\pi_{M_\al}(\dot{\C})$. Since $\dom(\vec{M})$ satisfies Definition \ref{def:fkcp}, we may find a $\ps$-name $\dot{d}$ for a condition which is forced to be an upper bound in $\dot{\C}$ for $\pi^{-1}_{M_\al[\name{G}_{\po}]}[\dot{H}]$. Then, by Lemma \ref{lemma:getCGConditions}, $\dot{d}$ is forced to be an $(M_\al[\dot{G}],\dot{\C})$-\emph{completely} generic condition. But $\dot{D}_i\in M_\al$ for each $i<\om_1$ and is dense, open. Therefore it is forced that $\dot{d}\in\bigcap_{i\in\om_1}\dot{D}_i$.
\end{proof}

By combining Lemma \ref{lemma:gettingthefilters}, Proposition \ref{prop:gettingESRFs}, and Lemma \ref{lemma:refineToGetSP}, we conclude {with} the following key result:

\begin{proposition}\label{prop:2stepSP} Suppose that $\dot{\C}$ is ${\cal{F}}$-completely proper. Then $\ps\ast\dot{\C}$ is ${\cal{F}}$-strongly proper.
\end{proposition}

\section{Properties of the Club-Adding Poset}\label{section:clubscompletelyproper}

In this section, we have two main tasks. In the first subsection, we will prove that our intended club adding iteration, as well as useful variants thereof, are ${\cal{F}}$-completely proper, and in the second subsection, we will prove that our intended club adding iteration does not add branches through various Aronszajn trees. Each of these results will be used as part of a larger inductive argument in the final section in which we prove Theorem \ref{thm:main}. {Again we comment that, in this case, $\ka$ is ineffable, but we are only using the ineffability of $\ka$ when we apply Proposition \ref{prop:statsetpreservation}, since this requires Proposition \ref{prop:InductiveIandII}.}

\subsection{Adding Clubs is ${\cal{F}}$-Completely Proper}

In order to anticipate arguments in the next subsection, where we show that appropriate ${\cal{F}}$-completely proper posets do not add branches through certain Aronszajn trees, we will need to not only show that our club adding poset is ${\cal{F}}$-completely proper, but also show that variants of it have this property. These variants are created by iterating the process of taking an initial segment of the iteration followed by products of finitely-many copies of the tail.

The following iteration follows Magidor's work \cite{Magidor} on adding clubs through reflection points of stationary subsets of a weakly compact cardinal $\kappa$, which has been collapsed to become $\omega_2$. 

Let $\rho<\ka^+$, and suppose that we have defined a $\ps$-name for an iteration $\la\dot{\C}_\si,\dot{\C}(\eta):\si\leq\rho,\eta<\rho\ra$ and a $(\ps\ast\dot{\C}_\rho)$-name $\la\dot{\S}_\si,\dot{\S}(\eta):\si\leq\rho,\eta<\rho\ra$ for an iteration  so that for all $\si<\rho$ the following assumptions are satisfied:
\begin{enumerate}
\item $\ps$ forces that the iteration $\la\dot{\C}_\si,\dot{\C}(\eta):\si\leq\rho,\eta<\rho\ra$ has $<\kappa$-support, and $\ps\ast\dot{\C}_\rho$ forces that $\la\dot{\S}_\si,\dot{\S}(\eta):\si\leq\rho,\eta<\rho\ra$ has countable support. Furthermore, $\dot{\S}_\si$ is a $(\ps\ast\dot{\C}_\si)$-name;
\item $\dot{\C}(\si)$ is a $(\ps\ast\dot{\C}_\si)$-name for $\mathsf{CU}(\dot{S}_\si,\dot{\S}_\si)$ (see Definition \ref{def:clubaddingposet}), where $\dot{S}_\si$ is a $(\ps\ast\dot{\C}_\si\ast\dot{\S}_\si)$-name for a stationary subset of $\ka\cap\cof(\om)$ and $\dot{\C}_{\si+1}=\dot{\C}_\si\ast\dot{\C}(\si)$;
\item $\dot{\S}(\si)$ is a $(\ps\ast\dot{\C}_{\si+1}\ast\dot{\S}_\si)$-name for $\S(\dot{T}_\si)$ (see Definition \ref{def:specposet}), where $\dot{T}_\si$ is a $(\ps\ast\dot{\C}_{\si+1}\ast\dot{\S}_\si)$-name for an Aronszajn tree on $\ka$;
\item $\dot{\C}_\si$ is ${\cal{F}}$-completely proper (and hence $\ps\ast\dot{\C}_\si$ is $\cal{F}$-strongly proper, by Proposition \ref{prop:2stepSP}), and $\ps\ast\dot{\C}_\si$ forces that $\dot{\S}_\si$ is a countable support iteration specializing Aronszajn trees, as defined in Section \ref{section:specialize}.
\end{enumerate}

Working in an arbitrary generic extension by $\ps$, we now define the variations of $\dot{\C}_\rho$ mentioned above; we call these \emph{Doubling Tail Products}. These will be the posets $\C_\rho(\vec{\delta})$, 
where $\vec{\delta} = \la \delta_0,\delta_1,\dots,\delta_{n-1}\ra \in [\rho]^n$ is a strictly decreasing sequence of ordinals. We use $[\rho]^{<\om}_\text{dec}$ to denote the set of all finite, strictly decreasing tuples from $\rho$; $[\rho]^n_\text{dec}$ is defined similarly.

We first introduce an auxiliary name for a poset $\dot\C_{\delta^*,\rho}(\vec{\delta})$, where  $\vec{\delta}$ is as above and $\de^*\leq\de_{n-1}$ is an additional ordinal.
This is done by recursion on $n = |\vec{\delta}|$ as follows:
\begin{itemize}
    \item For $n = 0$ (i.e., $\vec{\delta} = \emptyset$) we define $\dot\C_{\delta^*,\rho}(\emptyset) = \dot\C_{\delta^*,\rho}$ to be the tail-segment of the iteration $\C_\rho$, starting from stage $\delta^*$.
    \item For $n \geq 1$, $\vec{\delta} \in [\rho]^n$, and $\delta^* \leq \delta_{n-1}$, the poset $\C_{\delta^*,\rho}(\vec{\delta})$ is given by- 
     \[\dot\C_{\delta^*,\rho}(\vec{\delta})  = \dot\C_{\delta^*,\delta_{n-1}} * (\dot\C_{\delta_{n-1},\rho}(\vec{\delta}\uhr n-1))^2,\]
     where $\dot\C_{\delta^*,\delta_{n-1}}$ is the segment of the iteration $\C_\rho$ starting from (and including) stage $\delta^*$ to stage $\delta_{n-1}$, 
     and $\vec{\delta}\uhr (n-1) = \la \delta_0,\dots,\delta_{n-2} \ra$. 
\end{itemize}

We can now define $\C_{\rho}(\vec{\delta})$.
\begin{definition}
For $\vec{\delta} \in [\rho]^{<\omega}_\text{dec}$, define $\C_\rho(\vec{\delta}) = \C_{0,\rho}(\vec{\delta})$ (i.e., as $\C_{\delta^*,\rho}(\vec{\delta})$ with $\delta^* =0$).
\end{definition}

For example, if $\vec{\delta} = \la \delta_0\ra$ is a singleton, then $\C_{\rho}(\la \delta_0 \ra) = \C_{\delta_0}* \dot\C_{\delta_0,\rho}^2$.
Similarly, if $\vec{\delta} = \la \delta_0,\delta_1\ra$ has two elements $\delta_0 > \delta_1$ then 
\[
\C_\rho(\la \delta_0,\delta_1\ra) = \C_{\delta_1} * \dot\C_{\delta_1,\rho}^2(\la \delta_0\ra) = 
\C_{\delta_1} * \left( \dot\C_{\delta_1,\delta_0} * \dot\C^2_{\delta_0,\rho}\right)^2
\]

We refer to posets $\C_\rho(\vec{\delta})$ as the \textbf{ doubling tail products} of $\C_\rho$. We now return to working in $V$, in particular with the statement of the next item.

\begin{proposition}\label{prop:clubsareCP} 
(Given assumptions (1)-(4) stated at the beginning of the subsection) For each $\rho < \kappa^+$ and $\vec{\delta} = \la \delta_{0},\dots,\delta_{n-1}\ra \in [\rho]^{<\omega}_\text{dec}$,  the doubling tail product  $\dot\C_\rho(\vec{\delta})$ is ${\cal{F}}$-completely proper. 
In particular, $\dot{\C}_\rho$ is ${\cal{F}}$-completely proper. 
\end{proposition}

We will explain the necessity of proving Proposition \ref{prop:clubsareCP} for the doubling tail products of $\dot\C_\rho$ in the final section of the paper.

\begin{proof}

We will first work with $\dot{\C}_\rho$, rather than with the doubling tail products, in order to establish that a certain statement $(*)$ (see below) holds. We will then show that this statement $(*)$ can be used to prove the desired result for the doubling tail products. We use $\R_\si$, for $\si\leq\rho$, to denote $\ps\ast\dot\C_\si*\dot{\S}_\si$.

To begin, we fix an $\R_\rho$-suitable sequence $\vec{M}$; by removing a set in $\cal{I}$, we may assume that $\vec{M}$ is in pre-splitting configuration up to $\rho$. {By removing a further $\cal{I}$-null set, we may assume that $\vec{M}$ satisfies the conclusion of Proposition \ref{prop:statsetpreservation}.}

Let $M$ be a $\ka$-model so that $M$ contains the relevant parameters, including $\vec{M}$. {Since $\dom(\vec{M})$ is in $\cal{F}^+$, we may apply Proposition \ref{proposition:characterizefwc} to find an $M$-normal ultrafilter $U$ so that, letting $j:M\lra N$ be the corresponding ultrapower map, $\ka\in j(\dom(\vec{M}))$. Let $M_\ka$ be the $\ka$-th model on the sequence $j(\vec{M})$.}


 Fix a $V$-generic filter $G^*$ over $j(\ps)$, and let $G:=G^*\cap\ps$. For notational simplicity, we continue using $j$ to denote the lifted map $j:M[G]\lra N[G^*]$. {Recall by Lemma \ref{lemma:UsefulSuitable} that $j^{-1}\res M_\ka$ is the transitive collapse map of $M_\ka$ and that $j^{-1}$ lifts in the standard way to $M_\ka[G^*]$.}

 Suppose that $\si\leq\rho$ and that $\dot{H}$ is a $j(\ps)$-name in $N$ for a generic filter over
 $\pi_{M_\ka[\dot{G}^*]}(j(\dot{\C}_\si))=\dot\C_\si$. We define the \textbf{$\ka$-flat function for (the pull-back of) $\dot{H}$} to be the $j(\ps)$-name for the function with domain $j[\si]$,\footnote{{We note that $j[\si]\in N$: $M_\ka\cap j(\rho)=j[\rho]$, and so $j[\rho]$ is in $N$. Then intersect with $j(\si)$.}} so that for each $\eta<\si$, $\dot{r}(j(\eta))$ is forced to be equal to $\left(\bigcup\dot{H}(\eta)\right)\cup\lb\ka\rb$.

We will prove the following proposition $(*)$ by induction:
\begin{enumerate}
\item[$(*)$] for any $\si\leq\rho$, if $\dot{H}$ is a $j(\ps)$-name in $N$ for a generic filter over $\dot\C_\si=\pi_{M_\ka[G^*]}(j(\dot{\C}_\si))$ which is guided by the collapse at $\ka$ (see Definition \ref{def:genbycoll}), then it is forced by $j(\ps)$ that the $\ka$-flat function for $\dot{H}$ is a condition in $j(\dot{\C}_\si)$.
\end{enumerate}

We first consider the case that $\si\leq\rho$ is limit. Suppose that we know the result for all $\eta<\si$. We use throughout the fact that $j^{-1}$ equals the transitive collapse map of $M_\ka[G^*]$.

Let $H\in N[G^*]$ be a filter over $\C_\si$ which is guided by the collapse at $\ka$, and let $r$ be the $\ka$-flat function for $H$. Since $|\dom(r)|^N<j(\ka)$ and $j(\C_\si)$ is taken with $<j(\ka)$-supports, in order to see that $r\in j(\C_\si)$, it suffices to show that for all $\eta<\si$, $r\res j(\eta)\in j(\C_\eta).$
So let $\eta<\si$ be fixed. Since $\C_\si\cong\C_\eta\ast\dot{\C}_{{\eta},\si}$ and since $H$ is guided by the collapse at $\ka$ over $\C_\si$, 
we have  by Lemma \ref{lemma:projectcollapsefilter} that $H\res \C_\eta$ is also guided by the collapse at $\ka$ over $\C_\eta$. By induction, this implies that the $\ka$-flat {function} for $H\res \C_\eta$, namely $r\res j(\eta)$, is a condition in $j(\C_\eta)$. This completes the proof of $(*)$ in the limit case.

Now suppose that $\si+1\leq\rho$ and that we know that $(*)$ holds at $\si$. Let $H\in N[G^*]$ be a filter over $\C_{\si+1}$ which is guided by the collapse at $\ka$, and let $H_\si$ denote the restriction of $H$ to $\C_\si$. Again appealing to Lemma \ref{lemma:projectcollapsefilter}, we know that $H_\si$ is guided by the collapse at $\ka$. 

Let $r$ be the $\ka$-flat function for $H$, and let $\bar{r}$ denote $r\res j(\si)$, the $\ka$-flat function for $H_\si$. Since $H_\si$ is guided by the collapse at $\ka$, we may apply the induction hypothesis to conclude that $\bar{r}$ is a condition in $j(\C_\si)$. By Proposition \ref{prop:gettingESRFs}, since $H_\si$ is guided by the collapse at $\ka$, we know that in $N$ we may find a residue pair $\la(0_{j(\ps)},\dot{\bar{r}}),\vp^{M_\ka}\ra$ for the pair $(M_\ka,j(\ps^*))$, where $\dot{\bar{r}}$ is a $j(\ps)$-name in $N$ for $\bar{r}$. We use $p^*(M_\ka)$ to denote $(0_{j(\ps)},\dot{\bar{r}})$.

Since $\bar{r}$ is an upper bound for $ \pi_{M_\ka[G^*]}^{-1}[H_\si]=j[H_\si]$, we conclude that $\bar{r}$ forces in $j(\C_\si)$ over $N[G^*]$ that $\bigcup j[H(\si)]{=\bigcup H(\si)}$ is {club} in $\ka$ {(equality follows since $j$ is the identity on bounded subsets of $\ka$)}. Therefore, to see that $r\in j(\C_{\si+1})$ (which finishes the proof of $(*)$ in the successor case), it suffices to show that
$$
\bar{r}\Vdash^{N[G^*]}_{j(\C_\si)}\left(\bigcup H(\si)\cup\lb\ka\rb\right)\in j(\dot{\C}(\si)).
$$
Since $j(\dot{\C}(\si))$ is a $j(\ps\ast\dot{\C}_\si)$-name for $\mathsf{CU}(j(\dot{S}_\si),j(\dot{\S}_\si))$, the above holds if and only if 
$$
(\bar{r},0_{j(\dot{\S}_\si)})\Vdash^{N[G^*]}_{j(\C_\si\ast\dot{\S}_\si)}\left(\bigcup H(\si)\cup\lb\ka\rb\right)\seq   \left( \tr\left(j(\dot{S}_\si)\right)\cup \left(j(\ka)\cap\cof(\om)\right)\right).
$$

By the elementarity of $j$ and since $\bar{r}$ is an upper bound for $j[H_\si]$, we see that 
$$
(\bar{r},0_{j(\dot{\S}_\si)})\Vdash^{N[G^*]}_{j(\C_\si\ast\dot{\S}_\si)}\bigcup H(\si) \seq\left(\tr\left(j(\dot{S}_\si)\right)\cup \left(j(\ka)\cap\cof(\om)\right)\right).
$$
Since $\ka$ has cofinality $\om_1$ after forcing with $j(\C_\si\ast\dot{\S}_\si)$, it therefore suffices to show that
$$
(\bar{r},0_{j(\dot{\S}_\si)})\Vdash^{N[G^*]}_{j(\C_\si\ast\dot{\S}_\si)} (j(\dot{S}_\si)\cap\ka)\text{ is stationary in }\ka.
$$

Before continuing, we recall that $\R_\si$ denotes $\ps\ast\dot{\C}_\si\ast\dot{\S}_\si$. By Proposition \ref{proposition:ccreflect}, we know that 
$$
(p^*(M_\ka),0_{j(\dot{\S}_\si)})\Vdash^N_{j(\R_\si)} j(\dot{S}_\si)\cap\ka=(j(\dot{S}_\si)\cap M_\ka)[\dot{G}_{j(\R_\si)}\cap M_\ka].
$$
However, $j(\R_\si)\cap M_\ka=j[\R_\si]$ is isomorphic to $\R_\si$, and $j(\dot{S}_\si)\cap M_\ka=j[\dot{S}_\si]$. Since $\dot{S}_\si$ is a nice $\R_\si$-name for a stationary subset of $\ka\cap\cof(\om)$, $j[\dot{S}_\si]$ is therefore a $j(\R_\si)\cap M_\ka$-name for a stationary subset of $\ka\cap\cof(\om)$. 

We now apply Proposition \ref{prop:statsetpreservation} in $N$, recalling that $\vec{M}$ satisfies the conclusion of that proposition and that $\ka\in j(\dom(\vec{M}))$. Thus, 
applying the observations in the previous paragraph, we conclude that $(p^*(M_\ka),0_{j(\dot{\S}_\si)})$ forces over $N$ that $j[\dot{S}_\si]$ is stationary in $\ka$. Recalling that $p^*(M_\ka)=(0_{j(\ps)},\dot{\bar{r}})$, we now conclude that
$$
(\bar{r},0_{j(\dot{\S}_\si)})\Vdash^{N[G^*]}_{j(\C_\si\ast\dot{\S}_\si)} j(\dot{S}_\si)\cap\ka=j[\dot{S}_\si][\dot{G}_{j(\R_\si)}\cap M_\ka]\text{ is stationary}.
$$
This completes the proof that $r$, the $\ka$-flat function for $H$, is a condition in $j(\C_\si)$ and also finishes the proof that $(*)$ holds.\\

To finish the proof of Proposition \ref{prop:clubsareCP}, we prove by induction on $k<\om$ that
for any $\vec{\delta} = \la \delta_0,\dots,\delta_{k-1}\ra \in [\rho]^k_\text{dec}$,
the poset $\C_\rho(\vec{\delta})$ is $\cal{F}$-completely proper. Working by contradiction, let $k\in\om$ be the least so that for some (empty in case $k=0$) $\vec{\de}\in [\rho]^k_\text{dec}$, the proposition fails for $\C_\rho(\vec{\de})$. Let $\vec{M}$ be a $\vec{\de}$-suitable sequence which witnesses that $\C_\rho(\vec{\de})$ is not $\cal{F}$-suitable. Since $\vec{M}$ is $\vec{\de}$-suitable it is also $\ps\ast\dot{\C}_\rho(\vec{\delta})$-suitable. {By removing a $\cal{I}$-null set, we may assume that $\vec{M}$ satisfies the conclusion of Proposition \ref{prop:statsetpreservation}.}

Let $M$ be a $\ka$-model containing the relevant parameters, including $\vec{M}$. Since $\dom(\vec{M})\in{\cal{F}^+}$, we may find some $M$-normal ultrafilter $U$ so that, letting $j:M\lra N$ be the ultrapower embedding, $\ka\in j(\dom(\vec{M}))$.

{First we deal with the case $k=0$. Since $\vec{M}$ witnesses that $\dot{\C}_\rho$ is not $\cal{F}$-completely proper, $N$ satisfies that the conclusion of Definition \ref{def:fkcp} fails at $\ka$ with respect to $j(\ps)$ and $j(\dot{\C}_\rho)$. However, this directly contradicts $(*)$, which we showed holds in this set-up.}

{Now we assume that $k=l+1$ is a successor. Then we may find an $N$-generic filter $G^*$ over $j(\ps)$ so that in $N[G^*]$ there is a filter $H^*$ over $\pi_{{M_\ka}[G^*]}(j(\C_\rho(\vec{\de})))=\C_\rho(\vec{\de})$ so that no condition in $j(\C_\rho(\vec{\de}))$ is a least upper bound for $\pi^{-1}_{M_\ka[G^*]}[H^*]$. We will show, on the contrary, that there is such a least upper bound for the pull-back of $H^*$.}

{Write $\vec{\de}=\la\de_0,\dots,\de_{k-1}\ra$. For simplicity of notation, we also write $\C_\rho(\vec{\de}\res k-1)=\C_{\de_{k-2}}\ast\dot{\bb{D}}$ so that  $\C_\rho(\vec{\delta})=\C_{\delta_{k-1}}\ast\left(\dot{\C}_{[\delta_{k-1},\de_{k-2})}\ast\dot{\bb{D}}\right)^{2}$. Note in the case $k=1$, we just have $\C_\rho(\vec{\de})=\C_{\de_0}\ast(\dot{\C}_{[\de_0,\rho)})^2$.
The filter $H^*$ adds generics $J_0,J_1$ over $\C_{\rho}(\vec{\delta}\uhr k-1)$ so that 
$J_{0}$ and $J_{1}$ agree on $\C_{\delta_{k-1}}$ (recall that $\de_{k-1}<\de_{k-2}$) but are mutually generic afterwards. Since $H^*$ is guided by the collapse at $\ka$, both $J_0$ and $J_1$ are guided by the collapse at $\ka$. 
Hence, our inductive assumption implies that for each $i\in 2$, $\pi^{-1}_{M_\ka[G^*]}[J_i]$ has a sup $r_i$ in $j(\C_\rho(\vec{\delta}\uhr k-1))$. By the agreement between $J_0$ and $J_1$ up to stage $\delta_{k-1}<\de_{k-2}$, we know that $r_{0}\res j(\C_{\delta_{k-1}})=r_{1}\res j(\C_{\delta_{k-1}})$. We let $\bar{r}$ denote the common value. Finally, we define $r^*$ to be the function $\bar{r}\,^\frown\la r_i\res j(\C_\rho(\vec{\delta}\uhr k-1)):i<2\ra$. Then $r^*$ is a condition in $j(\C_\rho(\vec{\delta}))$ which is a sup of $\pi^{-1}_{M_\ka[G^*]}[H^*]$.}

This completes the inductive step and the proof of the proposition.

\end{proof}

\subsection{No New Branches}\label{ss:nnb}

In this subsection, we will show that various ${\cal{F}}$-completely proper posets do not add branches through Aronszajn trees of interest. We will use this general result to show, in particular, that {tails of} the club adding poset $\C_\rho$ {do} not add any branches to trees $\dot{T}$ which are Aronszajn trees in an intermediate extension obtained by forcing with $\C_\si{\ast\dot{\S}_\si}$, for $\si<\rho$. This will ensure that the tree specializing iteration of length $\rho$ above is in fact an iteration of specializing \emph{Aronszajn} trees in the extension by $\C_\rho$, a conclusion which is essential in order to see that the specializing iteration does not collapse $\ka$.

Arguments for securing that certain posets do not add new cofinal branches to trees play a crucial role in consistency results concerning the tree property, going back to the work of Mitchell and Silver (\cite{MitchellUgh}), and Magidor and Shelah (\cite{MagidorShelah}). Lemma 6 of Unger \cite{Unger} provides such an argument with respect to closed posets and trees named by posets with reasonable chain condition, given constraints on the continuum function.
Here, we prove a version of these results, in which the relevant posets (which in practice are variants of the club-adding poset) are
$\cal{F}$-completely proper (and thus $\ka$-distributive) but not $\kappa$-closed.

The statement of the following Proposition involves (names of) posets $\dot{\Q}_1,\dot{\Q}_2$, and $\dot{\S}$. 
To relate the statement to our scenario, we suggest keeping in mind the following assignments of the posets: Fixing $\rho <\rho^*< \kappa^+$, $\dot{\Q}_1 = \dot\C_\rho$ is the $(\ps$-name) of the first $\rho$ steps of the club adding iteration, $\dot{\S} = \dot\S_\rho$ is the $\ps* \dot\C_\rho$-name of the first $\rho$ steps of the iteration specializing trees, and  $\dot{\Q}_2 = \dot\C_{[\rho,\rho^*)}$ is the $\ps* \dot\C_\rho$-name of the segment of the final iteration from (and including) stage $\rho$ to stage $\rho^*$ (i.e. $\dot{\Q}_1 * \dot{\Q}_2 = \dot\C_{\rho^*}$).

\begin{proposition}\label{prop:NNB} Suppose that $\dot{\Q}_1$ is a $\ps$-name and that $\dot{\Q}_2$ and $\dot{\S}$ are $(\ps\ast\dot{\Q}_1)$-names so that $\dot{\Q}_1\ast\dot{\Q}^2_2$ is ${\cal{F}}$-completely proper and so that $\ps\ast\dot{\Q}_1\ast\dot{\Q}^2_2$ forces that $\dot{\S}$ is $\ka$-c.c. Let $\dot{T}$ be a $(\ps\ast\dot{\Q}_1\ast\dot{\S})$-name for an Aronszajn tree on $\ka$. Then $\ps\ast\dot{\Q}_1\ast(\dot{\Q}_2\times\dot{\S})$ forces that $\dot{T}$ is an Aronszajn tree.
\end{proposition}

That is to say, forcing with $\dot{\Q}_2$ after $\ps\ast\dot{\Q}_1\ast\dot{\S}$ does not add branches to $\dot{T}$. 
To show this, we will follow the standard approach and show that if $\dot{\Q}_2$ were to add such a branch, then we can find some model in which a level of the tree has too many nodes.

For the rest of this subsection, we suppose for a contradiction that $\dot{b}$ is $(\ps\ast\dot{\Q}_1\ast(\dot{\Q}_2\times\dot{\S}))$-name for a branch through $\dot{T}$, where $\dot{T}$ is a $(\ps\ast\dot{\Q}_1\ast\dot{\S})$-name for an Aronszajn tree on $\ka$. In the context of working with the forcing $\R^*:=\ps\ast\dot{\Q}_1\ast(\dot{\Q}^2_2\times\dot{\S})$, for which a typical generic looks like $G\ast Q_1\ast (Q_2^L\times Q_2^R\times F)$, we will use $\dot{b}_L$ to denote the $(\ps\ast\dot{\Q}_1\ast(\dot{\Q}_2\times\dot{\S}))$-name for $\dot{b}[\dot{G}\ast{\dot{Q}}_1\ast(\dot{Q}_2^L\times \dot{F})]$, i.e., the interpretation of $\dot{b}$ using the left generic filter added by $\dot{\Q}^2_2$. $\dot{b}_R$ is defined similarly.

The next lemma will be used as a successor step in obtaining a tree of conditions forcing incompatible information about a branch.

\begin{lemma}\label{lemma:densesplit} (Under the assumptions of Proposition \ref{prop:NNB}) $\ps$ forces that for each $\dot{\Q}_1$-name $\dot{d}$ for a condition in $\dot{\Q}_2$, there is a dense, open set of $c$ in $\dot{\Q}_1$ satisfying the following property: there exist names $\dot{d}_L,\dot{d}_R$ for conditions in $\dot{\Q}_2$ and an ordinal $\xi<\ka$ so that
\begin{enumerate}
\item $c\Vdash \dot{d}_Z\geq\dot{d}$ for each $Z\in\lb L,R\rb$;
\item $\la c,\dot{d}_L,\dot{d}_R,0_{\dot{\S}}\ra\Vdash^{V[\dot{G}]}_{\dot{\Q}_1\ast(\dot{\Q}_2\times\dot{\Q}_2\times\dot{\S})}\dot{b}_L(\xi)\neq\dot{b}_R(\xi)$.
\end{enumerate}
\end{lemma}
\begin{proof} We work in $V[G]$. Fix $c\in\Q_1$ and a $\Q_1$-name $\dot{d}$ for a condition in $\dot{\Q}_2$. Let $Q_1$ be $V[G]$-generic over $\Q_1$ containing $c$, and let $Q_2^L\times Q_2^R$ be $V[G\ast Q_1]$-generic over $\Q_2^2$ containing $(d,d)$.

We first claim that $0_{\S}$ forces over $V[G\ast Q_1\ast (Q_2^L\times Q_2^R)]$ that $\dot{b}_L\neq\dot{b}_R$. Thus let $F$ be an arbitrary $V[G\ast Q_1\ast (Q_2^L\times Q_2^R)]$-generic filter for $\S$. Since $\S$ and $\Q_2^2$ both live in $V[G\ast Q_1]$, the product lemma implies that $Q_2^L\times Q_2^R$ is $V[G\ast Q_1\ast F]$-generic over $\Q_2^2$. Since $Q_2^L$ and $Q_2^R$ are mutually $V[G\ast Q_1\ast F]$-generic filters over $\Q_2$, we conclude that
$$
V[G\ast Q_1\ast F]=V[G\ast Q_1\ast (F\times Q_2^L)]\cap V[G\ast Q_1\ast (F\times Q_2^R)].
$$
Therefore, if $b:=b_L=b_R$, then $b$ is in $V[G\ast Q_1\ast F]$, and therefore $T$ is not an Aronszajn tree in that model, a contradiction.

We now claim that there is an ordinal $\xi$ so that $0_\S$ forces over $V[G\ast Q_1\ast (Q_2^L\times Q_2^R)]$ that $\dot{b}_L(\xi)\neq\dot{b}_R(\xi)$. Let $A\seq\S$ be a maximal antichain in $V[G\ast Q_1\ast (Q_2^L\times Q_2^R)]$ consisting of conditions $g\in\S$ so that for some $\zeta_g<\ka$,
$$
g\Vdash^{V[G\ast Q_1\ast (Q_2^L\times Q_2^R)]}_{\S}\dot{b}_L(\zeta_g)\neq\dot{b}_R(\zeta_g).
$$
Because $\dot{b}_L$ and $\dot{b}_R$ name branches in $\dot{T}$, we see that for any $g\in A$ and $\zeta\geq\zeta_g$, $g$ forces that $\dot{b}_L(\zeta)\neq\dot{b}_R(\zeta)$. Since $\S$ is still $\ka$-c.c. after forcing to add $Q_2^L\times Q_2^R$, we know that $A$ has size $<\ka$ in $V[G\ast Q_1\ast (Q_2^L\times Q_2^R)]$. Therefore, letting $\xi:=\sup_{g\in A}\zeta_g$, $\xi<\ka$. Then $\xi$ witnesses the claim: indeed, if $f\in\S$ is any condition, we may extend it to $f^*$ so that $f^*$ is above some $g\in A$. By the remarks above and since $\xi\geq\zeta_g$, we know that $f^*\Vdash\dot{b}_L(\xi)\neq\dot{b}_R(\xi)$, completing the proof of the second claim.

Since $(c,\dot{d},\dot{d})\in Q_1\ast (Q_2^L\times Q_2^R)$, we may find an extension $(c^*,\dot{d}_L,\dot{d}_R)$ of $(c,\dot{d},\dot{d})$ as well as an ordinal $\xi<\ka$ so that $(c^*,\dot{d}_L,\dot{d}_R)$ forces that $0_{\dot{\S}}$ forces that $\dot{b}_L(\xi)\neq\dot{b}_R(\xi)$. Then $c^*\geq c$ is in the desired dense set.
\end{proof}

For the rest of the subsection, let $\vec{M}=\la M_\al:\al\in B\ra$ be a sequence which is suitable with respect to all parameters of interest. Since $\dot{\Q}_1\ast\dot{\Q}^2_2$ is $\cal{F}$-completely proper, which implies that $\dot{\Q}_1\ast\dot{\Q}_2$ is $\cal{F}$-completely proper, we may assume that $\dom(\vec{M})$ satisfies the conclusion of Definition \ref{def:fkcp} with respect to $\dot{\Q}_1\ast\dot{\Q}^2_2$ and with respect to $\dot{\Q}_1\ast\dot{\Q}_2$.

{Fix $M^*\prec H(\ka^{++})$, where $M^*$ has size $\ka$, is closed under $<\ka$-sequences, and contains $\vec{M}$ as well as $\lhd$ from Notation \ref{notation:wo} as an element.}
{Let $M$ denote the transitive collapse of $M^*$, so that $M$ is a $\ka$-model. Since $\dom(\vec{M})\in\cal{F}^+$, we may find an $M$-normal ultrafilter $U$ so that, letting $j:M\lra N$ be the associated ultrapower map, $\ka\in j(\dom(\vec{M}))$. As usual, we use $M_\ka$ to denote $j(\vec{M})(\ka)$.} The following claim shows that we can build the desired tree of conditions forcing incompatible information about the branch.

\begin{claim}\label{claim:split} $j(\ps)$ forces over $N$ that there exist sequences $\la \dot{c}_\nu:\nu<\om_1\ra$ and $\la\dot{d}_s:s\in 2^{<\om_1}\ra$ so that the following properties hold:
\begin{enumerate}
\item for each $f\in (2^{\om_1})^{N[\dot{G}_{j(\ps)}]}$, $\la (\dot{c}_\nu,\dot{d}_{f\res\nu}):\nu<\om_1\ra$ {is an increasing sequence of conditions in $\dot\Q_1\ast\dot{\Q}_2$ which is guided by the collapse at $\ka$ (see Definition \ref{def:genbycoll});}
\item if $s\neq t$ are in $2^\nu$ for some $\nu<\om_1$, then 
$$
j(\la \dot{c}_\nu,\dot{d}_s,\dot{d}_t,0_{\dot{\S}}\ra)\Vdash^{N[\dot{G}_{j(\ps)}]}_{j\left(\dot\Q_1\ast(\dot{\Q}_2\times\dot{\Q}_2\times\dot{\S})\right)} j(\dot{b})_L\res\ka\neq j(\dot{b})_R\res\ka.
$$
\end{enumerate}
\end{claim}
\begin{proof} {The definition is by recursion. Let $G^*$ be an arbitrary $V$-generic over $j(\ps)$, and let $G=G^*\cap\ps$. Let $(c_0,\dot{d}_0)$ be the trivial condition in $\Q_1\ast\dot{\Q}_2$. In order to show that the desired sequences generate filters which are guided by the collapse at $\ka$ (which in turn will guarantee that they have upper bounds), we will show that (1)-(3) of Lemma \ref{lemma:gettingthefilters} are satisfied. In particular, to secure (3) of that lemma, throughout the construction we will select objects which are minimal according to the fixed well-order $\lhd$ of $H(\ka^+)$. We remark that the entire construction takes place in $N[G^*]$, but the proper initial segments can be carried out in $M[G]$ using a proper initial segment of $f_\ka$, the standard surjection from $\ka$ onto $\om_1$ added by $G^*$.}

Suppose that $\nu$ is a limit and that for all $\mu<\nu$ and all $s\in 2^\mu$, we have defined $c_\mu$ and $\dot{d}_s$. {Then we let $\dot{c}_\nu$ be the $\lhd$-least $\ps$-name for a condition in $\dot{\Q}_1$ so that $c_\nu:=\dot{c}_\nu[G]$ is a sup of $\la c_\mu:\mu<\nu\ra$. Similarly, for $s\in 2^\nu$, we let $\dot{d}_s$ be a $\Q_1$-name forced to be a sup of $\la\dot{d}_{s\res\mu}:\mu<\nu\ra$ so that a $\ps$-name for $\dot{d}_s$ is $\lhd$-minimal.} Note that item (2) in the claim still holds since if $s\neq t$ are in $2^\nu$, then there exists some $\mu<\nu$ so that $s\res\mu\neq t\res\mu$. {So $j(\la c_\mu,\dot{d}_{s\res\mu},\dot{d}_{t\res\mu},0_{\dot{\S}}\ra)$ forces that $j(\dot{b})_L\res\ka\neq j(\dot{b})_R\res\ka$. Hence the extension $j(\la c_\nu,\dot{d}_s,\dot{d}_t,0_{\dot{\S}}\ra)$ also forces this.}


Now for the successor step. Suppose that we have defined $c_\nu$ and $\dot{d}_s$ for all $s\in 2^\nu$. In order to ensure that the assumptions of Lemma \ref{lemma:gettingthefilters} are satisfied, and thereby ensure that the sequences are guided by the collapse at $\ka$ (which in turn will guarantee they have an upper bound), we will first define an auxiliary extension $c^*_\nu\geq c_\nu$ and for each $s\in 2^\nu$, a $\Q_1$-name $\dot{d}^*_s$ forced by $c^*_\nu$ to extend $\dot{d}_s$. Towards this end, let $\ga:=f_\ka(\nu)$, where $f_\ka$ is the standard surjection added by $G^*$ from $\om_1$ onto $\ka$. {Let $u_\ga$ be the $\ga$-th condition in $\Q_1 \ast \dot{\Q}_2$}, and write $u_\ga$ as $\la c^\ga,\dot{d}^\ga\ra$. If $c^\ga$ does not extend $c_\nu$ in {$\Q_1$}, set $c^*_\nu=c_\nu$ and $\dot{d}^*_s=\dot{d}_s$. On the other hand, if $c^\ga\geq c_\nu$, we set $c^*_\nu=c^\ga$. Then, given $s\in 2^\nu$, if $c^*_\nu\Vdash\dot{d}^\ga\geq\dot{d}_s$, we set $\dot{d}^*_s=\dot{d}^\ga$, and otherwise we set $\dot{d}^*_s=\dot{d}_s$. Note that there is at most one $s$ that falls into the first of these, since $c^*_\nu\Vdash\lb\dot{d}_s:s\in 2^\nu\rb$ is an antichain in {$\dot{\Q}_2$}.

Now we move to defining $c_{\nu+1}$ and $\dot{d}_t$ for all $t\in 2^{\nu+1}$. By Lemma \ref{lemma:densesplit}, for each $s\in 2^\nu$, the set $D_s$ of all $c\in {\Q_1}$ for which there exist names $\dot{d}_L$ and $\dot{d}_R$ and an ordinal $\xi<\ka$ satisfying
\begin{enumerate}
\item[(i)] $c\Vdash\dot{d}_Z\geq\dot{d}_s$ for each $Z\in\lb L,R\rb$; and
\item[(ii)] $\la c,\dot{d}_L,\dot{d}_R,0_{\dot{\S}}\ra\Vdash {\dot{b}_L(\xi)\neq \dot{b}_R(\xi)}$
\end{enumerate}
{is dense, open in $\Q_{1}$. By Lemma \ref{lemma:cpimpliesdist} applied to $\Q_1$, there exists an extension of $c_\nu$ inside $\bigcap_{s\in 2^\nu}D_s$. Let $\dot{c}_{\nu+1}$ be the $\lhd$-minimal $\ps$-name so that $c_{\nu+1}:=\dot{c}_{\nu+1}[G]$ is such an extension. For each $s\in 2^\nu$, we may find an ordinal $\xi_s<\ka$ and $\Q_1$-names $\dot{d}_{s^\frown\la 0\ra}$ and $\dot{d}_{s^\frown\la 1\ra}$ so that $c_{\nu+1}\Vdash\dot{d}_{s^\frown\la i\ra}\geq\dot{d}_s$, so that $\la c_{\nu+1},\dot{d}_{s^\frown\la 0\ra},\dot{d}_{s^\frown\la 1\ra},0_{\dot{\S}}\ra$ forces that $\dot{b}_L(\xi_s)\neq \dot{b}_R(\xi_s)$, and so that $\ps$-names for $\dot{d}_{s^\frown\la 0\ra}$ and $\dot{d}_{s^\frown\la 1\ra}$ are $\lhd$-minimal. Applying $j$ to this statement, we secure (2) of the claim. This completes the proof.}
\end{proof}

Now that we have proven the above claim, we can finish the proof of Proposition \ref{prop:NNB}.

\begin{proof} (of Proposition \ref{prop:NNB}) Let $G^*$ be $N$-generic over $j(\ps)$, let $G:=G^*\cap\ps$, and fix sequences $\la c^{{0}}_\nu:\nu<\om_1\ra$ and $\la\dot{d}^{{0}}_s:s\in 2^{<\om_1}\ra$ satisfying Claim \ref{claim:split}. {Let $\la c_\nu:\nu<\om_1\ra$ and $\la\dot{d}_s:s\in 2^{<\om_1}\ra$ denote the sequences of their $j$-images.} For each $f\in (2^{\om_1})^{N[G^*]}$, $\la {(c^0_\nu,\dot{d}^0_{f\res\nu})}:\nu<\om_1\ra$ is guided by the collapse at $\ka$. {Moreover, $\dom(\vec{M})$ satisfies Definition \ref{def:fkcp}, and $\ka\in j(\dom(\vec{M}))$. Since $j(\Q_1\ast\dot{\Q}_2)$ is $j(\cal{F})$}-completely proper, we may find a condition $(c^*,\dot{d}_f)$ which is a sup in $j(\Q_1\ast\dot{\Q}_2)$ of $\la (c_\nu,\dot{d}_{f\res\nu}):\nu<\om_1\ra$ (note that $c^*$ is independent of $f$ since any two sups are $=^*$-equal). By {item (2) of the previous claim}, we know that if $f\neq g$ are in $(2^{\om_1})\cap N[G^*]$, then $\la c^*,\dot{d}_f,\dot{d}_g,0_{j(\dot{\S})}\ra$ forces that $j(\dot{b})_L\res\ka\neq j(\dot{b})_R\res\ka$.

Now let $Q_1^*\ast F^*$ be $N[G^*]$-generic over $j(\Q_1\ast\dot{\S})$ with $Q_1^*$ containing $c^*$. Applying {item (2) of the previous claim again}, we know that if $f\neq g$ are in $(2^{\om_1})\cap N[G^*]$, then $\la d_f,d_g\ra$ forces in $j(\Q^2_2)$ that $j(\dot{b})_L\res\ka\neq j(\dot{b})_R\res\ka$. We note here that the tree of interest, namely $T^*:=j(\dot{T})[G^*\ast Q_1^*\ast F^*]$, is a member of $N[G^*\ast Q_1^*\ast F^*]$, i.e., exists prior to forcing with $j(\Q_2^2)$.

For each $f\in (2^{\om_1})\cap N[G^*]$, let $d^*_f$ be an extension of $d_f$ in $j(\Q_2)$ which decides the value of $j(\dot{b})(\ka)$, say as $\al_f$. We claim that if $f\neq g$ are in $(2^{\om_1})\cap N[G^*]$, then $\al_f\neq\al_g$. Indeed, suppose for a contradiction that there were $f\neq g$ with $\al_f=\al_g$. Then force in $j(\Q^2_2)$ above the condition $\la d^*_f,d^*_g\ra$ to obtain a pair $\bar{Q}_2^L\times \bar{Q}_2^R$ of mutually generic filters for $j(\Q_2)$. Since $\al_f=\al_g$, the branch of $T^*$ below $\al_f$ is the same as the branch of $T^*$ below $\al_g$. But the branch of $T^*$ below $\al_f$ equals $(j(\dot{b})[\bar{Q}_2^L])\res\ka$ and the branch of $T^*$ below $\al_g$ equals $(j(\dot{b})[\bar{Q}_2^R])\res\ka$, contradicting the fact that $\la d^*_f,d^*_g\ra$ forces that the interpretations diverge below $\ka$.

Since $(2^{\om_1})\cap N[G^*]$ has size $j(\ka)$ in $N[G^*]$ and $j(\Q_2\ast\dot{\S})$ preserves $j(\ka)$, this set still has size $j(\ka)$ in $N[G^*\ast Q_1^*\ast F^*]$. Thus in the model $N[G^*\ast Q_1^*\ast F^*]$, the function taking $f\in (2^{\om_1})\cap N[G^*]$ to $\al_f$ is an injection. Therefore level $\ka$ of $T^*$ has size $j(\ka)$ which contradicts the fact that $j(\ka)$ is $\aleph_2$ in $N[G^*\ast Q_1^*\ast F^*]$ and that  $T^*$ is an Aronszajn tree on $j(\ka)$.
\end{proof}

\section{Putting it all Together}\label{section:final}

Up to this point in the paper, we have worked to establish a number of isolated results. In this section, we will now define the poset which will witness Theorem \ref{thm:main}. Each of the previous sections will function as a component in the inductive verification that this poset has the desired properties.

We recall that $\ps$ denotes $\col(\om_1,<\ka)$, the Levy collapse of the {ineffable} cardinal $\ka$. We define a $\ps$-name $\dot{\C}_{\ka^+}$ for a $\ka^+$-length iteration adding clubs, and we also define a $(\ps\ast\dot{\C}_{\ka^+})$-name $\dot{\S}_{\ka^+}$ for an iteration which specializes Aronszajn trees. This is done in such a way that for all $\rho<\ka^+$, the $(\ps\ast\dot{\C}_{\ka^+})$-name $\dot{\S}_\rho$ for the first $\rho$-stages of $\dot{\S}_{\ka^+}$ is actually a $(\ps\ast\dot{\C}_\rho)$-name.

More precisely, we define by recursion on $\rho\leq\ka^+$ the names $\dot{\C}_\rho$ and $\dot{\S}_\rho$. Suppose that $\rho=\rho_0+1$ is a successor and that $\dot{\C}_{\rho_0}$ and $\dot{\S}_{\rho_0}$ are both defined. Using the fixed well-order $\lhd$ from Notation \ref{notation:wo} as a bookkeeping device, we select a $(\ps\ast\dot{\C}_{\rho_0}\ast\dot{\S}_{\rho_0})$-name $\dot{S}_{\rho_0}$ for a stationary subset of $\ka\cap\cof(\om)$, and we define $\dot{\C}_\rho:=\dot{\C}_{\rho_0}\ast\mathsf{CU}(\dot{S}_{\rho_0},\dot{\S}_{\rho_0})$; see Definition \ref{def:clubaddingposet}. Next, we use $\lhd$ to select a $(\ps\ast\dot{\C}_{\rho}\ast\dot{\S}_{\rho_0})$-name $\dot{T}_{\rho_0}$ for an Aronszajn tree on $\ka$, and we set $\dot{\S}_\rho$ to be the $(\ps\ast\dot{\C}_\rho)$-name for $\dot{\S}_{\rho_0}\ast\S(\dot{T}_{\rho_0})$; see Definition \ref{def:specposet}.

Now suppose that $\rho$ is a limit and that we have defined the sequences $\la\dot{\C}_\xi,\dot{\C}(\xi):\xi<\rho\ra$ and $\la\dot{\S}_\xi,\dot{\S}(\xi):\xi<\rho\ra$. We first let $\dot{\C}_\rho$ be the $<\ka$-support limit of $\la\dot{\C}_\xi,\dot{\C}(\xi):\xi<\rho\ra$. Second, we see that $\ps\ast\dot{\C}_\rho$ forces that $\la\dot{\S}_\xi,\dot{\S}(\xi):\xi<\rho\ra$ names an iteration with countable support: by induction, if $\xi<\rho$ is a limit, then $\dot{\S}_\xi$ is the $(\ps\ast\dot{\C}_\xi)$-name for the countable support limit of $\la\dot{\S}_\zeta,\dot{\S}(\zeta):\zeta<\xi\ra$. But $\dot{\C}_\rho$ is $\om_1$-closed, and consequently, the countable support limit of $\la\dot{\S}_\zeta,\dot{\S}(\zeta):\zeta<\xi\ra$ is the same in both the extension by $\ps\ast\dot{\C}_\xi$ and the extension by $\ps\ast\dot{\C}_\rho$. In light of this, we let $\dot{\S}_\rho$ denote the countable support limit of $\la\dot{\S}_\xi,\dot{\S}(\xi):\xi<\rho\ra$, noting that this is an $\om_1$-closed poset in the extension by $\ps\ast\dot{\C}_\rho$.
This completes the definitions of the names. 
We may now define $\R^*:=\ps\ast\dot{\C}_{\ka^+}\ast\dot{\S}_{\ka^+}$.

We begin our analysis of $\R^*$ with some simple remarks. First, $\R^*$ is $\om_1$-closed, since all the posets under consideration are (and since our iterations were taken with supports which are at least countable). Additionally, $\R^*$ is $\ka^+$-c.c. Indeed, $\ps$ trivially is. Furthermore, $\ka^{<\ka}=\ka$ after forcing with $\ps$, and so if $\be<\ka^+$, $\dot{\C}_\be$ is forced to be a poset of size $\ka$. Since direct limits in the iteration $\dot{\C}_{\ka^+}$ are taken at all stages in $\ka^+\cap\cof(\ka)$, standard arguments (e.g., see \cite{BaumgartnerIF}) show that $\dot{\C}_{\ka^+}$ is $\ka^+$-c.c. Finally, since for every $\be<\ka^+$, $\dot{\S}_\be$ is forced to have size $\ka$ by $\ps\ast\dot{\C}_{\ka^+}$, and since $\dot{\S}_{\ka^+}$ is taken with countable supports, a standard $\Delta$-System argument shows that $\dot{\S}_{\ka^+}$ is $\ka^+$-c.c.

We next claim that if $\R^*$ preserves $\ka$, then it forces all Aronszajn trees on $\ka$ are special, that such trees exist, and that every stationary subset $S\subseteq \kappa \cap \cof(\omega)$ reflects on every ordinal of cofinality $\omega_1$ in some closed unbounded subset on $\kappa$. 
First, suppose that $\dot{T}$ is an $\R^*$-name for an Aronszajn tree on $\ka$. Because $\R^*$ is $\ka^+$-c.c., $\dot{T}$ is an $(\RR{\ga}{\ga})$-name for some
$\ga<\ka^+$, and hence names an Aronszajn tree in any extension between that given by $\RR{\ga}{\ga}$ and the full $\R^*$-extension. By our bookkeeping device, there is some $\de\geq\ga$ so that $\dot{\S}(\de)$ is forced by $\RR{\de+1}{\de}$ to equal $\dot{\S}(\dot{T})$. Hence $\R^*$ forces that $\dot{T}$ is special. Similarly, if $\dot{S}$ is an $\R^*$-name for a stationary subset of $\ka\cap\cof(\om)$, then there is some $\al<\ka^+$ so that $\dot{S}$ is a $(\RR{\al}{\al})$-name, and hence there is some $\be\geq\al$ so that $\dot{\C}(\be)$ is forced by $\ps*\dot\C_{\be}$ to equal $\mathsf{CU}(\dot{S},\dot{\S}_\be)$. Thus in the extension by $\RR{\be+1}{\be}$, $\dot{S}$ reflects almost everywhere, and since the forcing to complete $\RR{\be+1}{\be}$ to $\R^*$ is $\om_1$-closed, $\dot{S}$ still reflects almost everywhere in the full $\R^*$-extension.

As a result of the previous discussion, we see that in order to show that $\R^*$ witnesses Theorem \ref{thm:main}, we need to prove that $\R^*$ preserves $\ka$. To achieve this we verify that (i) $\ps$ forces $\dot\C_{\ka^+}$ is $\ka$-distributive, and that (ii) $\ps* \dot\C_{\ka^+}$ forces that $\dot\S_{\ka^+}$ is $\ka$-c.c.

To this end, we consider the following simplifications. 
First, concerning (i), we note that since $\dot\C_{\ka^+}$ is forced to be $\ka^+$-c.c, it is sufficient to verify that $\dot\C_\rho$ is forced to be $\ka$-distributive for all $\rho < \ka^+$ to show that (i) holds. 
We will use Lemma \ref{lemma:cpimpliesdist} to verify this, by proving that 
for every $\rho<\ka^+$, $\dot{\C}_\rho$ is {$\cal{F}$}-completely proper.

Second, concerning (ii), since $\dot\S_{\ka^+}$ names a countable support iteration, any $\ka$-sized antichain would witness that some proper initial segment $\dot\S_{\ga}$ is not $\ka$-c.c.
Therefore, it is sufficient to verify that $\ps*\dot\C_{\ka^+}$ forces $\dot\S_\ga$ is $\ka$-c.c for every $\ga < \ka^+$. Howover, 
as $\dot\C_{\ka^+}$ names a $\kappa^+$-{cc} poset, every $(\ps*\dot\C_{\ka^+})$-name for a subset of $\dot\S_{\ga}$ of size $\ka$ is equivalent to a $(\ps*\dot\C_{\rho})$-name, for some $\rho \geq \ga$. 
Clearly, if $\ps*\dot\C_{\rho}$ forces $\dot\S_{\ga}$ fails to satisfy the $\ka$-c.c for some $\ga \leq \rho$, then it forces $\dot\S_\rho$ is not $\ka$-c.c.
We conclude that (ii) follows from the assertion that  for every $\rho < \ka^+$,  $\ps*\dot\C_\rho$ forces that $\dot\S_\rho$ is $\ka$-c.c.

Combining the two simplifications, it remains to prove the next claim.
\begin{claim}\label{claim:InMainTHM}
The following holds for every $\rho<\ka^+$:
\begin{enumerate}
    \item $\dot{\C}_\rho$ is {$\cal{F}$}-completely proper, and
    \item $\ps*\dot\C_\rho$ forces $\dot\S_\rho$ is $\ka$-c.c.
\end{enumerate}
\end{claim}

We prove the claim by induction on $\rho < \ka^+$. 
Let $\rho < \ka^+$, and suppose that the claim holds for every $\sigma < \rho$. In particular, if $\si<\rho$, then since $\ps\ast\dot{\C}_\si$ forces that $\dot{\S}_\si$ is a countable support iteration specializing trees and since (2) holds at $\si$, we must have that $\ps\ast\dot{\C}_\si$ forces that $\dot{\S}_\si$ is a countable support iteration specializing \emph{Aronszajn} trees. Hence we see that assumptions (1)-(4) 
from the beginning of section \ref{section:clubscompletelyproper} hold. Applying Proposition \ref{prop:clubsareCP}, 
it follows that $\dot\C_\rho$ is {$\cal{F}$}-completely proper, and moreover, so is $\C_{\rho}(\vec{\delta})$ for every finite, decreasing sequence
$\vec{\delta}$ of ordinals below $\rho$. 
We use this to prove that (2) of the claim holds at $\rho$.

We aim to apply Corollary \ref{cor:IH1and2} to the {$\cal{F}$}-strongly proper poset $\ps^*:=\ps\ast\dot{\C}_\rho$, and to do so, we need to verify that for each $\si<\rho$, $\ps^*\ast\dot{\S}_\si\Vdash\dot{T}_\si$ is an Aronszjan tree. This is where the doubling tail products come into play. Indeed,
we consider a slightly more general statement, which would allow us to use Proposition \ref{prop:NNB}.
    For each decreasing sequence $\vec{\delta} = \la \delta_{0},\dots,\delta_{n-1}\ra$ in $\rho$, we let $\star_\rho(\vec{\delta})$  be the statement that $\po*\dot\C_\rho(\vec{\delta})$ 
    forces $\dot\S_{\delta_{n-1}}$ is $\ka$-c.c. (note that if $\vec{\delta} = \emptyset$ then the statement holds vacuously).
    We prove by induction on the reverse lexicographic order $<_{\rlex}$
    on $[\rho]^{<\omega}_\text{dec}$ 
    that $\star_\rho(\vec{\delta})$ holds.  
    The base case of the induction, where $\vec{\delta} = \emptyset$, trivially holds, as mentioned above. For the induction step, fix $\vec{\delta} = \la \delta_{{0}},\dots,\delta_{{n-1}}\ra \in [\rho]^{<\omega}_\text{dec}$ and suppose that $\star_\rho(\vec{\gamma})$ holds for every $\vec{\gamma} <_{\rlex} \vec{\delta}$. 
    To show that $\star_\rho(\vec{\delta})$ holds, we need to verify that $\ps * \dot\C_\rho(\vec{\delta})$ forces that $\dot\S_{\delta_{n-1}}$ is $\ka$-c.c.
    For this in turn, by Corollary \ref{cor:IH1and2}, it is sufficient to verify that for every $\gamma < \delta_{n-1}$, $\ps * \dot\C_\rho(\vec{\delta})\ast\dot{\S}_\ga$ forces that $\dot{T}_\gamma$ is an Aronszajn tree on $\kappa$. 
    If $\delta_{n-1} = 0$ there is nothing to prove. Otherwise, let $\gamma < \delta_{n-1}$, and 
    set $\vec{\gamma} :=  \vec{\delta}^\frown \la \gamma\ra$.
    By Proposition \ref{prop:NNB} and the definition of $\dot\C_\rho(\vec{\gamma})$, 
    it suffices to verify that $\ps * \dot\C_\rho( \vec{\gamma})$ forces that $\dot\S_\gamma$ is $\ka$-c.c, to conclude that 
    $\ps * \dot\C_\rho(\vec{\delta})\ast\dot{\S}_\ga$ forces that $\dot{T}_\gamma$ is Aronszajn.
    However, the last is just $\star_\rho(\vec{\gamma})$, which holds by our inductive assumption and the fact that 
     $\vec{\gamma} <_{\rlex} \vec{\delta}$.
This concludes the proof of Claim \ref{claim:InMainTHM}, which in turn finishes the proof of Theorem \ref{thm:main}.\\

{We conclude the paper with two questions:}

\begin{question} {Is an ineffable cardinal necessary for proving Theorem \ref{thm:main}?  Is a weakly compact cardinal sufficient?}
\end{question}

{As we've remarked throughout the paper, we only use  ineffability in the proof of Proposition \ref{prop:InductiveIandII} (and thus in the corollaries of this proposition). If one could prove this proposition assuming only a weakly compact, then that would suffice to show that a weakly compact is optimal.}\\

{We also mention the following long-standing question:}

\begin{question}{ Is a weakly compact cardinal needed for $\mathsf{SATP}(\om_2)+2^{\om_1}=\om_3$?}
\end{question}

{We recall that in the Laver-Shelah model of $\mathsf{SATP}(\om_2)$, $2^{\om_1}=\om_3$. Moreover, Rinot has shown (\cite{RinotHigherSouslin}) that if the $\mathsf{GCH}$ and $\mathsf{SATP}(\om_2)$ both hold, then $\om_2$ is weakly compact in $L$. }

\section{Acknowledgements}

The authors would like to thank Itay Neeman, for helpful comments and corrections, and also the referee, for a thorough reading of the manuscript and many valuable suggestions and comments.


\begin{thebibliography}{0}


\bibitem{Abraham}
Uri Abraham. On forcing without the continuum hypothesis. {\it J. Symb. Log.}. {\bf 48} (1983) no. 3, 658-661. 


\bibitem{AbrahamShelah}
Uri Abraham and Saharon Shelah. Isomorphism types of Aronszajn trees. {\it Israel J. Math.}. {\bf 50} (1985) no. 1-2, 75-113.

\bibitem{AG}
David Asper\'{o} and Mohammad Golshani. The special Aronszajn tree property at $\aleph_2$ and GCH. {\it Submitted}.



\bibitem{BaumgartnerDense}
James E. Baumgartner. All $\aleph_1$-dense sets of reals can be isomorphic. {\it Fund. Math.}. {\bf 79} (1973), 101-106.


\bibitem{Baumgartner}
James E. Baumgartner. A new class of order types. {\it Ann. Math. Logic}. {\bf 9} (1976) no. 3, 187-222.


\bibitem{BaumgartnerIneffability}
James E. Baumgartner. Ineffability properties of cardinas I. In \emph{Infinite and finite sets (Colloq., Keszthely, 1973: dedicated to P. Erd{\"o}s on his 60th birthday)}, volume 1 of \emph{Colloq. Math. Soc. J{\'a}nos Bolyai, Vol. 10}, pages 109-130. North-Holland, Amsterdam, 1975.



\bibitem{BaumgartnerIF}
James E. Baumgartner. Iterated Forcing. In Adrian R.D. Mathias, editor, Surveys in Set Theory, volume 87 of London Mathematical Society Lecture Note Series, pages 1-59. Cambridge University Press, Cambridge, 1983.



\bibitem{BMR}
James E. Baumgartner, Jerome Malitz, and William Reinhard. Embedding trees in the rationals. {\it Proceedings of the National Academy of Sciences}. {\bf 67} (1970) no. 4, 1748-1753.


\bibitem{BNdiamond}
Omer Ben-Neria. Diamonds, Compactness, and Measure Sequences. {\it J. Math. Log.}, {\bf 19} (2019), 1950002.




\bibitem{CoxKrueger}
Sean Cox and John Krueger. Quotients of Strongly Proper Forcings and Guessing Models. {\it J. Symb. Log.}. {\bf 81} (2016) no. 1, 264-283.



\bibitem{CFM}
James Cummings, Matthew Foreman and Menachem Magidor.
Squares, Scales and Stationary Reflection. {\it J. Math. Log.}. {\bf 01} (2001) 35-98.

\bibitem{CummingsHB}
James Cummings, Iterated Forcing and Elementary Embeddings. 
{\it. Handbook of Set Theory, Foreman, Matthew
and Kanamori, Akihiro edt.} Springer Netherlands (2010) 
775-883.



\bibitem{CummingsSchimmerling}
James Cummings and Ernest Schimmerling. Indexed squares. {\it Israel J. Math.}. {\bf 131} (2002) 61-99.


\bibitem{CummingsWylie}
James Cummings and Dorshka Wylie. More on full reflection below $\aleph_\om$. {\it Arch. Math. Logic}. {\bf 49} (2010) 659-671.
 


\bibitem{HayutFontanella}
Laura Fontanella and Yair Hayut.
Square and Delta reflection.
 {\it Ann. Pure Appl. Logic}. {\bf 167(08)} (2016) 663-683.
 
 
\bibitem{ForemanTodor}
Matthew Foreman and Stevo Todor{\v c}evi{\' c}. A New L{\"o}wenheim-Skolem Theorem. {\it Transactions of the American Mathematical Society},
{\bf 357} (2005), 1693--1715.


\bibitem{GiltonLevineS}
Thomas Gilton, Maxwell Levine, and {\v S}{\'a}rka Stejskalov{\'a}. Trees and Stationary Reflection at Double Successors of Regular Cardinals. Accepted to {\it J. Symb. Log.}.


\bibitem{GiltonNeeman}
Thomas Gilton and Itay Neeman. Side conditions and iteration theorems. (To appear in Appalachian Set Theory.)

\bibitem{GoldbergDiamond}
Gabriel Goldberg. 
Structure theorems above a strongly compact cardinal. {\it preprint}.




\bibitem{GH}
Mohammad Golshani and Yair Hayut. The Special Aronszajn Tree Property. {\it J. Math. Log..} {\bf 20} (2020) no. 1, 26 pp.


\bibitem{HS}
Leo Harrington, and Saharon Shelah. Some Exact Equiconsistency Results in Set Theory. {\it Notre Dame Journal of Formal Logic.} {\bf 26} (1985) no. 2, 178-188.

\bibitem{Hauser}
Kai Hauser. Indescribable cardinals and elementary embeddings. {\it J. Symb. Log.} {\bf 56} (1991), no. 2, 439-457.




\bibitem{HayutLambie}
Yair Hayut and Chris Lambie-Hanson.
Simultaneous stationary reflection and square sequences.
 {\it J. Math. Log.}. {\bf 17(02)} (2017) 1750010.
 
 \bibitem{HolyLuecke:SmallLargeInduced}
 Peter Holy and Philip L{\"u}cke. Small Models, Large Cardinals, and Induced Ideals. {\it Ann. Pure Appl. Logic}. {\bf 172} (2021), no. 2
 

\bibitem{JechShelah}
Thomas Jech and Saharon Shelah. Full reflection of stationary sets below $\aleph_\om$. {\it J. Symb. Log.}. {\bf 55} (1990), 822-830.


\bibitem{Jensen}
Ronald Bjorn Jensen. The fine structure of the constructible hierarchy. {\it Ann. Math. Logic}. {\bf 4}, (1972), 229-308.

\bibitem{Jensen:BelowZeroPistol}
Ronald Bjorn Jensen. Some remarks on $\Box$ below zero-pistol. Circulated notes.

 
 
\bibitem{Konig}
D{\'e}nes. K$\ddot{\text{o}}$nig. {\"U}ber eine Schlussweise aus dem Endlichen ins Unendliche. {\it Acta Sci. Math.}. {\bf 3} (1927), no. 2-3, 121-130.




 
 \bibitem{KruegerWC} John Krueger. Weak Compactness and No Partial Squares. {\it J. Symb. Log.}. {\bf 76} (2011), no. 13, 1035-1060.
 
 
 
\bibitem{KruegerWeakSquare}
John Krueger. Weak square sequences and Special Aronszajn trees. {\it Fund. Math.}. {\bf 221} (2013) no. 3, 267-284.


\bibitem{Krueger}
John Krueger. Club Isomorphisms on Higher Aronszajn Trees. {\it Ann. Pure Appl. Logic}. {\bf 169} (2018) no. 10, 1044-1081.





 

\bibitem{Kurepa}
Djuro Kurepa. Ensembles ordonn$\acute{\text{e}}$s et ramifi$\acute{\text{e}}$s. {\it Publ. Math. Univ. Belgrade}. {\bf 4}, (1935), 1-138.



\bibitem{LS} 
Richard Laver and Saharon Shelah. The $\aleph_2$-Souslin Hypothesis. {\it Trans. Amer. Math. Soc.}. {\bf 264} (1981) no. 2, 411-417.




\bibitem{Magidor}
Menachem Magidor. Reflecting stationary sets. {\it The J. Symb. Log.}, {\bf 47}  no. 4, 755-771 (1983).

\bibitem{MagidorShelah}
Menachem Magidor and Saharon Shelah. The tree property at successors of singular cardinals. {\it  Arch. Math. Logic}, {\bf 35} (5-6), 385-404 (1996).

\bibitem{Mitchell}
William Mitchell. Aronszajn trees and the independence of the transfer property. {\it Ann. Pure Appl. Logic}, {\bf 5} (1972/73), 21-46.

\bibitem{MitchellUgh}
William Mitchell. $I[\omega_2]$ can be the nonstationary ideal on Cof$(\omega_1)$. {\it Trans. Amer. Math. Soc.}, {\bf 361} (2009), no. 2, 561-601 .

\bibitem{RinotHigherSouslin}
Assaf Rinot.
Higher Souslin trees and the GCH, revisited.
{\it Adv. Math.}, {\bf 311(c)} (2017) 510-531.



\bibitem{RinotBSL}
Assaf Rinot.
Chain conditions of products, and weakly compact cardinals.
{\it Bull. Symbolic Logic}.  {\bf 20(3)} (2014), 293-314.



\bibitem{ShelahReflcSCH}
Saharon Shelah.
Reflection implies the SCH.
{\it Fund. Math.}, {\bf 198} (2008), 95-111.


\bibitem{SchimmerlingSquares}
Ernest Schimmerling. Combinatorial principles in the core model for one Woodin cardinal, {\it Ann. Pure Appl. Logic} {\bf74 (2)} (1995), 153-201.




\bibitem{Specker}
Ernst Specker. Sur un probl$\grave{\text{e}}$me de \text{S}ikorski. {\it Colloq. Math.}. {\bf 2} (1949), 9-12.



\bibitem{TodorcevicSquares}
Stevo Todor{\v c}evi{\' c}. Coherent sequences. {\it Handbook of set theory} Vol. 1, Springer, Dordrecht (2010) 215-296.


\bibitem{Unger}
Spencer Unger. Fragility and indestructibility of the tree property. {\it Arch. Math. Logic}, {\bf 51} (2012), 635-645 .




\end{thebibliography}
\end{document}